\documentclass[11pt, twoside]{article}
\usepackage{amsfonts,amsmath,amssymb,amsthm}
\usepackage[mathscr]{euscript}
\usepackage{indentfirst}
\usepackage{enumerate}
\usepackage{color}
\usepackage{lipsum}
\usepackage{paralist}
\usepackage{graphics}
\usepackage{graphicx}
\usepackage{subfigure}
\usepackage{tikz}
\usepackage{epstopdf}
\usepackage{float}\usepackage{mathrsfs,caption}
\usepackage{color}
\usepackage{accents}
\usepackage{enumerate}

\graphicspath{{Figures/}}
 \usepackage{paralist}

\newtheorem{theo}{Theorem}[section]

\newtheorem{lem}{Lemma}[section]
\newtheorem{remark}{Remark}[section]
\newtheorem{col}{Corollary}[section]

\numberwithin{equation}{section}
\newtheorem{thmx}{Theorem}

\renewcommand {\baselinestretch}{1.2}
\newcommand{\lbl}[1]{\label{#1}}
\allowdisplaybreaks[4]

\newcommand{\be}{\begin{equation}}
\newcommand{\ee}{\end{equation}}
\newcommand\bes{\begin{eqnarray}} \newcommand\ees{\end{eqnarray}}
\newcommand{\bess}{\begin{eqnarray*}}
\newcommand{\eess}{\end{eqnarray*}}
\newcommand{\bbbb}{\left\{\begin{aligned}}
\newcommand{\nnnn}{\end{aligned}\right.}
\newcommand{\bea}{\begin{align*}}
\newcommand{\eea}{\end{align*}}

\newcommand\ep{\varepsilon}
\newcommand\kk{\left}
\newcommand\rr{\right}
\newcommand\dd{\displaystyle}
\newcommand\vp{\varphi}

\newcommand\dx{{\rm d}x}
\newcommand\dy{{\rm d}y}

\newcommand\lm{\lambda}
\newcommand\nm{\nonumber}
\newcommand\yy{\infty}
\newcommand\qq{\eqref}

\newcommand\ol{\overline}
\newcommand\ud{\underline}

\newcommand\pp{\partial}

\newcommand\oo{\Omega}
\newcommand\boo{\overline\Omega}
\newcommand\bd{\boldmath}
\newcommand\sss{{\cal S}}
\newcommand\vvv{\vspace{1.5mm}}

\newcommand\rrr{\color{red}}

\def\theequation{\arabic{section}.\arabic{equation}}

\begin{document}\thispagestyle{empty}
\setlength{\abovedisplayskip}{7pt}
\setlength{\belowdisplayskip}{7pt}

\begin{center}{\Large\bf Approximation of the generalized principal eigenvalue of }\\[2mm]
 {\Large\bf cooperative nonlocal dispersal systems and applications}\\[2mm]
Mingxin Wang\footnote{{\sl E-mail}: mxwang@sxu.edu.cn. Mingxin Wang was supported by National Natural Science Foundation of China Grant 12171120.}\\
{\small School of Mathematics and Statistics, Shanxi University, Taiyuan 030006, China}\\
Lei Zhang\footnote{The corresponding author, {\sl E-mail}: zhanglei890512@gmail.com. Lei Zhang was supported by the National Natural Science Foundation of China (12471168, 12171119) and the Fundamental Research Funds for the Central Universities (GK202304029, GK202306003, GK202402004). }\\
{\small School of Mathematics and Statistics, Shaanxi Normal University, Xi'an 710119, China}
\end{center}

\begin{quote}
\noindent{\bf Abstract.} It is well known that, in the study of the dynamical properties of nonlinear evolution system with nonlocal dispersals, the principal eigenvalue of linearized system play an important role. However, due to lack of compactness, in order to obtain the existence of principal eigenvalue, certain additional conditions must be attached to the coefficients. In this paper, we approximate the generalized principal eigenvalue of nonlocal dispersal cooperative and irreducible system, which admits the Collatz-Wielandt characterization, by constructing the monotonic upper and lower control systems with principal eigenvalues; and show that the generalized principal eigenvalue plays the same role as the usual principal eigenvalue.


\end{quote}

\pagestyle{myheadings}
\section{Introduction}{\setlength\arraycolsep{2pt}
\markboth{\rm$~$ \hfill Approximation of the generalized principal eigenvalue and applications \hfill $~$}{\rm$~$ \hfill M.X. Wang \& L. Zhang\hfill $~$}

The reaction diffusion system effectively describes the proliferation and diffusion of microorganisms, with the eigenvalue in such a system often playing a crucial role in the dynamics analysis. In recent years, nonlocal dispersal, described by integral operators, has been increasingly used to model long-distance diffusion instead of local diffusion. However, the solution mapping of a nonlocal dispersal system loses compactness, which presents a significant challenge in the dynamics analysis.

Let $n\ge 1$ be an integer. We define
 \[\sss=\{1,\cdots,n\},\;\;\;\mathbb{R}_+^n=\big\{u\in\mathbb{R}^n:\,u_i\ge 0,\,\,\forall\, i\in\sss\big\}.\]
Given functions $J_1(x,y),\cdots, J_n(x,y)$. Assume that they satisfy the following condition:
 \begin{itemize}\vspace{-2mm}
	\item[{\bf(J)}] $J_i(x,y)\ge 0$ is a continuous function of $(x,y)\in\mathbb{R}^{2N}$, and
 \[J_i(x,x)>0,\;\;\int_{\mathbb{R}^N}J_i(x,y)\dy=1,\;\;\forall\;
 x\in\mathbb{R}^N,\;i\in\sss.\]
\end{itemize}\vspace{-2mm}
Given positive constants $d_1,\cdots,d_n$. We define $d^*_1(x),\cdots,d^*_n(x)$ by the following manner:
\begin{itemize}
	\item[{\bf(D)}] For each $i\in\sss$, either $d^*_i(x)=d_i$ (corresponding to Dirichlet boundary condition), or $d^*_i(x)=d_ij_i(x)$ (corresponding to Neumann boundary condition), where
 \[j_i(x)=\int_{\Omega} J_i(y,x){\rm d}y,\;\;\,x\in\ol\oo.\]
  \end{itemize}\vspace{-2mm}
Let $\oo\subset\mathbb{R}^N$ be a bounded and smooth domain. We consider the initial value problem of cooperative system with nonlocal dispersals
 \bes\left\{\begin{array}{lll}
 u_{it}=d_i\dd\int_\oo J_i(x,y)u_i(y,t)\dy-d^*_i(x)u_i+f_i(x,u),\; &x\in\ol\oo,\;&t>0,\\[3mm]
 u_i(x,0)=u_{i0}(x)\ge0,\,\not\equiv 0, &x\in\boo,\\[1mm]
 i=1,\cdots,n.
 \end{array}\rr.\lbl{1.1}\ees
Its corresponding linearized eigenvalue problem at zero solution is
  \bes\begin{cases}
 d_i\dd\int_\oo J_i(x,y)\phi_i(y)\dy-d^*_i(x)\phi_i+\sum_{k=1}^n \pp_{u_k}f_i(x,0)\phi_k(x)=\lm \phi_i,\; x\in\boo,\\
  i=1,\cdots,n.
  \end{cases}\lbl{1.2}\ees
It is well known that if $\kk(\partial_{u_k}f_i(\tilde x,0)\rr)_{n \times n}$ is irreducible for some $\tilde x\in\oo$, then the corresponding eigenvalue problem with local diffusions must have the principle eigenvalue. However, for the nonlocal dispersal eigenvalue problem \qq{1.2}, to ensure the existence of principle eigenvalue, certain additional conditions are always required due to the lack of compactness. To our knowledge, these additional conditions are generally quite stringent, even for the scalar case.

Several approaches have been explored to establish sufficient conditions for the existence of the principal eigenvalue when the nonlocal dispersal system is strongly order-preserving. The first method involves transforming the problem into perturbations of the generators of positive semigroups, as proposed by B\"{u}rger (\cite{B1988MZ}). Shen and her collaborators applied this method to establish the principal eigenvalue theory for scalar autonomous and periodic equations and systems (see \cite{SZ10JDE, SX15, BS17}). The second method is based on the generalized Krein-Rutmann theorem (\cite{EPS72,N1981FPT}), which shows that the principal eigenvalue of a bounded positive operator exists if there is a gap between its spectral radius and its essential spectral radius. Coville \cite{Cov} utilized this method to present several sufficient conditions for the existence of the principal eigenvalue for a scalar equation. Liang, Zhang and Zhao \cite{LZZ2019JDE} also employed the generalized Krein-Rutmann theorem to investigate the principal eigenvalue problem for a time-periodic cooperative nonlocal dispersal system with time delay.  The third method is the fundamental analysis based on the Collatz-Wielandt characterization, introduced by Li, Coville and Wang \cite{LCWdcds17}. By applying this characterization, Su et al. \cite{SWZ23, SLLW2023JMPA} examined the principal eigenvalue for cooperative systems with nonlocal and coupled dispersal. Recently, there have been some investigations employing perturbations of resolvent positive operators (\cite{Thieme1998DCDS}) to study the principal eigenvalue of nonlocal dispersal equations with age structure (\cite{KangR22}), in a time periodic environment (\cite{FLRX24}) and with time delay (\cite{LZZ2019JDE}). Since the generators of a positive semigroup are resolvent positive operators, this method can be regarded as a generalization of B\"{u}rger's approach.

However, the principal eigenvalue may not exist for nonlocal dispersal equations (see \cite{SZ10JDE}). In such cases, the generalized principal eigenvalue serves as a suitable alternative, fulfilling a similar role to that of the principal eigenvalue.  Berestycki, Coville, and Vo \cite{B-JFA16} define the generalized principal eigenvalue for a scalar nonlocal dispersal operator using the Collatz-Wielandt characterization.

In order to demonstrate our main results, we provide a short review about the scalar equation. The following initial value problem of scalar nonlocal dispersal equation with logistic-type growth
 \bes\begin{cases}
 u_t=d\dd\int_{\Omega}J(x,y) u(y,t){\rm d}y-d^*(x)u
 +u(a(x)-u),\;\;&x\in\ol\oo,\;t>0\\[2mm]
	u(x,0)=u_0(x)>0,&x\in\ol\oo
\end{cases}\lbl{1.3}\ees
has been systematically studied, where $a\in C(\boo)$ and $a>0$ in $\boo$.

Based on the work of Berestycki,  Coville and Vo (\cite{B-JFA16}), we can define
 \bess
 \bar\lambda&=&\dd\inf_{\phi\in C(\boo),\,\phi\gg 0}\,\sup_{x\in\boo}\frac{d\int_\oo J(x,y)\phi(y)\dy
 -d^*(x)\phi(x)+a(x)\phi(x)}{\phi(x)},\\[2mm]
\underline{\lambda}&=&\sup_{\phi\in C(\boo),\,\phi\gg 0}\,\inf_{x\in\boo}\frac{d\int_\oo J(x,y)\phi(y)\dy
 -d^*(x)\phi(x)+a(x)\phi(x)}{\phi(x)}
 \eess
where $\phi \gg 0$ represents $\phi(x)>0$ for all $x\in\boo$.
It is easy to verify that $\bar\lambda\geq \underline{\lambda}$.
Thanks to \cite[Theorem 2.2]{LCWdcds17}, one can prove that $\bar\lambda=\underline{\lambda}$, which is called the generalized principal eigenvalue. Such formulae are usually referred to as the Collatz-Wielandt characterization. Define an  operator $\mathcal{P}$  by
\[\mathcal{P}[\phi]= d\dd\int_{\Omega}J(x,y) \phi(y){\rm d}y-d^*(x)\phi+a(x)\phi, \;\; \phi \in C(\ol\oo).\]
The generalized principal eigenvalue of ${\cal P}$, which is denoted by $\lm({\cal P})$, coincides with the spectral bound, i.e., $\lm({\cal P})=s({\cal P})$  (\cite{LCWdcds17}).

The positive equilibrium solution of \eqref{1.3}, i.e., the positive solution of
 \bes
 d\dd\int_{\Omega}J(x,y)U(y){\rm d}y-d^*(x)U+U(a(x)-U)=0,\;\;x\in\ol\oo,
  \lbl{1.4}\ees
and the large-time behavior of the solution to \eqref{1.3} have been investigated in \cite[Theorem C]{SZ2012PAMS}, \cite[Theorems 1.6 and 1.7]{Cov}, and \cite[Theorem 3.10]{YLR2019JDE}. The dynamics of the non-critical case ($\lm(\mathcal{P})\neq 0$) are well-understood: if $\lm(\mathcal{P})>0$, then \eqref{1.4} has a unique positive solution, which is globally asymptotically stable. Conversely, if $\lm(\mathcal{P})<0$, then \eqref{1.4} has no positive solution, and the zero solution is globally asymptotically stable. However, for the critical case ($\lm(\mathcal{P})=0$), the dynamics are not yet fully resolved. Moreover, there is a natural question: Is it possible to give a general conclusion for a cooperative system?

The main aim of this paper is to give a positive answer about the above question. We approximate the generalized principal eigenvalue, which has the Collatz-Wielandt characterization, by perturbing the matrix $B(x)$ with identity matrix $I$ and constructing the monotonic upper and lower control systems with principal eigenvalues; and show that the generalized principal eigenvalue plays the same role as the usual principal eigenvalue.

For the convenience of description, we will introduce some standing notations. For any given $u,v\in\mathbb{R}^n$. We say $u\geq v$ refers to $u_i\geq v_i$ for all $i\in\sss$; $u >v$ refers to $u_i \geq v_i$ for all $i\in\sss$ but $u\neq v$; and $u \gg v$ refers to $u_i>v_i$ for all $i\in\sss$.
Let $M=(m_{ik})_{n\times n}$ be an matrix with constant coefficients. We recall that $M$ is cooperative (essentially positive) if $m_{ik}\ge 0$ for all $i\not=k$, and that $M$ is irreducible (fully coupled) if the index set $\sss$ cannot be split up in two disjoint nonempty sets ${\cal I}$ and ${\cal K}$ such that $m_{ik}=0$ for $i\in{\cal I}, k\in{\cal K}$.

We first provide the approximation and characterization of generalized principal eigenvalue. Set
  \[a_{ik}(x)=\pp_{u_k}f_i(x,0);\;\;\;b_{ik}(x)=a_{ik}(x),\;\;i\not=k;
   \;\;\;b_{ii}(x)=a_{ii}(x)-d_i^*(x),\]
and
 \bess
  B(x)=(b_{ik}(x))_{n\times n}.
   \eess
We define an operator ${\mathscr B}$ by
 \bes\begin{cases}
	{\mathscr B}[\phi]=({\mathscr B}_1[\phi],\cdots,{\mathscr B}_n[\phi]),\\[1mm]
	{\mathscr B}_i[\phi]=d_i\dd\int_\oo J_i(x,y)\phi_i(y)\dy+\sum_{k=1}^n b_{ik}(x)\phi_k(x),
\end{cases}\lbl{1.5}\ees
and assume that\vspace{-2mm}
\begin{itemize}
	\item[{\bf(B)}] $b_{ik}\in C(\boo)$ for all $i,k\in\sss$ and $B(x)=(b_{ik}(x))_{n \times n}$ is a cooperative matrix, and there exists $\tilde{x}\in\boo$ such that $B(\tilde{x})$ is irreducible.
\end{itemize}\vspace{-2mm}

A number $\lm$ is called the principal eigenvalue of ${\mathscr B}$ if it is an eigenvalue of ${\mathscr B}$ and corresponding eigenfunction is strongly positive. The principal eigenvalue of ${\mathscr B}$ is denoted by $\lm_p({\mathscr B})$. We also remark that $\lm_p({\mathscr B})$ must be the spectral bound of ${\mathscr B}$ (see, e.g., \cite[Lemma 2.4]{Zhang24}).
Throughout this paper, we always assume that {\bf(J)} and {\bf(D)} hold.
The first main result of this paper is the following theorem.\vspace{-1mm}

\begin{thmx}\lbl{thA}\rm(Approximation and characterization of the generalized principal eigenvalue)\, Assume that the condition {\bf(B)} holds. Then there exist $\ud B^\ep(x)=(\ud b^\ep_{ik}(x))_{n\times n}$ and $\ol B^\ep(x)=(\bar b^\ep_{ik}(x))_{n\times n}$, with $\ep>0$, satisfying
 \begin{enumerate}\vspace{-1mm}
\item[$\bullet$] $\ud b^\ep_{ik}, \bar b^\ep_{ik}\in C(\boo)$, and $\ud b^\ep_{ik}$ and $\bar b^\ep_{ik}$ are decreasing and increasing in $\ep$, respectively; and
 \bess
 \ud b^\ep_{ik}\le b_{ik}\le\bar b^\ep_{ik}\;\;\;\mbox{in}\;\;\ol\oo,\;\;\;
 \lim_{\ep\to 0^+}\ud b^\ep_{ik}=\lim_{\ep\to 0^+}\bar b^\ep_{ik}=b_{ik} \;\;\;\mbox{in}\;\; C(\ol\oo),
  \eess
  \end{enumerate}\vspace{-1mm}
such that\vspace{-1mm}
\begin{enumerate}[$(1)$]
\item operators $\ud{\mathscr B}^\ep$ and $\ol{\mathscr B}^\ep$
 have principal eigenvalues $\lm_p(\ud{\mathscr B}^\ep)$ and $\lm_p(\ol{\mathscr B}^\ep)$, respectively, where $\ud{\mathscr B}^\ep$ and $\ol{\mathscr B}^\ep$ are defined by \qq{1.5} with $b_{ik}$ replaced by $\ud b^\ep_{ik}$ and $\ol b^\ep_{ik}$, respectively.
  \vskip 6pt
\item $\lm_p(\ud{\mathscr B}^\ep)\le\lm_p(\ol{\mathscr B}^\ep)$, and $\lm_p(\ud{\mathscr B}^\ep)$ and $\lm_p(\ol{\mathscr B}^\ep)$ are strictly decreasing and increasing in $\ep$, respectively, and
  \bes
 \lim_{\ep\to 0^+}\lm_p(\ud{\mathscr B}^\ep)=\lim_{\ep\to 0^+}\lm_p(\ol{\mathscr B}^\ep)=\lm({\mathscr B}).
  \lbl{1.6}\ees
 \item $\lm({\mathscr B})$ has the characterization:
 \bes
\lm({\mathscr B})&=&\inf_{\phi\in X^{++}}\,\sup_{x\in\oo,\,i\in\sss}\frac{d_i\int_\oo J_i(x,y)\phi_i(y)\dy+\sum_{k=1}^n b_{ik}(x)\phi_k(x)}{\phi_i(x)}\nm\\[2mm]
 &=&\sup_{\phi\in X^{++}}\,\inf_{x\in\oo,\,i\in\sss}\frac{d_i\int_\oo J_i(x,y)\phi_i(y)\dy+\sum_{k=1}^n b_{ik}(x)\phi_k(x)}{\phi_i(x)},
 \lbl{1.7}\ees
where
  \[X^{++}=\{\phi=(\phi_1,\cdots,\phi_n):\,\phi_i\in C(\boo),\;\, \phi_i(x)>0\;\; {\rm in }\;\boo\}.\]
 \end{enumerate}
  \end{thmx}\vspace{-1mm}

We call this number $\lm({\mathscr B})$ the {\rm generalized principal eigenvalue of ${\mathscr B}$}. Then $\lm({\mathscr B})$ is continuous respect to $B(x)$ (see, e.g., \cite[Lemma A.1]{Zhang24}).}  Clearly, the equation \qq{1.7} holds when $\lm({\mathscr B})$ is indeed a principal eigenvalue. In the following we will see that the approximation \qq{1.6} is also a powerful tool in applications, and the generalized principal eigenvalue plays the same role as the usual principal eigenvalue in analyzing the global dynamics of the problem \qq{1.1}.

Throughout this paper we always assume that $f_i(x,u)\in C(\boo\times\mathbb{R}_+^n)$ and denote $f(x,u)=(f_1(x,u),\cdots,f_n(x,u))$. Sometimes, we need the following assumptions.\vspace{-0.2mm}
\begin{itemize}
	\item[{\bf(H1)}]  $f_i(x,0)=0$ for all $x\in\boo$ and $i\in\sss$; $\partial_{u_k}f_i(x,u)$ is continuous in $\boo\times\mathbb{R}_+^n$ for all $i,k\in\sss$; and $f(x,u)$ is cooperative in $u\ge 0$, i.e., $\partial_{u_k}f_i(x,u)\ge 0$ for $x\in\boo$, $u\ge 0$ and $k\not=i$;\vspace{1.5mm}
 \item[{\bf(H2)}]  there exists $\tilde{x} \in \boo$ such that
 \bess
\kk(\partial_{u_k}f_i(\tilde x,u)\rr)_{n \times n}\;\;\mbox{is irreducible for all }\; u \geq 0.\vspace{-3mm}\eess
	\item[{\bf(H3)}]  $f(x,u)$ is strictly subhomogeneous with respect to $u\gg 0$, i.e.,
 \bess
 f(x,\delta u)>\delta f(x,u),\;\;\forall\;\delta\in (0,1),\; x\in\ol\oo,\; u\gg 0
 \eess
\end{itemize}

The equilibrium problem of \qq{1.1} is
 \bes\begin{cases}\label{1.8}
  d_i\dd\int_\oo J_i(x,y)U_i(y)\dy-d^*_i(x)U_i+f_i(x,U)=0,\;\;\; x\in\ol\oo, \\
  i=1,\cdots,n.
  \end{cases} \ees
Then the operator ${\mathscr B}$, defined by \qq{1.5}, is the linearization operator at zero corresponding to \qq{1.8}, and the condition {\bf(B)} holds under conditions {\bf(H1)}--{\bf(H2)}. Let $\lm({\mathscr B})$ be the generalized principle eigenvalue of ${\mathscr B}$. The second main result of this paper is the following theorem.

\begin{thmx}\lbl{thB}\rm(Global dynamics of \qq{1.1}) Assume that {\bf(H1)}--{\bf(H3)} hold. Let $u(x,t;u_0)$ be the unique solution of \qq{1.1}. Then the following statements are valid:\vspace{-1.2mm}
 \begin{enumerate}[$(1)$]
\item If $\lm({\mathscr B})>0$ and there exists $\ol U\in [C(\boo)]^n$ with $\ol U\gg 0$ in $\boo$ such that
	\bes
d_i\dd\int_\oo J_i(x,y)\ol U_i(y)\dy-d^*_i(x)\ol U_i
+f_i(x,\ol U)\leq 0,\;\; x\in\ol\oo,
	\lbl{1.9}\ees
for any $i \in \sss$, then \eqref{1.8} has a unique bounded positive solution $U$, and $U\in[C(\boo)]^n$, $U\le\ol U$ in $\boo$. Moreover,
 \bes
\lim\limits_{t\rightarrow +\infty} u(x,t;u_0)= U(x) \;\;\text{ uniformly in } \; \overline{\Omega}.\lbl{1.10}\ees
\item If $\lm({\mathscr B})< 0$, then \eqref{1.8} has no positive solution in $[C(\boo)]^n$. Moreover,  there exist $\sigma>0$ and $C>0$ such that
	\bes
u(x,t;u_0)\le C{\rm e}^{-\sigma t},\;\;\forall\; x\in\boo,\; t>0.
	\lbl{1.11}\ees
This shows that $u(x,t;u_0)$ converges exponentially to zero.\vspace{1.5mm}
	\item If $\lm({\mathscr B})=0$, and $f$ is strongly subhomogeneous, i.e., $f(x,\rho u) \gg \rho f(x,u)$ for all $x \in \ol \oo$, $u \gg 0$ and $\rho\in (0,1)$, then \eqref{1.8} has no positive solution in $[C(\boo)]^n$. If, in addition, there exists $\ol U\in [C(\boo)]^n$ with $\ol U\gg 0$ in $\boo$ such that \eqref{1.9} holds, then
	\bes
\lim_{t\rightarrow +\infty} u(x,t;u_0)=0\;\;\text{ uniformly in } \; \overline{\Omega}.\lbl{1.12}\ees
	\end{enumerate}\vspace{-2mm}
\end{thmx}

In Theorem \ref{thB}, we established the threshold dynamics using the generalized principal eigenvalue, which show that the generalized principal eigenvalue plays the same role as the usual principal eigenvalue, even for the critical case. This is also a new observation even for the scalar nonlocal dispersal  equation. In order to obtain this theorem, we first concern with the strong maximum principle and positivity of bounded non-negative solutions (without assumption about continuity), and the uniqueness and continuity of bounded positive solutions of \qq{1.8}. In the case where $\lm({\mathscr B})>0$ and \eqref{1.8} admits a positive upper solution, we prove that a bounded positive continuous solution of \qq{1.8} exists and that it is globally asymptotically attractive for \qq{1.1} using the upper and lower solutions method, where the lower solution is constructed by the lower control system. In the case where $\lm({\mathscr B})<0$, we obtain \eqref{1.11} by constructing a suitable upper solution using the upper control system. For the critical case ($\lm({\mathscr B})=0$), we first prove that the bounded positive continuous solution of \qq{1.8} does not exist by utilizing the Collatz-Wielandt characterization and show that the zero solution is globally asymptotically attractive combining with the perturbation methods and the results established in the case where $\lm({\mathscr B})>0$.
Besides, we also discuss the continuity of bounded non-negative solutions without assumptions {\bf(H1)}--{\bf(H3)}.

In this paper, we also investigate the dynamics of the West Nile virus model by applying the theoretical results obtained earlier. We begin by analyzing the dynamics of a limiting system and its perturbation systems using the conclusions established in Theorem \ref{thB}. Motivated by \cite{Wang24}, we employ the upper and lower solutions method to study the global dynamics, as opposed to the chain transitive sets theory. Compared to the chain transitive sets theory, the upper and lower solutions method is more fundamental and accessible for readers. Furthermore, when using the chain transitive sets theory to analyze nonlocal dispersal problems, additional conditions are required to ensure the asymptotic compactness of the solution mapping. These additional conditions can be bypassed by adopting the upper and lower solutions method, offering a more streamlined approach.

The organization of this paper is as follows. In Section 2, we prove Theorem \ref{thA}. We first construct the upper and lower control matrices $\ol B^\ep(x)$ and $\ud B^\ep(x)$ of $B(x)$, by perturbing $B(x)$ with identity matrix $I$, such that the corresponding operators $\ol{\mathscr B}^\ep$ and $\ud{\mathscr B}^\ep$ have principal eigenvalues $\lm_p(\ol{\mathscr B}^\ep)$ and $\lm_p(\ud{\mathscr B}^\ep)$, respectively. Then prove that $\lm_p(\ol{\mathscr B}^\ep)$ and $\lm_p(\ud{\mathscr B}^\ep)$ have the same limit, and this limit is exact the generalized principal eigenvalue $\lm({\mathscr B})$ of the operator ${\mathscr B}$. We remark that this perturbation method is very intuitive, simple and powerful in applications. In Section \ref{sec:TD} we study the threshold dynamics for cooperative systems. The existence, uniqueness, continuity and stabilities of positive equilibrium solutions are obtained. In Section \ref{sec:APP}, we use the abstract results obtained in Sections \ref{sec:GPE} and \ref{sec:TD} to investigate a West Nile virus model. The section \ref{sec:discussion} is a brief discussion.

\section{Proof of Theorem \ref{thA}}\setcounter{equation}{0} \label{sec:GPE}
{\setlength\arraycolsep{2pt}

Before giving the proof of Theorem \ref{thA}, we first state a sufficient condition to ensure the existence of the principal eigenvalue. Under the condition {\bf(B)}, by a variant of Perron-Frobenius Theorem, the maximum of the real parts of all eigenvalues of $B(x)$, denoted by $s(B(x))$, is an eigenvalue of $B(x)$; since $B(x)$ is also continuous in $x$, so is $s(B(x))$. Hence $\max_{\ol\oo}s(B(x))$ is well-defined.\vspace{-2mm}
	
\begin{lem}\lbl{le2.1}	Assume that the condition {\bf(B)} holds and there exists an open set $\Omega_0\subset\Omega$ such that $[\max_{\ol\oo}s(B(x))-s(B(x))]^{-1}\not \in L^1(\Omega_0)$. Then the operator ${\mathscr B}$ has a principal eigenvalue $\lm_p({\mathscr B})$.
 \end{lem}\vspace{-2mm}

Under the condition that $B(x)$ is irreducible for each $x\in\boo$, i.e., $B(x)$ is strongly irreducible, Lemma \ref{le2.1} was obtained in \cite[Theorem 2.2]{BS17} and \cite[Corollary 1.3]{SWZ23}. For our present case, Lemma \ref{le2.1} may be known, but we do not know the relevant literature. For the convenience of readers, we will provide its proof in the appendix.\vspace{-2mm}

\begin{proof}[Proof of Theorem \ref{thA}] {\it Step 1}: The construction of the lower control matrix $\ud B^\ep(x)$ of $B(x)$. Set
 \[\eta=\max_{\ol\oo}s(B(x)).\]
For the given $0<\ep\ll 1$, we define
 \[\oo_\ep=\big\{x\in\ol\oo:\,s(B(x))\ge\eta-\ep\big\}.\]
Then $\oo_\ep$ is a closed subset of $\ol\oo$, and $s(B(x))=\eta-\ep$ on $\partial\oo_\ep$. Define
 \bess
\ud b_{ik}^\ep(x)&=&b_{ik}(x),\;\;x\in\ol\oo\;\;\;\mbox{for}\;\;i\not=k,\eess
and
 \bess
 \ud b_{ii}^\ep(x)&=&\begin{cases}
 b_{ii}(x)-2\ep+\eta-s(B(x)),\; &x\in\oo_\ep,\\
 b_{ii}(x)-\ep,\; &x\in\ol\oo\setminus\oo_\ep.
  \end{cases}\eess
Then $\ud b_{ii}^\ep(x)\le b_{ii}(x)-\ep$ in $\boo$. Set $\ud B^\ep(x)=(\ud b_{ik}^\ep(x))_{n\times n}$, i.e.,
 \bess
 \ud B^\ep(x)=\left\{\begin{array}{ll}
 B(x)+[\eta-2\ep-s(B(x))]I,\; &x\in\oo_\ep,\\[1mm]
  B(x)-\ep I,\; &x\in\ol\oo\setminus\oo_\ep. \end{array}\right.
   \eess
Then $\ud B^\ep(x)$ is continuous and cooperative, and $\ud B^\ep(\tilde{x})$ is irreducible. It is clear that
 \bess\left\{\begin{array}{ll}
 s(\ud B^\ep(x))=s(B(x))+\eta-2\ep-s(B(x))=\eta-2\ep,\; &x\in\oo_\ep,\\[1.5mm]
 s(\ud B^\ep(x))=s(B(x))-\ep<\eta-2\ep,\; &x\in\ol\oo\setminus\oo_\ep.
  \end{array}\right.\eess
Define the operator $\ud{\mathscr B}^\ep$ as the manner \qq{1.5} with $b_{ik}$ replaced by $\ud b^\ep_{ik}$. Then $\ud{\mathscr B}^\ep$ has a principal eigenvalue by Lemma \ref{le2.1}, denoted by $\lm_p(\ud{\mathscr B}^\ep)$. Moreover, the following hold (\cite{SWZ23}):
  \bes
\lm_p(\ud{\mathscr B}^\ep)&=&\inf_{\phi\in X^{++}}\,\sup_{x\in\oo,\,i\in\sss}\frac{d_i\int_\oo J_i(x,y)\phi_i(y)\dy+\sum_{k=1}^n \ud b_{ik}^\ep(x)\phi_k(x)}{\phi_i(x)}\nm\\[2mm]
&=&\sup_{\phi\in X^{++}}\,\inf_{x\in\oo,\,i\in\sss}\frac{d_i\int_\oo J_i(x,y)\phi_i(y)\dy+\sum_{k=1}^n \ud b_{ik}^\ep(x)\phi_k(x)}{\phi_i(x)}.
 \lbl{2.1}\ees

{\it Step 2}: The construction of the upper control matrix $\ol B^\ep(x)$ of $B(x)$.
Define
\bess
\bar b_{ik}^\ep(x)=b_{ik}(x),\;\;x\in\ol\oo\;\;\;\mbox{for}\;\;i\not=k,\eess
and
  \bess
\bar b_{ii}^\ep(x)=\left\{\begin{array}{ll}
	b_{ii}(x)+\ep+\eta-s(B(x)),\; &x\in\oo_\ep,\\[1.5mm]
	b_{ii}(x)+2\ep,\; &x\in\ol\oo\setminus\oo_\ep. \end{array}\right.\eess
Then $\bar b_{ii}^\ep(x)\ge b_{ii}(x)+\ep$ in $\boo$. Set $\ol B^\ep(x)=(\bar b_{ik}^\ep(x))_{n\times n}$. Then $\ol B^\ep(x)$ is continuous and cooperative, and $\ol B^\ep(\tilde{x})$ is irreducible. Moreover,
  \bess\left\{\begin{array}{ll}
	s(\ol B^\ep(x))=s(B(x))+\ep+\eta-s(B(x))=\ep+\eta,\; &x\in\oo_\ep,\\[1.5mm]
	s(\ol B^\ep(x))=s(B(x))+2\ep<\ep+\eta,\; &x\in\ol\oo\setminus\oo_\ep.
  \end{array}\right.\eess
Define the operator $\ol{\mathscr B}^\ep$ as the manner \qq{1.5} with $b_{ik}$ replaced by $\bar b^\ep_{ik}$. Then $\ol{\mathscr B}^\ep$ has a principal eigenvalue by Lemma \ref{le2.1}, denoted by $\lm_p(\ud{\mathscr B}^\ep)$. Furthermore, the following hold (\cite{SWZ23}):
  \bes
\lm_p(\ol{\mathscr B}^\ep)&=&\inf_{\phi\in X^{++}}\,\sup_{x\in\oo,\,i\in\sss}\frac{d_i\int_\oo J_i(x,y)\phi_i(y)\dy+\sum_{k=1}^n \bar b_{ik}^\ep(x)\phi_k(x)}{\phi_i(x)}\nm\\[2mm]
&=&\sup_{\phi\in X^{++}}\,\inf_{x\in\oo,\,i\in\sss}\frac{d_i\int_\oo J_i(x,y)\phi_i(y)\dy+\sum_{k=1}^n \bar b_{ik}^\ep(x)\phi_k(x)}{\phi_i(x)}.
 \lbl{2.2}\ees
It is obvious that
 \bess
 \ol B^\ep(x)= \ud B^\ep(x)+3\ep I.
\eess

{\it Step 3}: From the constructions of $\ud B^\ep$ and $\ol B^\ep$, and the expressions of \qq{2.1} and \qq{2.2}, we easily see that
 \bess
 \lm_p(\ud{\mathscr B}^\ep)\le\inf_{\phi\in X^{++}}\,\sup_{x\in\oo,\,i\in\sss}\frac{d_i\int_\oo J_i(x,y)\phi_i(y)\dy+\sum_{k=1}^n b_{ik}(x)\phi_k(x)}{\phi_i(x)}\le \lm_p(\ol{\mathscr B}^\ep)=\lm_p(\ud{\mathscr B}^\ep)+3\ep,
 \eess
and $\lm_p(\ud{\mathscr B}^\ep)$ and $\lm_p(\ol{\mathscr B}^\ep)$ are strictly decreasing and increasing in $\ep$, respectively. Thus the limits $\lim_{\ep\to 0}\lm_p(\ud{\mathscr B}^\ep)$ and $\lim_{\ep\to 0}\lm_p(\ol{\mathscr B}^\ep)$ exist, and they are equal, denoted by $\lm({\mathscr B})$. This number $\lm({\mathscr B})$ is called the generalized principal eigenvalue of ${\mathscr B}$. Certainly, \qq{1.7} holds. The proof is complete.
\end{proof}\vspace{-2mm}

 
\section{Threshold dynamics for cooperative systems}\setcounter{equation}{0}\label{sec:TD}

As the applications of Theorem \ref{thA}, in this section we study the threshold dynamics for cooperative systems \qq{1.1}.

\subsection{Some properties of nonnegative solutions of \qq{1.8}}

Because the classical reaction diffusion problem
\bes\left\{\begin{array}{lll}
	u_{it}=d_i\Delta u_i+f_i(x,u),\; &x\in\oo,\;&t>0,\\[1mm]
	u_i=0,\;\;\mbox{or}\;\;\pp_\nu u_i=0,\; &x\in\pp\oo,\;&t>0,\\[1mm]
	u_i(x,0)=u_{i0}(x)\ge0,\,\not\equiv 0, &x\in\boo,\\[1mm]
	i=1,\cdots,n.
\end{array}\rr.\lbl{3.1}\ees
and its equilibrium problem have regularity, its non-negative equilibrium solution must belong to $[W^2_p(\oo)]^n$. In the study of equilibrium solutions, using the upper and lower solution method, the limit of the iterative sequence must be continuous. Moreover, if the solution $u(x,t)$ of \qq{3.1} is bounded and monotone in time $t$, then it must continuously converge to an equilibrium solution. However, the solution map of the nonlocal dispersal problem \qq{1.1} has no regularity, and hence, it is challenging to obtain the continuity of the limit of the iterative sequence.

In this part, we concern with the strong maximum principle and positivity of bounded non-negative solutions, and the uniqueness of bounded positive solutions of \qq{1.8}, and then discuss the continuity of bounded non-negative solutions.

We first provide the strong maximum principle.\vspace{-2mm}

\begin{lem}\lbl{l3.1}{\rm(Strong maximum principle)} Let $p\in L^\yy(\oo)$ and $0\le U\in L^\yy(\oo)$ satisfy
 \bes
 \int_\oo J(x,y)U(y)\dy-d^*(x)U(x)+p(x)U(x)\le 0,\;\;x\in\boo,
  \lbl{3.2}\ees
where $d^*(x)=d>0$ or $d^*(x)=dj(x)$. Then either $U\equiv 0$, or $U>0$ in $\boo$ and $\inf_{\oo}U>0$.
\end{lem}

\begin{proof} It is easy to see that the function
  \[h(x)=\int_\oo J(x,y)U(y)\dy\]
is nonnegative and continuous in $\boo$. If $h\equiv 0$, we can easily see that $U\equiv 0$.

Assume that there is $x_0\in\boo$ such that $h(x_0)>0$. Then
the set
 \[ {\cal O}=\{x\in\boo:\, h(x)>0\}\]
is an empty open subset of $\boo$ by the continuity of $h(x)$. We can also observe that $U(x)>0$ in ${\cal O}$.

We prove that ${\cal O}$ is a closed subset of $\boo$. Let $x_l\in{\cal O}$ and $x_l\to \bar x\in\boo$, then $U(x_l)>0$. If $h(\bar x)=0$, then $U=0$ in a neighborhood $B_\sigma(\bar x)\cap\boo$ for some $\sigma>0$ as $J(\bar x,\bar x)>0$ and $J(\bar x,y)$ is continuous in $y\in\boo$.  There exist $\tau >0$ small and $l$ large enough such that $B_{\tau}(x_l)\subset B_{\sigma}(\bar{x})$ and $B_{\tau}(x_l)\cap\boo\subset{\cal O}$. Hence, $U(x)>0$ for all $x\in B_{\tau}(x_l)\cap\boo$. This contradiction yields that $h(\bar x)>0$. Thus, ${\cal O}$ is a closed subset of $\boo$. So ${\cal O}=\boo$, and $U>0$ in $\boo$.

Suppose that $U>0$ in $\oo$ and $\inf_{\oo}U=0$. Then there exist $x_k\in\oo$ and $x_0\in\boo$ such that $x_k\to x_0$ and $U(x_k)\to 0$. Hence, $h(x_k)\to h(x_0)>0$ and $p(x_k)U(x_k)\to 0$. Taking $x=x_k$ in \qq{3.2} and letting $k\to+\yy$ we can get a contradiction.
\end{proof}\vspace{-2mm}

\begin{col}\lbl{c3.1} Let $p_{ik}\in L^\yy(\oo)$ and $(p_{ik}(x))_{n\times n}$ be cooperative, i.e., $p_{ik}(x)\ge 0$ in $\boo$ when $i\not=k$. Assume that $U\in[L^\yy(\oo)]^n$, $U\ge 0$ in $\boo$ and satisfies
 \bess
 d_i\dd\int_\oo J_i(x,y)U_i(y)\dy-d^*_i(x)U_i+\sum_{k=1}^np_{ik}(x)U_k< 0,\;\;\; x\in\ol\oo
 \eess
for all $i \in \sss$. Then $U_i>0$ in $\boo$ and $\inf_{\oo}U_i>0$ for all $i\in\sss$.
\end{col}

\begin{theo}\lbl{th3.1}{\rm(Strong maximum principle)} Let $p_{ik}\in{L^{\infty}(\Omega)}$ and $(p_{ik}(x))_{n\times n}$ be cooperative, i.e., $p_{ik}(x)\ge 0$ in $\boo$ when $i\not=k$. Assume that $U\in[L^\yy(\oo)]^n$, $U\ge 0$ in $\boo$ and satisfies
 \bess
d_i\dd\int_\oo J_i(x,y)U_i(y)\dy-d^*_i(x)U_i+\sum_{k=1}^np_{ik}(x)U_k\le 0,\;\;\; x\in\ol\oo
 \eess
for all $i \in \sss$. If there exists a nonzero measure set $\Omega_1\subset\boo$ such that $(p_{ik}(x))_{n\times n}$ is irreducible for all $x\in\Omega_1$, then we have that either $U_i\equiv 0$ in $\boo$ for all $i\in\sss$, or $U_i>0$ in $\boo$ and $\inf_{\oo}U_i>0$ for all $i\in\sss$.
\end{theo}

\begin{proof} For any $i\in\sss$, we have
 \[d_i\dd\int_\oo J_i(x,y)U_i(y)\dy-d^*_i(x)U_i+p_{ii}(x)U_i\le 0,\;\;\; x\in\ol\oo.\]
By Lemma \ref{l3.1}, we have that either $U_i\equiv 0$ in $\boo$, or $U_i>0$ in $\boo$ and $\inf_{\oo}U_i>0$.

Assume that there exists $\bar i \in \sss$ such that $U_{\bar i}\equiv 0$ in $\boo$. Then the set ${\cal I}=\big\{i\in\sss: U_i\equiv 0\;\;\mbox{in}\;\,\boo\big\}\not=\emptyset$. Define ${\cal I}^c=\sss\setminus {\cal I}$. We next prove that ${\cal I}^c$ is an empty set. Suppose that ${\cal I}^c$ is not empty, then $U_k>0$ in $\boo$ for all $k\in {\cal I}^c$, and $U_i\equiv 0$ in $\boo$ for all $i\in {\cal I}$ by Lemma \ref{l3.1} again. We then obtain
 \bess
 0&\ge& d_i\dd\int_\oo J_i(x,y)U_i(y)\dy-d^*_i(x)U_i(x)
+\sum_{k=1}^n p_{ik}(x)U_k(x)\\
 &=&\sum_{k\in {\cal I}^c}p_{ik}(x)U_k(x),\;\;\forall\; i\in {\cal I},\; x\in\Omega_1.
 \eess
This implies that $p_{ik}\equiv 0$ in $\Omega_1$ for all $i\in {\cal I}$, $k\in {\cal I}^c$, which contradicts the fact that $(p_{ik}(x))_{n\times n}$ is irreducible on $\Omega_1$. We thus have that ${\cal I}^c$ is empty. Hence, $U_i\equiv 0$ in $\boo$ for all $i\in\sss$. \end{proof}\vspace{-2mm}

Then we investigate the positivity of non-negative and nontrivial solutions.

\begin{theo}\lbl{th3.2}{\rm(Positivity of non-negative and nontrivial solutions)} Assume that {\bf(H1)}--{\bf(H2)} hold. Let $U\in[L^\yy(\oo)]^n$ and $U\ge 0$ in $\boo$ be a solution of \qq{1.8}. Then either $U_i\equiv 0$ in $\boo$ for all $i\in\sss$, or $U_i>0$ in $\boo$ and $\inf_{\oo}U_i>0$ for all $i\in\sss$.
\end{theo}

\begin{proof} Set
 \[p_{ik}(x)=\int_0^1\partial_{u_k}f_i(x,s U(x)){\rm d}s.\]
Then $p_{ik}\in C(\boo)$, $p_{ik}(x)\ge 0$ for $i\not =k$, and $(p_{ik}(\tilde x))_{n\times n}$ is irreducible by the assumptions {\bf(H1)} and {\bf(H2)}. Moreover,
 \[f_i(x, U)=f_i(x, U)- f_i(x,0)=\sum_{k=1}^np_{ik}(x)U_k(x),\]
and $U_i$ satisfies
 \bess
0=d_i\dd\int_\oo J_i(x,y)U_i(y)\dy-d^*_i(x)U_i(x)+\sum_{k=1}^np_{ik}(x)U_k(x).
 \eess
The desired conclusions can be deduced by Theorem \ref{th3.1}.
\end{proof}\vspace{-1mm}

Now we use Theorem \ref{th3.2} to derive the uniqueness of bounded positive solutions.\vspace{-1mm}

\begin{theo}\lbl{th3.3}{\rm(Uniqueness of bounded positive solutions)} Assume that {\bf(H1)}--{\bf(H3)} hold. Then \qq{1.8} has at most one bounded positive solution $U\; ($no need for continuity,
just $U\in[L^\yy(\oo)]^n$ and $U\gg 0$ in $\boo)$.
\end{theo}

\begin{proof} Let $U,V\in[L^\yy(\oo)]^n$ with $U, V\gg 0$ in $\boo$ be two solutions of \qq{1.8}. By Theorem \ref{th3.2}, $\inf_{\boo}U_i=:\alpha_i>0$ and  $\inf_{\boo}V_i=:\tau_i>0$ for all $i\in\sss$.

As $U,V\in[L^\yy(\oo)]^n$, we can find $0<q\ll 1$ such that $V\ge qU$ in $\boo$.
Hence, the set
 \[\Sigma=\kk\{0<q\le 1: V(x)\ge q U(x),\;\;\forall\;x\in\boo\rr\}\]
is nonempty. So $\bar{q}:=\sup\Sigma$ exists, and $V(x)\ge\bar qU(x)$ in $\boo$. Moreover, there exists $\bar i\in\sss$ such that $\inf_{\oo}(V_{\bar i}(x)-\bar qU_{\bar i}(x))=0$.

If $\bar{q}=1$, then $V(x)\ge U(x)$ in $\overline{\Omega}$.
Assume that $\bar{q}<1$. Since $f$ is strictly subhomogeneous, we have
 \[f(x,\bar{q}U(x))>\bar{q}f(x,U(x)), \;\; x \in \overline{\Omega}.\]
Write $W=(W_1,\cdots,W_n)=V-\bar{q}U$. It then follows that, for any given $x\in\boo$,
 \bess\begin{aligned}
0&=d_i\dd\int_\oo J_i(x,y)W_i(y)\dy-d^*_i(x)W_i(x)+f_i(x, V)-\bar{q}f_i(x,U)\\
	&\geq d_i\dd\int_\oo J_i(x,y)W_i(y)\dy-d^*_i(x)W_i(x)+f_i(x, V)- f_i(x,\bar{q}U)\\
	&=d_i\dd\int_\oo J_i(x,y)W_i(y)\dy-d^*_i(x)W_i(x)+\sum_{k=1}^n g_{ik}(x)W_k(x),
	\end{aligned}\eess
and the strictly inequality holds for some $i_x\in\sss$, where
 \[g_{ik}(x)=\int_0^1\partial_{u_k}f_i\big(x,\bar q\ol U(x)
 +\tau W(x)\big){\rm d}\tau\ge 0.\]
By the assumption {\bf(H2)}, there exist a positive number
 \[M>\max\kk\{\sup_{i\in\sss}\sup_{\Omega} U_i,\;
  \sup_{i \in \sss}\sup_{\Omega} V\rr\}\]
and a nonzero measure set $\Omega_1\subset\boo$, which is a neighborhood of $\tilde{x}$, such that $(\partial_{u_k}f_i(x,u))_{n\times n}$ is irreducible for all $(x,u)\in\Omega_1 \times [0,M]$.
Then $(g_{ik}(x)_{n\times n}$ is irreducible for all $x \in \Omega_1$. It follows from
 \[d_{i_x}\dd\int_\oo J_{i_x}(x,y)W_{i_x}(y)\dy-d^*_{i_x}(x)W_{i_x}(x)+g_{i_xi_x}(x)W_{i_x}(x)<0\]
that $W_{i_x}(x)>0$. So $W_{i_x}(z)>0$ for all $z\in\boo$ by Lemma \ref{l3.1}. On the other hand, since $\inf_{\oo}W_{\bar i}=0$, it deduces that $W_{\bar i}\equiv 0$ in $\boo$ by Lemma \ref{l3.1}. This is a contradiction with Theorem \ref{th3.1}. So $\bar q=1$ and $V\ge U$ in $\overline{\Omega}$.

Similarly, we can prove $U\ge V$ in $\overline{\Omega}$. The proof is complete.
\end{proof}\vspace{-1mm}

In the following we study the continuity of non-negative solutions. Under the conditions {\bf(H1)}--{\bf(H2)}, if $U\in [L^\yy(\oo)]^n$ is a nonnegative solution of \qq{1.8}, then either $U_i\equiv 0$ in $\boo$ for all $i\in\sss$, or $U_i>0$ in $\boo$ and $\inf_{\oo}U_i>0$ for all $i\in\sss$ by Theorem \ref{th3.2}. When the first situation occurs, $U$ is of course continuous. So we only need to discuss the continuity of the positive solution.\vspace{-1mm}

\begin{theo}\lbl{th3.4}{\rm (Continuity)} Assume that {\bf(H1)} and {\bf(H2)} hold. Let $U\in [L^\yy(\oo)]^n$ be a positive solution of \qq{1.8}. For any given $x_0\in\boo$, if the algebraic system
 \bes\begin{cases}
 	 d_i\dd\int_\oo J_i(x_0,y)U_i(y)\dy-d^*_i(x_0)v_i+f_i(x_0,v)=0,\\[1.2mm]
 	 i=1,\cdots,n
 \end{cases} \lbl{3.3} \ees
has at most one positive solution $v$, then $U$ is continuous at $x_0$.\vspace{-1mm}
\end{theo}

\begin{proof} First, by Theorem \ref{th3.2}, $\inf_{\oo}U_i>0$ for all $i\in\sss$.  Let $x_0,x_l\in\boo$ and $x_l\to x_0$. As $J_i(x,y)$ is continuous in $x,y\in\boo$ and $U$ is bounded, we have
 \bess
 d_i\dd\int_\oo J_i(x_l,y)U_i(y)\dy\to d_i\dd\int_\oo J_i(x_0,y)U_i(y)\dy
 \eess
by the dominated convergence theorem. On the other hand, there exist a subsequence $\{U(x_{l'})\}$ of $\{U(x_l)\}$ and $\widetilde U$, satisfying $\inf_{\oo}U_i\le\widetilde U_i\le M$ for all $i\in\sss$, such that $\lim_{l'\to+\yy}U(x_{l'})=\widetilde U$. Take $l'\to+\yy$ in
 \[d_i\dd\int_\oo J_i(x_{l'},y)U_i(y)\dy-d^*_i(x_{l'})U_i(x_{l'})+f_i(x_{l'},U(x_{l'}))=0\]
to obtain
  \bes\begin{cases}
 	d_i\dd\int_\oo J_i(x_0,y)U_i(y)\dy-d^*_i(x_0)\widetilde U_i+f_i(x_0,\widetilde U)=0,\\[2mm]
 	i=1,\cdots,n.
 \end{cases}
 \lbl{3.4}\ees
Noticing that $U(x_0)$ is a positive solution of \qq{3.4}. By the uniqueness, $U(x_0)=\widetilde U$, i.e., $U(x_{l'})\to U(x_0)$. This process also illustrates $U(x_l)\to U(x_0)$ as $l\to+\yy$. So, $U$ is continuous at $x_0$.
\end{proof}\vspace{-1mm}

In the rest of this subsection, we handle general situations, which do not require $f(x,u)$ to satisfy any one of the conditions {\bf(H1)}--{\bf(H3)}.\vspace{-.1mm}

\begin{theo}\lbl{th3.5} Let $U\in [L^\yy(\oo)]^n$ be a nonnegative solution of \qq{1.8}. For any given $x_0\in\boo$, if there exist $\rho,\gamma>0$ such that $U_i\ge\gamma$ in $B_\rho(x_0)\cap\boo$ for all $i\in\sss$, and the algebraic system \qq{3.3} has at most one positive solution $v$, then $U$ is continuous at $x_0$.
\end{theo}

The proof of Theorem \ref{th3.5} is similar to that of Theorem \ref{th3.4}, and we omit the details.

For the case $n=2$ (systems with two components), we have the following theorem.

\begin{theo}\lbl{tA.2} Let $x_0\in\boo$. Suppose $f_i(x_0,0,0)=0$, and $f_1(x_0,0,z)\ge0$ and $f_2(x_0,z,0)\ge 0$ for all $z>0$. Let $U\in [L^\yy(\oo)]^2$ be a nonnegative solution of \qq{1.8} with $n=2$, and set
 \[h_1(x_0)=d_1\dd\int_\oo J_1(x_0,y)U_1(y)\dy,\;\;\;
 h_2(x_0)=d_2\dd\int_\oo J_2(x_0,y)U_2(y)\dy.\]
 Assume that the algebraic system
 \bes
 h_i(x_0)-d^*_i(x_0)v_i+f_i(x_0,v)=0,\; \;\;  i=1,2
 \lbl{3.5}\ees
as well as the algebraic equations
 \bes
  h_1(x_0)-d^*_1(x_0)w_1+f_1(x_0,w_1,0)=0
  \lbl{3.6}\ees
 and
  \bess
  h_2(x_0)-d^*_2(x_0)w_2+f_2(x_0,0,w_2)=0
 \eess
have at most one positive solutions $(v_1, v_2)$, $w_1$ and $w_2$, respectively. Then $U$ is continuous at $x_0$.
\end{theo}

\begin{proof} It is obvious that $h_1(x_0)\ge 0$ and $h_2(x_0)\ge 0$. When  $h_1(x_0)>0,h_2(x_0)>0$, we have $U_1(x_0)>0$ and $U_2(x_0)>0$ by the first and second equations of \qq{1.8}, respectively. That is, $U(x_0)$ is a positive solution of \qq{3.5}. Take $x_l\in\boo$ such that $x_l\to x_0$. Similar to the arguments in the proof of Theorem \ref{th3.4}, there exists $\widetilde U\ge 0$ such that $U(x_l)\to \widetilde U$ and $\widetilde U$ satisfies \qq{3.5}. This combined with $h_i(x_0)>0$ gives $\widetilde U_i>0$, $i=1,2$, i.e., $\widetilde U$ is a positive solution of \qq{3.5}. By the uniqueness, $\widetilde U=U(x_0)$. Hence $U$ are continuous at $x_0$.

When $h_1(x_0)>0$ and $h_2(x_0)=0$, we have $U_2\equiv 0$ in a neighborhood $B(x_0)$ of $x_0$ since $J_2(x_0,x_0)>0$ and $J_2(x,x)$ is continuous. Therefore, $U_2(x)$ is continuous at $x_0$. Thus $U_1$ satisfies
  \bess
  d_1\dd\int_\oo J_1(x,y)U_1(y)\dy-d^*_1(x)U_1+f_1(x,U_1,0)=0,\; x\in B(x_0)
  \eess
and $U_1(x_0)>0$. By the assumption, \qq{3.6} has at most one positive solution. Similar to the proof of the above case, we can show that $U_1(x)$ is continuous at $x_0$. Similarly, when $h_1(x_0)=0$ and $h_2(x_0)>0$, $U(x)$ is continuous at $x_0$.

When $h_1(x_0)=h_2(x_0)=0$, we have $U_1\equiv 0$ and $U_2\equiv 0$ in a neighborhood of $x_0$. Certainly, $U(x)$ is continuous at $x_0$.
\end{proof}\vspace{-1mm}

We remark that by applying ideas to discuss the continuity of non-negative solutions, we can also establish continuity under a series of tedious conditions for a general system. However, these conditions are overly complicated and difficult to verify, so we will omit them here.\vspace{-1mm}

\begin{col}\lbl{cA.1}{\rm(Scalar equation case)}  $U\in L^\yy(\oo)$ be a nonnegative solution of
 \bess
  d\dd\int_\oo J(x,y)U(y)\dy-d^*(x)U+f(x,U)=0,\; \;x\in\ol\oo.
  \eess
For any given $x_0\in\boo$, if the algebraic equation
 \bess
  d\dd\int_\oo J(x_0,y)U(y)\dy-d^*(x_0)w+f(x_0,w)=0
  \eess
has at most one positive solution $w$, then $U$ is continuous at $x_0$.
\end{col}\vspace{-1mm}

In the following we shall use the upper and lower solutions method to show the existence of positive solutions of \qq{1.8}, and use Theorem \ref{th3.3} to obtain the uniqueness. From now on in this section, we always assume that {\bf(H1)}--{\bf(H3)} hold.

\begin{theo}\label{th3.7}{\rm(The upper and lower solutions method)}  Assume that there exist $\ol U, \ud U\in [C(\boo)]^n$ with $\ol U>\ud U\gg 0$ such that
 \bess
d_i\dd\int_\oo J_i(x,y)\ol U_i(y)\dy-d^*_i(x)\ol U_i
+f_i(x,\ol U)\leq 0,\;\;\; x\in\ol\oo,
 \eess
and
 \bess
d_i\dd\int_\oo J_i(x,y)\ud U_i(y)\dy-d^*_i(x)\ud U_i
+f_i(x,\ud U)\geq 0,\;\;\; x\in\ol\oo
 \eess
for all $i \in \sss$. Then \eqref{1.8} has a unique bounded positive solution $U$, and $U\in[C(\boo)]^n$,  $\ud U\le U\le\ol U$ in $\boo$.
\end{theo}

\begin{proof} Using a basic iterative scheme, we can construct two sequences $\{\ol U^k\}$ and $\{\ud U^k\}$ satisfying
 \[\underline U\leq\underline U^k\leq\underline U^{k+1}\leq\ol U^{k+1}
 \leq\ol U^k\leq\ol U,\;\;\forall\;k\geq 1,\]
such that
 \[\lim_{k\to+\yy}\underline U^k(x)=\widetilde U(x),\;\;\;\lim_{k\to+\yy}\ol U^k(x)=\widehat U(x),
 \;\;x\in\boo.\]
Moreover, both $\widetilde U$ and $\widehat U$ are positive solutions of \qq{1.8}. Clearly, $\widetilde U\le\widehat U$ in $\boo$. It follows from Theorem \ref{th3.3} that $\widetilde U=\widehat U=:U$ in $\overline{\Omega}$. Certainly, $U\in[C(\boo)]^n$ as $\widetilde U$ and $\widehat U$ are semi-lower and semi-upper continuous, respectively. Making use of Theorem \ref{th3.3} again, we get the uniqueness of bounded positive solutions of \eqref{1.8}.
\end{proof}\vspace{-1mm}

\subsection{The global dynamics of \qq{1.1}---Proof of Theorem \ref{thB}}

In this subsection we prove Theorem \ref{thB}.

As above, we set $b_{ii}(x)= a_{ii}(x)-d_i^*(x)$ and $b_{ik}(x)=a_{ik}(x)$ for $i\not=k$. Then $B(\tilde{x})=(b_{ik}(\tilde{x}))_{n\times n}$ is irreducible.

(1) Assume that $\lm({\mathscr B})>0$ and there exists $\ol U\in [C(\boo)]^n$ with $\ol U\gg 0$ in $\boo$ such that \qq{1.9} holds.

{\it Step 1: The existence and uniqueness of continuous positive solution of \qq{1.8}}. Let $\ud B^\ep(x)$ be the lower control matrix of $B(x)$ as in Theorem \ref{thA}. Then $\lm_p(\ud{\mathscr B}^\ep)\to \lm({\mathscr B})$ as $\ep\to 0^+$, so $\lm_p(\ud{\mathscr B}^\ep)>0$ when $0<\ep\ll1$. Let $\ud\phi$ with $\|\ud\phi\|_{L^\yy(\oo)}=1$ be a positive eigenfunction corresponding to $\lm_p(\ud{\mathscr B}^\ep)$. Set $\ud U=\rho\ud\phi$ with $0<\rho\ll 1$ being determined later. Noticing that
 \bess
 f_i(x,\rho\ud\phi)=f_i(x,\rho\ud\phi)-f_i(x,0)=
 \sum_{ k=1}^n\kk(\int_0^1\partial_{u_k}f_i(x, \tau \rho\ud\phi){\rm d}\tau\rr)\rho\ud\phi_k=:\sum_{k=1}^na_{ik}^\rho(x)\rho\ud\phi_k,
 \eess
and $a_{ik}^\rho\to a_{ik}$ uniformly in $\boo$ as $\rho\to 0$. Then there exists $\rho_0>0$ such that
 \[\lm_p(\ud{\mathscr B}^\ep)\ud\phi_i+\sum_{k=1}^n[a_{ik}^\rho(x)-a_{ik}(x)]\ud\phi_k>0,
 \;\;x\in\boo\]
provided $0<\rho\le\rho_0$ because of $\lm_p(\ud{\mathscr B}^\ep)>0$ and $\ud\phi$ is positive and continuous in $\boo$. Therefore,
 \bess
 &&d_i\dd\int_\oo J_i(x,y)\ud U_i(y)\dy-d^*_i(x)\ud U_i+f_i(x,\ud U)\\
 &=& d_i\dd\int_\oo J_i(x,y)\ud U_i(y)\dy-d^*_i(x)\ud U_i
  +\sum_{k=1}^na_{ik}^\rho(x)\rho\ud\phi_k\\
 &=&\rho d_i\dd\int_\oo J_i(x,y)\ud\phi_i(y)\dy+\rho\sum_{k=1}^n\ud b^\ep_{ik}(x)\ud\phi_k+\rho \sum_{k=1}^n\kk\{(b_{ik}-\ud b_{ik}^\ep)\ud\phi_k
 +(a_{ik}^\rho-a_{ik})\ud\phi_k\rr\}\\
 &\ge&\rho \lm_p(\ud{\mathscr B}^\ep)\ud\phi_i+\rho\sum_{k=1}^n[a_{ik}^\rho(x)
 -a_{ik}(x)]\ud\phi_k\\
 &>&0,\;\;x\in\boo,\; 0<\rho\le\rho_0.
  \eess
Certainly, $\ud U\le\ol U$ in $\boo$ when $0<\rho\ll1$. By Theorem \ref{th3.7}, the system \eqref{1.8} has a unique bounded positive solution $U$, and $U\in[C(\boo)]^n$, $\ud U\le U\le\ol U$ in $\boo$.

{\it Step 2: The proof of \qq{1.10}.} Notice that $u(x,t;u_0)$ is continuous in $\boo\times\mathbb{R}_+$ and $u(x,t;u_0)\gg 0$ in $\boo$ for all $t>0$. Without loss of generality we may assume that $u_0\gg 0$ in $\boo$. Then there exist  $0<\rho\le\rho_0$ and $\gamma >1$ such that
 \[\ud U(x):=\rho\ud\phi(x)\le u_0(x)\leq \gamma \ol U(x),\;\;x\in\boo.\]
As $f$ is strictly subhomogeneous, we have $f(x,\gamma u)<\gamma f(x,u)$ for all $x\in\ol \oo$, $u \gg 0$ and $\gamma>1$. Remind this fact, it is easy to verify that $\gamma\ol U$ still satisfies \qq{1.9}.

On the other hand, by the comparison principle,
 \bes
 u(x,t;\ud U)\le u(x,t;u_0)\le u(x,t;\gamma\ol U),\;\;x\in\boo,\;t\ge 0,
  \lbl{3.7}\ees
and $u(x,t;\ud U)$ and  $u(x,t;\gamma\ol U)$ are strict increasing and decreasing in $t$, respectively. Therefore, the point-wise limits
 \[\lim_{t\to+\yy}u(x,t;\ud U)=\widetilde U(x),\;\;\;\lim_{t\to+\yy}
 u(x,t;\gamma\ol U)=\widehat U(x)\]
exist and are positive solutions of \qq{1.8}; and $\widetilde U$ and $\widehat U$ are semi-lower and semi-upper continuous, respectively. It is deduced by Theorem \ref{th3.3} that $\widetilde U=\widehat U=U$, the unique positive solution of \qq{1.8}. As $U\in[C(\boo)]^n$, recalling the relation \qq{3.7} and using the Dini Theorem, we see that the limit \qq{1.10} holds.

(2) Assume that $\lm({\mathscr B})<0$.

{\it Step 1: The nonexistence of continuous positive solution of \qq{1.8}}. Assume on the contrary that $U\in[C(\boo)]^n$ is a positive solution of \qq{1.8}. Then, for $0<\rho\ll 1$,
\bess
 f_i(x,\rho U)=f_i(x,\rho U)-f_i(x,0)=\sum_{k=1}^n\kk(\int_0^1\partial_{u_k}f_i(x, \tau\rho U){\rm d}\tau\rr)\rho U_k=:\sum_{k=1}^na^\rho_{ik}(x)\rho U_k,
 \eess
and $a_{ik}^\rho\to a_{ik}$ uniformly in $\boo$ as $\rho\to 0$.
Let $B_\rho(x)=(b_{ik}^\rho(x))_{n\times n}$ be a cooperative {\rrr matrix-valued} function with $b_{ik}^\rho(x)=a_{ik}^\rho(x)$ for $i\neq k$ and $b_{ii}^{\rho}=a_{ii}^{\rho}(x)-d^*_i(x)$, and defined ${\mathscr B}_{\rho}$ as \qq{1.5} with $B(x)$ replaced by $B_\rho(x)$. By the continuity, there exists $\rho_0>0$ such that $\lm({\mathscr B}_{\rho})<0$ for all $0<\rho\le\rho_0$. Let $\ol B^\ep_\rho(x)$ be the upper control matrix of $B_\rho(x)$ as in Theorem \ref{thA}. Then $\lm_p(\ol{\mathscr B}^\ep_\rho)<0$ when $0<\ep\ll1$. Besides, as $f$ is strictly subhomogeneous, we have that, by \qq{1.8},
 \bess
0&=&d_i\dd\int_\oo J_i(x,y)\rho U_i(y)\dy-d^*_i(x)\rho U_i+\rho f_i(x,U)\\
&\le &d_i\dd\int_\oo J_i(x,y)\rho U_i(y)\dy-d^*_i(x)\rho U_i+ f_i(x,\rho U)\\
&=&\rho d_i\dd\int_\oo J_i(x,y) U_i(y)\dy-\rho d^*_i(x) U_i
+\rho\sum_{k=1}^na^\rho_{ik}(x) U_k\\
&=&\rho d_i\dd\int_\oo J_i(x,y) U_i(y)\dy+\rho\sum_{k=1}^n b^\rho_{ik}(x) U_k\\
&\le&\rho d_i\dd\int_\oo J_i(x,y)U_i(y)\dy
+\rho\sum_{k=1}^n\bar b^{\rho,\ep}_{ik}(x)U_k,\;\;\; x\in\ol\oo.
 \eess
Since $U(x)$ is positive in $\boo$, it follows that $\lm_p(\ol{\mathscr B}^\ep_\rho)\ge 0$, and we have a contradiction. So, the system \eqref{1.8} has no continuous positive solution.

{\it Step 2: The proof of \qq{1.11}.} Let $\ol B^\ep(x)$ be the upper control matrix of $B(x)$ as in Theorem \ref{thA} with $0<\ep\ll1 $. Since $\lm({\mathscr B})<0$, there exists $0<\ep\ll1$ such that $\lm_p(\ol{\mathscr B}^\ep)<0$. Let $0\ll\bar\phi\in [C(\boo)]^n$, with $\|\bar\phi\|_{L^\yy(\oo)}=1$, be a positive eigenfunction corresponding to $\lm_p(\ol{\mathscr B}^\ep)$, and $0<\rho\ll 1$ be determined later. Set
 \[ a_{ik}^\rho(x)=\int_0^1\partial_{u_k}f_i(x,s\rho\bar\phi(x)){\rm d}s.\]
Then $a_{ik}^\rho\to a_{ik}$ uniformly in $\boo$ as $\rho\to 0$ and
 \bes
-\rho d^*_i(x)\bar\phi_i+f_i(x,\rho\bar\phi)&=&-\rho d^*_i(x)\bar\phi_i+\rho\sum_{k=1}^n a^\rho_{ik}(x)\bar\phi_k(x)\nm\\
 &=&\rho\sum_{k=1}^n\kk\{\bar b^\ep_{ik}(x)+
 [b_{ik}(x)-\bar b^\ep_{ik}(x)]+[a^\rho_{ik}(x)-a_{ik}(x)]\rr\}\bar\phi_k(x)\nm\\
&\le&\rho\sum_{k=1}^n\bar b^\ep_{ik}(x)\bar\phi_k(x)+
\rho\sum_{k=1}^n[a^\rho_{ik}(x)-a_{ik}(x)]\bar\phi_k(x),\;\;x\in\boo
 \lbl{3.8}\ees
 for all $i\in\sss$. Choose $0<\rho_1\ll 1$ such that, for all $0<\rho\le\rho_1$, there holds:
 \bess
 \sum_{k=1}^n[a^\rho_{ik}(x)-a_{ik}(x)]\bar\phi_k(x)<-\frac 12\lm_p(\ol{\mathscr B}^\ep)\bar\phi_i(x)
 \;\;x\in\boo\eess
for all $i\in\sss$. This together with \qq{3.8} gives
 \bes
 &&d_i\dd\int_\oo J_i(x,y)\rho\bar\phi_i(y)\dy-d^*_i(x)\rho\bar\phi_i+f_i(x,\rho\bar\phi)\nm\\
 &<&\rho d_i\dd\int_\oo J_i(x,y)\bar\phi_i(y)\dy+\rho\sum_{k=1}^n\bar b^\ep_{ik}(x)\bar\phi_k(x)-\frac \rho 2\lm_p(\ol{\mathscr B}^\ep)\bar\phi_i(x)\nm\\
&=&\frac1 2\lm_p(\ol{\mathscr B}^\ep)\rho\bar\phi_i(x)<0,\;\;\; x\in\ol\oo
 \lbl{3.9}\ees
for all $i\in\sss$. Set $\sigma=-\frac 12\lm_p(\ol{\mathscr B}^\ep)$ and $v(x,t)=\rho{\rm e}^{-\sigma t}\bar\phi(x)$.  As $\sigma>0$, we have $\rho{\rm e}^{-\sigma t}\le\rho$ for all $t\ge 0$. Of course, by \qq{3.9} we have that $v(x,t)$ satisfies, for all $i\in\sss$,
 \bess
 &&d_i\dd\int_\oo J_i(x,y)v_i(y,t)\dy-d^*_i(x)v_i(x,t)+f_i(x,v(x,t))
 <\frac1 2\lm_p(\ol{\mathscr B}^\ep)v_i(x,t),\;\; x\in\boo,\;t\ge 0,\eess
and so
 \bes
 d_i\dd\int_\oo J_i(x,y)v_i(y,t)\dy-d^*_i(x)v_i(x,t)+f_i(x,v(x,t))-v_{i,t}(x,t)<0,\;\; x\in\boo,\;t\ge 0.\lbl{3.10}\ees

Take $\gamma >1$ such that $u_0(x)\leq \gamma\rho\bar\phi(x)$ in $\boo$.
Then, by the comparison principle,
 \bes
 u(x,t;u_0)\le u(x,t;\gamma \rho\bar\phi),\;\;x\in\boo,\;t\ge 0.
   \lbl{3.11}\ees
Set $\bar u(x,t)= \gamma v(x,t)$. Noticing that $f$ is strictly subhomogeneous, i.e., $f(x,\gamma u)<\gamma f(x,u)$ for all $x\in\boo$, $u \gg 0$ and $\gamma>1$. In view of \qq{3.10}, we see that $\bar u(x,t)$ satisfies
 \bess
 &&d_i\dd\int_\oo J_i(x,y)\bar u_i(y,t)\dy-d^*_i(x)\bar u_i(x,t)+f_i(x,\bar u(x,t))-\bar u_{i,t}(x,t)\nm\\
 &\le&d_i\dd\int_\oo J_i(x,y)\bar u_i(y,t)\dy-d^*_i(x)\bar u_i(x,t)+\gamma f_i(x,v(x,t))-\bar u_{i,t}(x,t)\nm\\
 &=&\gamma\kk(d_i\dd\int_\oo J_i(x,y)v_i(y,t)\dy-d^*_i(x)v_i(x,t)+f_i(x,v(x,t))-v_{i,t}(x,t)\rr)\\
 &<&0,\;\;\;x\in\boo,\;t\ge 0
 \eess
for all $i \in \sss$. Since $\bar u(x,0)=\gamma \rho\bar\phi(x)$, the comparison principle gives $u(x,t;\gamma \rho\bar\phi)\le \gamma \rho{\rm e}^{-\sigma t}\bar\phi(x)$. This combines with \qq{3.11} implies \qq{1.11}.

(3) Assume that $\lm({\mathscr B})=0$ and there exists $\ol U\in [C(\boo)]^n$ with $\ol U\gg 0$ in $\boo$ such that \eqref{1.9} holds.

{\it Step 1: The nonexistence of continuous positive solution of \qq{1.8}}. Assume on the contrary that $U\in[C(\boo)]^n$ is a positive solution of \qq{1.8}. Since $f$ is strongly subhomogeneous, we have $\rho f_i(x,U)\ll f_i(x,\rho U)$ for all $x\in\overline{\Omega}$ and $\rho\in (0,1)$. Fix $\rho_0$ small enough. Then there exists $\eta_0>0$ such that $\rho_0 f_i(x,U)\leq f_i(x,\rho_0 U)-\eta_0\rho_0 U_i$ in $\boo$. Therefore,
  \bes
  \rho f_i(x,U)\leq \frac{\rho}{\rho_0} f_i(x,\rho_0 U)-\eta_0\rho U_i \leq  f_i(x,\rho U)-\eta_0\rho U_i,\;\;\forall\;\rho \in (0,\rho_0).
  \lbl{3.12}\ees
For any given $\delta>0$ small enough, there exists $\rho \in (0,\rho_0)$ small enough such that
 \[f_i(x,\rho U)=f_i(x,\rho U)-f_i(x,0) \leq \rho \sum_{k=1}^n(a_{ik}(x) +\delta) U_k,\;\; x\in\ol\oo.\]
It then follows that
	\bes
0&=&d_i\dd\int_\oo J_i(x,y)\rho U_i(y)\dy-d^*_i(x)\rho U_i+\rho f_i(x,U)\nm\\
&\leq& d_i\dd\int_\oo J_i(x,y)\rho U_i(y)\dy-d^*_i(x)\rho U_i
+f_i(x,\rho U)-\eta_0\rho U_i\nm\\
&\leq&  d_i\dd\int_\oo J_i(x,y)\rho U_i(y)\dy-d^*_i(x)\rho U_i+\rho \sum_{k=1}^n (a_{ik}(x)+\delta) U_k- \eta_0 \rho U_i,\;\; x\in\ol\oo.
	\lbl{3.13}\ees
This indicates that $\lambda(\widetilde{\mathscr B}_{\delta})\geq\eta_0$, where the operator $\widetilde{\mathscr B}_{\delta}$ is defined by \eqref{1.5} with $B(x)$ replaced by $\widetilde{B}_{\delta}(x)=(b_{ik}(x)+\delta)_{n \times n}$. Letting $\delta \rightarrow 0$, we can obtain that $\lm({\mathscr B})\geq\eta_0$. This contradiction indicates that \eqref{1.8} has no positive solution in $[C(\boo)]^n$.

{\it Step 2: The proof of \qq{1.12}}.

We can find a constant $\gamma >1$ such that $u_0\le \gamma\ol U$ in $\boo$. Then, by the comparison principle,
 \bes
 u(x,t; u_0)\le u(x,t; \gamma\ol U),\;\;x\in\boo,\; t\ge 0.
 \lbl{3.14}\ees

(i) \ud{The upper control function of $u(x,t; \gamma\ol U)$}. Let $\ol B^\ep(x)$ be the upper control matrix of $B(x)$ given in Theorem \ref{thA}. We will use $\ol B^\ep(x)$ to construct an upper control problem of $u(x,t; \gamma\ol U)$ to control it. For each $i \in \sss$, set
 \[f^\ep_i(x,u)=[\bar b_{ii}^\ep(x)-b_{ii}(x)]u_i+f_i(x,u).\]
Then $f^\ep(x,u)=(f^\ep_1(x,u),\cdots,f^\ep_n(x,u))$ is also cooperative, strongly subhomogeneous, and
 \[\kk(\partial_{u_k}f^\ep_i(\tilde{x},u)\rr)_{n \times n}\]
is irreducible for all $u \geq 0$.

Consider the following problem
 \bes \label{3.15} \begin{cases}
	d_i\dd\int_\oo J_i(x,y)U^\ep_i(y)\dy-d^*_i(x)U^\ep_i+f^\ep_i(x,U^\ep)=0,&x\in\ol\oo,\\[1.2mm]
	i=1,\cdots,n.
 \end{cases}\ees
Recalling that $0\ll \ol U\in [C(\boo)]^n$ satisfies \qq{1.9}. In consideration of $\gamma>1$ and the strongly subhomogeneity of $f$, there exists $\ep_0>0$ small enough such that
  \[f_i(x,\gamma\ol U)\leq\gamma f_i(x,\ol U)-\ep\gamma\ol U_i,\;\;\forall\; 0<\ep\le\ep_0,\; i\in\sss\]
(refer to the derivation of \qq{3.12}). As $\bar b_{ii}^\ep(x)-b_{ii}(x)\le\ep$
we have
  \[ f^\ep_i(x,\gamma\ol U)\le f_i(x,\gamma\ol U)+\ep\gamma\ol U_i\leq\gamma f_i(x,\ol U),\;\;\forall\; 0<\ep\le\ep_0,\; i\in\sss.\]
This combines with \qq{1.9} yields that, for any $0<\ep<\ep_0$ and $i\in\sss$,
 \bes
0&\geq & d_i\dd\int_\oo J_i(x,y)\gamma\ol U_i(y)\dy-\gamma d^*_i(x)\ol U_i
+\gamma f_i(x,\ol U)\nm\\
&\ge&d_i\dd\int_\oo J_i(x,y)\gamma\ol U_i(y)\dy-\gamma d^*_i(x)\ol U_i
+f^\ep_i(x,\gamma\ol U),\;\; x\in\ol\oo.
  \lbl{3.16}\ees

Owing to
 \bess
 \tilde a^\ep_{ii}(x):&=&\partial_{u_i}f^\ep_i(x,0)=\bar b_{ii}^\ep(x)-b_{ii}(x)
 +\partial_{u_i}f_i(x,0)=\bar b_{ii}^\ep(x)+d_i^*(x),\\
 \tilde a^\ep_{ik}(x):&=&\partial_{u_k}f^\ep_i(x,0)=a_{ik}(x)=b_{ik}(x),
 \;\;k\not=i,
 \eess
we have the following correspondence:
 \[\tilde b^\ep_{ii}(x)=\tilde a_{ii}^\ep(x)-d_i^*(x)=\bar b_{ii}^\ep(x),\;\;
 \tilde b^\ep_{ik}(x)=\tilde a_{ik}^\ep(x)=b_{ik}(x),\;\;k\not=i. \]
Set $\widetilde B^\ep(x)=(\tilde b_{ik}^{\ep}(x))_{n\times n}$. Then $\widetilde B^\ep(x)=\ol B^\ep(x)$, so $\lm_p(\widetilde{\mathscr B}^\ep)=\lm_p(\ol{\mathscr B}^\ep)>\lm({\mathscr B})=0$ by Theorem \ref{thA}.

Taking advantage of \qq{3.16} we have that, by the conclusion (1), the problem \qq{3.15} has a unique positive solution $U^\ep\in [C(\boo)]^n$ satisfying   $U^\ep\le\gamma\ol U$ in $\boo$, and the solution $u^\ep(x,t;\gamma\ol U)$ of
 \bess\begin{cases}
u^\ep_{i,t}=d_i\dd\int_\oo J_i(x,y)u^\ep_i(y,t)\dy-d^*_i(x)u^\ep_i
+f^\ep_i(x,u)=0,&x\in\ol\oo, \; t>0,\\[1mm]
u^\ep_i(x,0)=\gamma\ol U_i(x), &x\in\boo,\\
i=1,\cdots,n
 \end{cases}\eess
satisfies
 \bes
 \lim_{t\to+\yy}u^\ep(x,t;\gamma\ol U)=U^\ep(x) \;\;\text{ uniformly in } \; \overline{\Omega}.\lbl{3.17}\ees
Moreover,  as $f_i\le f^\ep_i$, by the comparison principle we have
 \bes
 u(x,t;\gamma\ol U)\le u^\ep(x,t;\gamma\ol U),\;\;x\in\boo,\; t\ge 0.
 \lbl{3.18}\ees

(ii) \ud{Prove $\lim_{\ep\to 0^+}U^\ep(x)=0$}. Notice that $f^\ep$ is increasing in $\ep>0$, so is $U^\ep$ by the comparison principle. Therefore, the limit
$\lim_{\ep\to 0^+}U^\ep(x)=U(x)\ge 0$
exists and is a nonnegative solution of \qq{1.8}. By Theorem \ref{th3.2},
either $U\equiv 0$ in $\boo$, or $U\gg 0$ in $\boo$ and $\inf_\oo U_i>0$ for all $i\in\sss$.

Now we prove $U\equiv 0$ in $\boo$. If $U\gg 0$ in $\boo$,
then $\theta_i:=\inf_{\boo}U_i(x)>0$ for all $i\in\sss$. It is deduced that
 \[U^\ep_i(x)\ge U_i(x)\ge\theta_i,\;\;\forall\; x\in\boo,\;i\in\sss,\]
where $U^\ep\in [C(\boo)]^n$ is the unique positive solution of \qq{3.15}. As $U^\ep\le\gamma\ol U$ in $\boo$, we can find $\beta_i>0$ such that $\rrr U_i^\ep\le\beta_i$ in $\boo$. Take $\Sigma:=\prod_{i=1}^n[\theta_i,\beta_i]$. Then $U^\ep(x)\in \Sigma$ for all $x\in\boo$. Noticing that $f^\ep(x,u)$ and $f(x,u)$ are strongly subhomogeneous, and $\lim_{\ep\to 0^+}f^\ep(x,u)=f(x,u)$ uniformly in $\boo\times\Sigma$. If we denote $f^0_i=f_i$, then $f^\ep(x,u)$ is continuous in $\ep\ge 0$, $x\in\boo$ and $u\ge 0$; and
  \[f^\ep(x,\rho u)\gg\rho f^\ep(x,u),\;\;\forall\;\ep\ge 0, \; x\in\ol\oo, \; u\gg 0, \; \rho\in (0,1).\]

For the given $0<\rho\ll 1$, the function
 \[h_i(x,w,\ep):=f_i^\ep(x, \rho w)-\rho f_i^\ep(x,w)\]
is continuous and positive in $\boo\times \Sigma\times[0,1]$. There exists $\sigma_i>0$ such that $h_i(x,w,\ep)\ge\sigma_i$ for all $(x,w,\ep)\in\boo\times \Sigma\times[0,1]$. Denote $U^0=U$. Then $U^\ep(x)\in \Sigma$ for all $x\in\boo$ and $0\le\ep\le 1$. So  $h_i(x,U^\ep(x) ,\ep)\ge\sigma_i$ for all $x\in\boo$ and $0\le\ep\le 1$. Hence there exists $\eta_i>0$ such that $h_i(x,U^\ep(x),\ep)\ge\eta_i\rho U^\ep_i(x)$ for all $x\in\boo$ and $0\le\ep\le 1$. Take $\eta_0=\min_{i\in\sss}\eta_i$. Then we have $h_i(x,U^\ep(x),\ep)\ge\eta_0\rho U^\ep_i(x)$ for all $x\in\boo$ and $0\le\ep\le 1$, i.e.,
 \[\rho f_i^\ep(x,U^\ep(x))\leq f_i^\ep(x,\rho U^\ep(x))-\eta_0\rho U^\ep_i(x),\;\;\forall\; x\in\boo,\;0\le\ep\le 1.\]
Using $f^\ep_i$ instead of $f_i$, similar to the discussion in Step 1 we can show that \qq{3.13} holds. It then follows that $\lambda_p(\widetilde {\mathscr B}^\ep_{\delta})\geq\eta_0$ for all $0\le\ep\ll 1$, where the operator $\widetilde{\mathscr B}^\ep_{\delta}$ corresponds to $\widetilde B^\ep_{\delta}(x)=(\bar b^\ep_{ik}(x)+\delta)_{n\times n}$. Letting $\delta\to 0$, we can obtain $\lm_p(\ol{\mathscr B}^\ep)=\lm_p(\widetilde {\mathscr B}^\ep)\geq\eta_0$ for all $0\le\ep\ll 1$. This contradicts with the fact that $\lim_{\ep\to 0^+}\lm_p(\ol{\mathscr B}^\ep)=\lm({\mathscr B})=0$. Therefore, $U\equiv 0$ in $\boo$. This together with \qq{3.17}, \qq{3.18} and \qq{3.14} derives \qq{1.12}. The proof of Theorem \ref{thB} is complete.

\begin{remark} In Theorem {\rm\ref{thB}(1)}, the existence of upper solution $\ol U$ of \qq{1.8} is necessary to obtain the positive solutions of \qq{1.8}.

The role of strict subhomogeneity conditions is manifested in the following aspects:
\begin{enumerate}[$(1)$]\vspace{-2mm}
\item If $\ud U$ is a lower solution of \qq{1.8}, then for any $0<\rho<1$, $\rho\ud U$ is still a lower solution of \qq{1.8}; If $\ol U$ is an upper solution of \qq{1.8}, then for any $\rho>1$, $\rho\ol U$ is still an upper  solution of \qq{1.8}. This provides convenience for constructing appropriate upper and lower solutions and estimating solutions.\vspace{1mm}
\item The strict subhomogeneous condition plays a crucial role in proving the uniqueness of positive equilibrium solutions.\vspace{1mm}
\item Especially, when $\lm({\mathscr B})<0$, the nonexistence of the positive solution can be proved by using strict subhomogeneity. At the same time, the strict subhomogeneous condition plays an important role in constructing the upper solution of the initial value problem by using the positive eigenfunction corresponding to the principle eigenvalue of the control problem.\vspace{1mm}
\item The case $\lm({\mathscr B})=0$ indicates that the zero solution is degenerate, which is a critical case.	Due to the loss of compactness for the solution maps, it is challenging to prove the nonexistence of positive solutions and the stability of the zero solution without additional conditions. Here, we prove these conclusions under a strong subhomogeneous condition.
\end{enumerate}
\end{remark}\vspace{-4mm}

\begin{remark}\lbl{r3.2} In Theorem {\rm\ref{thB}}, it is assumed that $f(x,u)$ is strictly subhomogeneous with respect to all $u\gg 0$, i.e., {\bf(H3)} holds. Especially, in Theorem {\rm\ref{thB}(3)}, it is required that $f$ is strongly subhomogeneous, i.e., $f(x,\rho u) \gg \rho f(x,u)$ for all $x\in\ol\oo$, $u \gg 0$ and $\rho\in (0,1)$. However, some models do not meet these conditions $($for example, the following system \qq{4.16}$)$. Based on the proof of Theorem {\rm\ref{thB}} we can see that these can be weakened.
\begin{enumerate}[$(1)$]\vspace{-2mm}
\item In Theorem {\rm\ref{thB}(1)}, the conditions {\bf(H3)} and \qq{1.9} can be replaced by the following assumptions:
\begin{enumerate}\vspace{-1mm}
\item [{\rm(1a)}] $\ol U$ is a strict upper solution, that is, for each $i\in \sss$,
\[d_i\dd\int_\oo J_i(x,y)\ol U_i(y)\dy-d^*_i(x)\ol U_i
+f_i(x,\ol U)< 0,\;\; x\in\ol\oo,\]
\item[{\rm(1b)}] $f(x,u)$ is subhomogeneous with respect to $u$, that is,
\[f(x,\delta u)\geq\delta f(x,u), \;\;\forall\; x\in\ol\oo, \; u \ge 0, \; \delta\in (0,1),\]
\item[{\rm(1c)}] $f(x,u)$ is strictly subhomogeneous with respect to $u$ with $0\ll u \leq\ol U$, that is,
  \[f(x,\delta u)>\delta f(x,u), \;\;\forall\;0\ll u \le\ol U,\; x\in\ol\oo, \;\delta\in (0,1).\]
 \end{enumerate}
 These show that when \qq{1.8} has a strict upper solution $\ol U$, it only needs $f(x,u)$ to be subhomogeneous with respect to all $u\gg 0$ and strictly subhomogeneous with respect to $0\ll u\le\ol U$, without the need for $f(x,u)$ to be strictly subhomogeneous with respect to all $u\gg 0$.\vspace{1mm}
\item In Theorem {\rm\ref{thB}(2)}, the condition {\bf(H3)} can be replaced by that $f(x,u)$ is subhomogeneous with respect to all $u$.\vspace{1mm}
\item In Theorem {\rm\ref{thB}(3)}, the conditions \qq{1.9} and that $f(x,u)$ is strongly subhomogeneous with respect to all $u\gg 0$ can be replaced by the following assumptions:

\hspace{4mm}The problem \qq{1.8} has a strongly strict upper solution $\ol U$, that is, for each $i \in \sss$,
\[d_i\dd\int_\oo J_i(x,y)\ol U_i(y)\dy-d^*_i(x)\ol U_i
+f_i(x,\ol U) \ll 0,\;\; x\in\ol\oo;\]
 $f(x,u)$ is subhomogeneous with respect to all $u\gg 0$, and $f(x,u)$ is strongly subhomogeneous with respect to $0\ll u\leq\ol U$, i.e., $f(x,\delta u)\gg \delta f(x,u)$ for all $0\ll u \le \ol U,\;x\in\ol\oo$ and $\delta\in(0, 1)$.
\end{enumerate}
\end{remark}\vspace{-2mm}

Before concluding this section, we will present the conclusion regarding the logistic equation \qq{1.3}, as it is a highly anticipated equation. Clearly, the function $u(a(x)-u)$ is strictly subhomogeneous with respect to $u\gg 0$, and a large constant $M>0$ is an upper solution of  \eqref{1.4}. Set $b(x)=a(x)-d^*(x)$. Then Theorem \ref{thB} holds for $g(x)$. Noticing that if $U$ is a nontrivial and nonnegative solution of \qq{1.4}, then $\inf_{\boo}U>0$ by Lemma \ref{l3.1}, and $U\in C(\boo)$ by Corollary \ref{cA.1}. Let $\lm({\mathscr B})$ be the generalized principle eigenvalue of ${\mathscr B}$. We have the following conclusion.\vspace{-1mm}

\begin{theo}\lbl{t3.3}\, Let $u(x,t;u_0)$ be unique solution of \qq{1.3}. \begin{enumerate}\vspace{-2mm}
\item[{\rm(1)}]\, If $\lm({\mathscr B})>0$, then \qq{1.4} has a unique positive solution $U\in C(\boo)$ and $\dd\lim_{t\to+\yy}u(x,t;u_0)=U(x)$ uniformly in $\ol\oo$.\vspace{1mm}
\item[{\rm(2)}]\, If $\lm({\mathscr B})\le 0$, then \qq{1.4} has no positive solution, and $\dd\lim_{t\to+\yy}u(x,t;u_0)=0$ uniformly in $\ol\oo$.
Especially, $u(x,t;u_0)$ converges exponentially to zero when $\lm({\mathscr B})<0$.
  \end{enumerate}\vspace{-2mm}
\end{theo}

\section{A West Nile virus model}\lbl{S5}
\setcounter{equation}{0}\label{sec:APP}
 {\setlength\arraycolsep{2pt}

Let $H_u(x,t)$, $H_i(x,t)$, $V_u(x,t)$ and $V_i(x,t)$ be the densities of susceptible birds, infected birds, susceptible mosquitoes, and infected mosquitoes at location $x$ and time $t$, respectively. Then $H(x,t)=H_u(x,t)+H_i(x,t)$ and $V(x,t)=V_u(x,t)+V_i(x,t)$ are respectively the total densities of birds and mosquitoes at location $x$ and time $t$. Recently, the authors of this paper (\cite{WZhang24})  proposed and studied the following West Nile (WN) virus model in the spatiotemporal heterogeneity with general boundary conditions
\bes\left\{\begin{aligned}
&\pp_tH_{u}=\nabla\cdot d_1\nabla H_u+a_1(H_u+H_i)-\mu_1H_u-c_1(H_u+H_i)H_u
-\ell_1H_uV_i,&&x\in\oo,\; t>0,\\
&\pp_tH_{i}=\nabla\cdot d_1\nabla H_i+\ell_1H_uV_i-\mu_1H_i-c_1(H_u+H_i)H_i,\!\!&&x\in\oo,\; t>0,\\
&\pp_tH_{u}=\nabla\cdot d_2\nabla V_u+a_2(V_u+V_i)-\mu_2V_u
-c_2(V_u+V_i)V_u-\ell_2H_iV_u,&&x\in\oo,\; t>0,\\
&\pp_tH_{i}=\nabla\cdot d_2\nabla V_i+\ell_2H_iV_u-\mu_2V_i-c_2(V_u+V_i)V_i,&&x\in\oo,\; t>0,\\
&\alpha_1\dd\frac{\partial H_u}{\partial\nu}+\beta_1H_u=\alpha_1\dd\frac{\partial H_i}{\partial\nu}+\beta_1H_i=0,&&x\in\partial\oo,\; t>0,\\[1mm]
&\alpha_2\dd\frac{\partial V_u}{\partial\nu}+\beta_2V_u=\alpha_2\dd\frac{\partial V_i}{\partial\nu}+\beta_2V_i=0,&&x\in\partial\oo,\; t>0,\\
&\big(H_u, H_i, V_u, V_i\big)=\big(H_{u0}(x), H_{i0}(x), V_{u0}(x), V_{i0}(x)\big),&&x\in\oo,\; t=0,\vspace{-2mm}
   \end{aligned}\rr.\lbl{4.1}\ees
where $\nu$ is the outward normal vector of $\partial\Omega$, all coefficients are functions of $(x,t)$ and $T$-periodic in time $t$; $d_1$ and $d_2$ are the diffusion rates of birds and mosquitoes, $a_1$ and $a_2$ ($\mu_1$ and $\mu_2$)  are the birth (death) rates of susceptible birds and mosquitoes, $c_1$ and $c_2$ are the loss rates of birds and mosquitoes population due to environmental crowding, $\ell_1$ and $\ell_2$ are the transmission rates of uninfected birds and uninfected mosquitoes, respectively; they are all H\"{o}lder continuous, nonnegative and nontrivial; $\mu_k, \ell_k\ge 0$, $d_k, a_k, c_k>0$ in $\overline\oo\times[0, T]$. It has been proved that the problem \qq{4.1} has a positive time periodic solution if and only if birds  and mosquitoes persist and the basic reproduction ratio is greater than one, and moreover the positive time periodic solution is unique and globally asymptotically stable when it exists.

In model \qq{4.1}, it is assumed that birds and mosquitoes exhibit random (local) diffusion. However, because the movements of birds and mosquitoes are mainly through flight, it is more reasonable to use nonlocal diffusion (long-distance diffusion) instead of local diffusion. In this part, we consider the following West Nile virus model with nonlocal dispersals
 \bes\left\{\begin{aligned}
& H_{ut}=d_1\!\dd\int_\oo\! J_1(x,y)H_u (y,t) \dy-d^*_1(x)H_u+a_1(x)(H_u+H_i)\\
 &\hspace{17mm}-\mu_1(x)H_u-c_1(x)(H_u+H_i)H_u
 -\ell_1(x)H_uV_i,&&x\in\oo,\; t>0,\\
 &H_{it}=d_1\!\dd\int_\oo\! J_1(x,y)H_i (y,t) \dy-d^*_1(x) H_i+\ell_1(x)H_uV_i\\
 &\hspace{17mm}-\mu_1(x)H_i-c_1(x)(H_u+H_i)H_i,\!\!&&x\in\oo,\; t>0,\\
 &V_{ut}=d_2\!\dd\int_\oo\! J_2(x,y)V_u (y,t) \dy-d^*_2(x) V_u+a_2(x)(V_u+V_i)\\
 &\hspace{17mm}-\mu_2(x)V_u-c_2(x)(V_u+V_i)V_u-\ell_2(x)H_iV_u,&&x\in\oo,\; t>0,\\
 &V_{it}=d_2\!\dd\int_\oo\! J_2(x,y)V_i (y,t) \dy-d^*_2(x) V_i+\ell_2(x)H_iV_u\\
	&\hspace{17mm}-\mu_2(x)V_i-c_2(x)(V_u+V_i)V_i,&&x\in\oo,\; t>0,\\
	&\big(H_u, H_i, V_u, V_i\big)=\big(H_{u0}(x), H_{i0}(x), V_{u0}(x), V_{i0}(x))>0,&&x\in\oo,\; t=0,\vspace{-2mm}
 \end{aligned}\rr.\lbl{4.2}\ees
where $d^*_k(x)$ is defined by the manner {\bf(D)}  and $J_k(x,y)$ satisfies the condition {\bd (J)}, $k=1,2$.

Throughout this section, we always assume that\vspace{-2mm}
 \begin{itemize}
\item[$\bullet$]\; initial functions and all coefficient functions are continuous in $\ol\oo$; and $a_k(x), c_k(x)>0$ and $\mu_k(x), \ell_k(x)\ge 0$ in $\ol\oo$, $k=1,2$. Moreover, there exists $\tilde x\in\boo$ such that $\ell_1(\tilde x)>0, \ell_2(\tilde x)>0$.\vspace{-2mm}
 \end{itemize}

\begin{theo}\lbl{t5.2} The problem \qq{4.2} has a unique global solution $(H_u, H_i, V_u, V_i)$, which is positive and bounded.
\end{theo}

\begin{proof} The proof is standard. For reader's convenience, we provide an outline of the proof. Firstly, the local existence can be proved by the upper and lower solutions method. The positivity and uniqueness are a direct application of the maximum principle. Let $T$ be the maximum existence time of $(H_u, H_i, V_u, V_i\big)$ and set
    \[H=H_u+H_i,\;\;\;V=V_u+V_i.\]
Then $H$ and $V$ satisfy
 \bes\left\{\!\begin{aligned}
 &H_t=d_1\!\dd\int_\oo\! J_1(x,y)H(y,t)\dy+[a_1(x)
 -d_1^*(x)-\mu_1(x)]H-c_1(x)H^2,&&\!x\in\boo,\; 0<t<T,\\
&H(x,0)=H_u(x,0)+H_i(x,0)>0,&&\!\!x\in\boo
	\end{aligned}\rr.\quad\;\;\;\lbl{4.3}\ees
and
 \bes\left\{\begin{aligned}
&V_t=d_2\!\dd\int_\oo\! J_2(x,y)V(y,t)\dy+[a_2(x)
 -d_2^*(x)-\mu_2(x)]V-c_2(x)V^2,&&\!x\in\oo,\; 0<t<T,\\
&V(x,0)=V_u(x,0)+V_i(x,0)>0,&&\!\!x\in\oo,
	\end{aligned}\rr.\quad\;\;\;\lbl{4.4}\ees
respectively. By the maximum principle, $H$ and $V$ exist globally and are bounded, i.e., $T=\infty$. Consequently, $(H_u, H_i, V_u, V_i\big)$ exists globally and is positive and bounded.
\end{proof}\vspace{-1mm}

The corresponding equilibrium problem of \qq{4.2} is
 \bes\left\{\!\begin{aligned}
&d_1\!\int_\oo\! J_1(x,y)\mathsf{H}_u(y)\dy-[d_1^*(x)+\mu_1(x)]\mathsf{H}_u+
a_1(x)(\mathsf{H}_u+\mathsf{H}_i)\\
&\hspace{17mm}-c_1(x)(\mathsf{H}_u+\mathsf{H}_i)\mathsf{H}_u
-\ell_1(x)\mathsf{H}_u\mathsf{V}_i=0,&&\!\!x\in\ol\oo,\\
&d_1\!\int_\oo\! J_1(x,y)\mathsf{H}_i(y)\dy-[d_1^*(x)+\mu_1(x)]\mathsf{H}_i+
\ell_1(x)\mathsf{H}_u\mathsf{V}_i
-c_1(x)(\mathsf{H}_u+\mathsf{H}_i)\mathsf{H}_i=0,&&\!\!x\in\ol\oo,\\
&d_2\!\int_\oo\! J_2(x,y)\mathsf{V}_u(y)\dy-[d_2^*(x)+\mu_2(x)]\mathsf{V}_u+
a_2(x)(\mathsf{V}_u+\mathsf{V}_i)\\
&\hspace{17mm}-c_2(x)(\mathsf{V}_u+\mathsf{V}_i)\mathsf{V}_u
-\ell_2(x)\mathsf{H}_i\mathsf{V}_u=0,&&\!\!x\in\ol\oo,\\
&d_2\!\int_\oo\! J_2(x,y)\mathsf{V}_i(y)\dy-[d_2^*(x)+\mu_2(x)]\mathsf{V}_i
+\ell_2(x)\mathsf{H}_i\mathsf{V}_u
-c_2(x)(\mathsf{V}_u+\mathsf{V}_i)\mathsf{V}_i=0,&&\!\!x\in\ol\oo.
 \end{aligned}\rr.\;\;\lbl{4.5}\ees

Now we state our basic ideas. Let $(\mathsf{H}_u, \mathsf{H}_i, \mathsf{V}_u, \mathsf{V}_i)$ be a nonnegative solution of \qq{4.5}, then $\mathsf{H}=\mathsf{H}_u+\mathsf{H}_i$ and $\mathsf{V}=\mathsf{V}_u+\mathsf{V}_i$ satisfy the following $(4.6_1)$ and $(4.6_2)$, respectively,
 \[d_k\dd\int_\oo J_k(x,y)\mathsf{U}(y)\dy
 +[a_k(x)-d_k^*(x)-\mu_k(x)]\mathsf{U}-c_k(x)\mathsf{U}^2=0,\;\;\;x\in\boo.
 \eqno(4.6_k)\]\setcounter{equation}{6}
Define operators ${\cal G}_k$, $k=1,2$, by
 \bes\label{4.7}
 {\cal G}_k[u]=d_k\dd\int_\oo J_k(x,y)u(y)\dy+g_k(x)u(x),
 \ees
and let $\lm({\cal G}_k)$ be the generalized principle eigenvalue of ${\cal G}_k$, where
 \[g_k(x)=a_k(x)-d_k^*(x)-\mu_k(x),\;\;k=1,2.\]
As $c_k(x)>0$ in $\boo$, we see that the large constant $M$ is an upper solution of $(4.6_k)$. By Theorem \ref{t3.3}, the problem $(4.6_k)$ has a unique positive solution which is globally asymptotically stable when $\lm({\cal G}_k)>0$, while $(4.6_k)$ has no positive solution and $0$ is asymptotically stable when $\lm({\cal G}_k)\le 0$.

In the following we assume that $\lm({\cal G}_k)>0$, $k=1,2$. Then $(4.6_1)$ and $(4.6_2)$  have unique positive solutions $\mathsf{H}$ and $\mathsf{V}$, respectively, and $\mathsf{H}$ and $\mathsf{V}$ are globally asymptotically stable, i.e.,
 \bes
 \lim_{t\to\yy}H(x,t)=\mathsf{H}(x),\;\;\;\lim_{t\to\yy}V(x,t)
 =\mathsf{V}(x)\;\;\;\;\mbox{in}\;\,C(\boo). \lbl{4.8}\ees
Set
 \[b_{11}^*(x,t)=g_1(x)-a_1(x)-c_1(x)H(x,t),\;\;
 b_{22}^*(x,t)=g_2(x)-a_2(x)-c_2(x)V(x,t).\]
Then $(H_i, V_i)$ satisfies
 \bes\left\{\!\begin{aligned}
&H_{it}=d_1\!\int_\oo\! J_1(x,y)H_i(y)\dy+b_{11}^*(x,t)H_i+\ell_1(x)(H(x,t)-H_i)V_i,&&\!\!x\in\ol\oo,\;t>0,\\[0.1mm]
&V_{it}=d_2\!\int_\oo\! J_2(x,y)V_i(y)\dy+b_{22}^*(x,t)V_i+\ell_2(x)(V(x,t)-V_i)H_i,&&\!\!x\in\ol\oo,\; t>0.
 \end{aligned}\rr.\qquad\lbl{4.9}\ees
According to \qq{4.8}, we know that the asymptotic autonomous system of \qq{4.9} is
 \bes\left\{\!\begin{aligned}
&U_t=d_1\!\int_\oo\! J_1(x,y)U(y)\dy+b_{11}(x)U
+\ell_1(x)(\mathsf{H}(x)-U)Z,&&\!\!x\in\ol\oo,\;t>0,\\[0.1mm]
&Z_t=d_2\!\int_\oo\! J_2(x,y)Z(y)\dy+b_{11}(x)Z
+\ell_2(x)(\mathsf{V}(x)-Z)U,&&\!\!x\in\ol\oo,\; t>0,
 \end{aligned}\rr.\qquad\lbl{4.10}\ees
where
 \bess
b_{11}(x)=g_1(x)-a_1(x)-c_1(x)\mathsf{H}(x),\;\;
 b_{22}(x)=g_2(x)-a_2(x)-c_2(x)\mathsf{V}(x)].
 \eess
The corresponding equilibrium problem of \qq{4.10} is
 \bes\left\{\!\begin{aligned}
&d_1\!\int_\oo\! J_1(x,y)\mathsf{H}_i(y)\dy+
b_{11}(x)\mathsf{H}_i+\ell_1(x)(\mathsf{H}(x)-\mathsf{H}_i)
\mathsf{V}_i=0,&&\!\!x\in\ol\oo,\\[0.1mm]
&d_2\!\int_\oo\! J_2(x,y)\mathsf{V}_i(y)\dy+b_{22}(x)\mathsf{V}_i
+\ell_2(x)(\mathsf{V}(x)-\mathsf{V}_i)\mathsf{H}_i=0,&&\!\!x\in\ol\oo.
 \end{aligned}\rr.\qquad\lbl{4.11}\ees

The basic ideas in this section are as follows:   Firstly, by use of the upper and lower solutions method we show that when $\lm({\mathscr B})>0$, the problem \qq{4.11} has a unique positive solutions $(\mathsf{H}_i, \mathsf{V}_i)$ and satisfies  $\mathsf{H}_i<\mathsf{H}$, $\mathsf{V}_i<\mathsf{V}$ in $\boo$, which implies that \qq{4.5} has a unique continuous positive solution, and that when $\lm({\mathscr B})\le 0$, the problem \qq{4.11} has no positive solution, and so \qq{4.5} has no positive solution. Secondly, we study the stabilities of non-negative equilibrium solutions.

\subsection{The existence and uniqueness of positive solutions of \qq{4.5}}\lbl{S4.1}
{\setlength\arraycolsep{2pt}\renewcommand {\baselinestretch}{1.3}

We need to emphasize that \qq{4.11} is not a cooperative system within the range of $(\mathsf{H}_i, \mathsf{V}_i)\ge (0,0)$, but only within the range of $(0,0)\le(\mathsf{H}_i, \mathsf{V}_i)\le(\mathsf{H}, \mathsf{V})$. To utilize the abstract conclusion presented earlier, it is necessary to consider an auxiliary system
 \bes\left\{\!\begin{aligned}
&d_1\!\int_\oo\! J_1(x,y)\mathsf{H}_i(y)\dy+b_{11}(x)\mathsf{H}_i+\ell_1(x)(\mathsf{H}(x)-\mathsf{H}_i)^+
\mathsf{V}_i=0,&&\!\!x\in\ol\oo,\\
&d_2\!\int_\oo\! J_2(x,y)\mathsf{V}_i(y)\dy+b_{22}(x)\mathsf{V}_i+\ell_2(x)(\mathsf{V}(x)-\mathsf{V}_i)^+
\mathsf{H}_i=0,&&\!\!x\in\ol\oo,
 \end{aligned}\rr.\qquad\lbl{4.12}\ees
which is a cooperative system within the range of $(\mathsf{H}_i, \mathsf{V}_i)\ge (0,0)$. In fact, we will show that \qq{4.11} and \qq{4.12} are equivalent.

Linearize the equations of \qq{4.11} at $(\mathsf{H}_i, \mathsf{V}_i)=(0,0)$ to obtain an operator ${\mathscr B}=({\mathscr B}_1, {\mathscr B}_2)$:
 \bess\left\{\begin{aligned}
&{\mathscr B}_1[\phi]=d_1\!\int_\oo\!J_1(x,y)\phi_1(y)\dy+b_{11}(x)\phi_1+b_{12}(x)\phi_2,\\[.1mm]
&{\mathscr B}_2[\phi]=d_2\!\int_\oo\! J_2(x,y)\phi_2(y)\dy+b_{21}(x)\phi_1+b_{22}(x)\phi_2,
  \end{aligned}\rr.\eess
where $\phi=(\phi_1,\phi_2)$ and
 \bess
  b_{12}(x)=\ell_1(x)\mathsf{H}(x),\;\;\;
 b_{21}(x)=\ell_2(x)\mathsf{V}(x).
 \eess
Then $B(x):=(b_{ik}(x))_{2\times 2}$ is a cooperative matrix, and $B(\tilde x)$ is irreducible. So Theorem \ref{thA} holds. Let $\lm({\mathscr B})$ be the generalized principle eigenvalue of ${\mathscr B}$.

\begin{theo}\lbl{t5.1} Assume that $\lm({\cal G}_k)>0$, $k=1,2$.\vspace{-2mm}
\begin{enumerate}
\item[{\rm(1)}]\; If $\lm({\mathscr B})>0$, then \qq{4.11} has a unique continuous positive solution $(\mathsf{H}_i, \mathsf{V}_i)$ and satisfies
 \[\mathsf{H}_i<\mathsf{H},\;\;\;\mathsf{V}_i<\mathsf{V},\;\;\;x\in\boo.\]
So \qq{4.5} has a unique continuous positive solution\vspace{-0.2mm}
 \[(\mathsf{H}_u, \mathsf{H}_i, \mathsf{V}_u, \mathsf{V}_i)=(\mathsf{H}-\mathsf{H}_i,\, \mathsf{H}_i,\,  \mathsf{V}-\mathsf{V}_i, \,\mathsf{V}_i).\vspace{-2mm}\]
\item[{\rm(2)}]\;If $\lm({\mathscr B})\le 0$, then \qq{4.11} has no positive solution, and so \qq{4.5} has no positive solution.
 \end{enumerate}
 \end{theo}

Before giving the proof of Theorem \ref{t5.1}, we first prove the a general conclusion (the following Lemma \ref{l4.1}), which will be used in the sequel. To that end, let's make some preparations.

Let $\ud g_k^\ep$ be the lower control function of $g_k$ given in Theorem \ref{thA}, and $\phi^\ep_k$ be the positive eigenfunction corresponding to $\lm_p(\ud{\cal G}_k^\ep)$ with $\|\phi^\ep_k\|_{L^\infty(\oo)}=1$, $k=1,2$, where $\ud{\cal G}_k^\ep$ is defined by \eqref{4.7} with $g_k$ replaced by $\ud g_k^\ep$. Then $(\lm_p(\ud{\cal G}_k^\ep),\,\phi^\ep_k)$ satisfies
  \[d_k\dd\int_\oo J_k(x,y)\phi^\ep_k(y)\dy+\ud g_k^\ep\phi^\ep_k(x)
  =\lm_p(\ud{\cal G}_k^\ep)\phi^\ep_k(x),\;\;x\in\boo,\;\;k=1,2.
 \eqno(4.13_k)\]
As $\lm({\cal G}_k)>0$, there exists $0<\ep_0\ll 1$ such that{ \setcounter{equation}{13}}
 \bess
 \lm_p(\ud{\cal G}_k^\ep)>0,\;\;\;\forall\, 0<\ep\le\ep_0,\; k=1,2.
 \eess
When $|\sigma|\ll 1$, there hold: $\mathsf{H}+\sigma\phi^\ep_1>0$ and $\mathsf{V}+\sigma\phi^\ep_2>0$ in $\boo$. Set
 \bes\begin{cases}
 b_{11}^\sigma(x)=g_1(x)-a_1(x)-c_1(x)[\mathsf{H}(x)-\sigma\phi^\ep_1(x)],
 \;\;b_{12}^\sigma(x)=\ell_1(x)[\mathsf{H}(x)+\sigma\phi^\ep_1(x)],\\[1mm]
 b_{21}^\sigma(x)=\ell_2(x)[\mathsf{V}(x)+\sigma\phi^\ep_2(x)],\;\; b_{22}^\sigma(x)=g_2(x)-a_2(x)-c_2(x)[\mathsf{V}(x)-\sigma\phi^\ep_2(x)].&
 \end{cases}\lbl{4.14}\ees
Then $b_{11}^\sigma<0$, $b_{22}^\sigma<0$ in $\boo$ and
 \bes
 B_\sigma(x)=(b_{kl}^\sigma(x))_{2\times 2}
  \lbl{4.15}\ees
is a cooperative matrix provided $|\sigma|\ll 1$. Moreover, $B_\sigma(\tilde x)$ is irreducible. Now we provide  the following lemma.

\begin{lem}\lbl{l4.1} Assume that $\lm({\cal G}_k)>0$, $k=1,2$, and $\lm({\mathscr B})>0$.
\begin{enumerate}\vspace{-2mm}
\item[{\rm(1)}]\;There is $0<\sigma_0\ll 1$ such that, when $|\sigma|\le\sigma_0$, the problems
 \bes\left\{\begin{aligned}
&d_1\!\int_\oo\! J_1(x,y)\mathsf{H}_i(y)\dy+b_{11}^\sigma(x)\mathsf{H}_i
+\ell_1(x)\big(\mathsf{H}(x)+\sigma\phi^\ep_1(x)-\mathsf{H}_i\big)^+ \mathsf{V}_i=0,&&\!\!x\in\ol\oo,\\
&d_2\!\int_\oo\! J_2(x,y)\mathsf{V}_i(y)\dy+b_{22}^\sigma(x)\mathsf{V}_i
+\ell_2(x)\big(\mathsf{V}(x)+\sigma\phi^\ep_2(x)-\mathsf{V}_i\big)^+ \mathsf{H}_i=0,&&\!\!x\in\ol\oo,
  \end{aligned}\rr.\quad\;\;\lbl{4.16}\ees
and
  \bes\left\{\begin{aligned}
&d_1\!\int_\oo\! J_1(x,y)\mathsf{H}_i(y)\dy+b_{11}^\sigma(x)\mathsf{H}_i
+\ell_1(x)\big(\mathsf{H}(x)+\sigma\phi^\ep_1(x)-\mathsf{H}_i\big)\mathsf{V}_i=0, &&\!\!x\in\ol\oo,\\
&d_2\!\int_\oo\! J_2(x,y)\mathsf{V}_i(y)\dy+b_{22}^\sigma(x)\mathsf{V}_i
+\ell_2(x)\big(\mathsf{V}(x)+\sigma\phi^\ep_2(x)-\mathsf{V}_i\big) \mathsf{H}_i=0,&&\!\!x\in\ol\oo,
 \end{aligned}\rr.\quad\;\;\lbl{4.17}\ees
have, respectively, unique continuous positive solutions $(\widehat{\mathsf{H}}_i^\sigma,\widehat{\mathsf{V}}_i^\sigma)$ and $(\mathsf{H}_i^\sigma, \mathsf{V}_i^\sigma)$. Moreover,
 \bess
 \widehat{\mathsf{H}}_i^\sigma, \,\mathsf{H}_i^\sigma<\mathsf{H}+\sigma\phi^\ep_1, \;\;\;\widehat{\mathsf{V}}_i^\sigma,\, \mathsf{V}_i^\sigma<\mathsf{V}+\sigma\phi^\ep_2\;\;\;{\rm in}\;\;\boo.
 \eess
Therefore, $(\widehat{\mathsf{H}}_i^\sigma, \widehat{\mathsf{V}}_i^\sigma)
 =(\mathsf{H}_i^\sigma, \mathsf{V}_i^\sigma)$ in $\boo$. This implies that \qq{4.16} and \qq{4.17} are equivalent.\vskip 3pt
 \item[{\rm(2)}] Any bounded nonnegative solution $(\mathsf{H}_i^\sigma, \mathsf{V}_i^\sigma)$ of \qq{4.17} is continuous in $\boo$, and then the bounded positive solution of \qq{4.17} is unique. Any nonnegative solution $(\widehat{\mathsf{H}}_i^\sigma,
    \widehat{\mathsf{V}}_i^\sigma)$ of \qq{4.16} satisfying $(\widehat{\mathsf{H}}_i^\sigma,\widehat{\mathsf{V}}_i^\sigma)
    \le(\mathsf{H}+\sigma\phi^\ep_1,\,\mathsf{V}+\sigma\phi^\ep_2)$ is continuous in $\boo$, and then the positive solution of \qq{4.16} that does not exceed $(\mathsf{H}+\sigma\phi^\ep_1,\mathsf{V}+\sigma\phi^\ep_2)$ is unique.\vspace{-2mm}
\end{enumerate}
\end{lem}

\begin{proof} (1) Here, we only discuss the existence and uniqueness of positive solutions for problem \qq{4.16}. The existence and uniqueness of positive solutions for problem \qq{4.17} can be addressed in a similar manner, which is simpler than that of \qq{4.16}.

{\it Step 1: The construction of upper solution of \qq{4.16}}. Owing to $a_1, a_2, \mathsf{H}, \mathsf{V}>0$ in $\boo$, we can choose $0<\sigma_0\ll 1$ such that, when $|\sigma|\le\sigma_0$,
 \bes\left\{\begin{aligned}
&\mathsf{H}+\sigma\phi^\ep_1>0, \;\; \mathsf{V}+\sigma\phi^\ep_2>0\;\;&&
\mbox{in}\;\;\ol\oo, \\[.1mm]
&a_1(\mathsf{H}+\sigma\phi^\ep_1)>\sigma\phi^\ep_1(\lm_p(\ud{\cal G}_1^\ep)+g_1-\ud g_1^\ep+\sigma c_1\phi^\ep_1)\;\;&&\mbox{in}\;\;\ol\oo,\\[.1mm]
&a_2(\mathsf{V}+\sigma\phi^\ep_2)>\sigma\phi^\ep_2(\lm_p(\ud{\cal G}_2^\ep)+g_2-\ud g_2^\ep+\sigma c_2\phi^\ep_2)\;\;&&\mbox{in}\;\;\ol\oo.
 \end{aligned}\rr.\lbl{4.18}\ees

Now we verify that
 \bess
 (\ol{\mathsf{H}}_i(x), \ol{\mathsf{V}}_i(x))=\big(\mathsf{H}(x)+\sigma\phi^\ep_1(x), \mathsf{V}(x)+\sigma\phi^\ep_2(x)\big)\eess
is a strict upper solution of \qq{4.16}. Making use of the second inequality of \qq{4.18} we have that, by the direct calculations,
 \bess
&&d_1\!\int_\oo\! J_1(x,y)\ol{\mathsf{H}}_i(y)\dy+b_{11}^\sigma(x)\ol{\mathsf{H}}_i
+\ell_1(x)\big(\mathsf{H}(x)+\sigma\phi^\ep_1(x)-\ol{\mathsf{H}}_i\big)^+
\ol{\mathsf{V}}_i\\[.2mm]
&=&d_1\!\int_\oo\! J_1(x,y)\mathsf{H}(y)\dy+g_1(x)\mathsf{H}(x)-c_1(x)\mathsf{H}^2(x)
+\sigma\kk(d_1\!\int_\oo\! J_1(x,y)\phi^\ep_1(y)\dy+\ud g_1^\ep(x)\phi^\ep_1(x)\rr)\\[1mm]
&& -a_1(x)[\mathsf{H}(x)+\sigma\phi^\ep_1(x)]+\sigma[g_1(x)-\ud g_1^\ep(x)]\phi^\ep_1(x)+c_1(x)\sigma^2(\phi^\ep_1)^2(x)\\[.1mm]
&=&\sigma\lm_p(\ud{\cal G}_1^\ep)\phi^\ep_1(x)-a_1(x)[\mathsf{H}(x)+\sigma\phi^\ep_1(x)]+\sigma [g_1(x)-\ud g_1^\ep(x)]\phi^\ep_1(x)+c_1(x)\sigma^2(\phi^\ep_1)^2(x)\\[.1mm]
&=&\sigma\phi^\ep_1(x)[\lm_p(\ud{\cal G}_1^\ep)+g_1(x)-\ud g_1^\ep(x)+\sigma c_1(x)\phi^\ep_1(x)]-a_1(x)[\mathsf{H}(x)+\sigma\phi^\ep_1(x)]\\[.1mm]
 &<&0,\;\;\; x\in\boo.
 \eess
Likewise,
 \bess
 d_2\!\int_\oo\! J_2(x,y)\ol{\mathsf{V}}_i(y)\dy+b_{22}^\sigma(x)\ol{\mathsf{V}}_i
 +\ell_2(x)\big(\mathsf{V}(x)+\sigma\phi^\ep_2(x)-\ol{\mathsf{V}}_i\big)^+\ol{\mathsf{H}}_i<0,\;\;\; x\in\boo.\vspace{-1mm}
\eess

{\it Step 2: The existence of positive solutions of \qq{4.16}}. In view of $\lm({\mathscr B})>0$, we can find a $0<\sigma_0\ll 1$ such that $\lm({\mathscr B}_\sigma)>0$ for all $|\sigma|\le\sigma_0$ by the continuity. Set
 \bess
 f_1(x,\mathsf{H}_i,\mathsf{V}_i)&=&d^*_1(x) H_i+b_{11}^\sigma(x)\mathsf{H}_i
+\ell_1(x)\big(\mathsf{H}(x)+\sigma\phi^\ep_1(x)-\mathsf{H}_i\big)^+ \mathsf{V}_i,\\[0.1mm]
f_2(x,\mathsf{H}_i,\mathsf{V}_i)&=&d^*_2(x) H_i+b_{22}^\sigma(x)\mathsf{V}_i
+\ell_2(x)\big(\mathsf{V}(x)+\sigma\phi^\ep_2(x)-\mathsf{V}_i\big)^+ \mathsf{H}_i.
 \eess
It is easy to see that $f(x,\mathsf{H}_i,\mathsf{V}_i)=
(f_1(x,\mathsf{H}_i,\mathsf{V}_i),f_2(x,\mathsf{H}_i,\mathsf{V}_i))$ satisfies the conditions {\bf(H1)}--{\bf(H3)} when $0\le\mathsf{H}_i\leq\mathsf{H}+\sigma\phi^\ep_1$ and $0\le\mathsf{V}_i\le\mathsf{V}+\sigma\phi^\ep_2$. Noticing that $B(\tilde x)$ is irreducible. By repeating the arguments to those in the proofs of Theorem \ref{thB}(1) and Theorem \ref{th3.7}, locating between $(0,0)$ and $(\mathsf{H}+\sigma\phi^\ep_1, \mathsf{V}+\sigma\phi^\ep_2)$, the problem \qq{4.16} has a positive solution $(\widehat{\mathsf{H}}_i^\sigma, \widehat{\mathsf{V}}_i^\sigma)\in [C(\boo)]^2$.
Furthermore, it follows from Corollary \ref{c3.1} that any positive solution $(\widetilde{\mathsf{H}}_i^\sigma, \widetilde{\mathsf{V}}_i^\sigma)$ of \qq{4.16} locating between $(0,0)$ and $(\mathsf{H}+\sigma\phi^\ep_1, \mathsf{V}+\sigma\phi^\ep_2)$ satisfies
 \[\widehat{\mathsf{H}}_i^\sigma<\mathsf{H}+\sigma\phi^\ep_1,\;\;\;
 \widehat{\mathsf{V}}_i^\sigma<\mathsf{V}+\sigma\phi^\ep_2\;\;\;\mbox{in}\;\;\boo\]
since $(\mathsf{H}+\sigma\phi^\ep_1, \mathsf{V}+\sigma\phi^\ep_2)$ is a strict upper solution of \qq{4.16}. Therefore, the positive solution of \qq{4.16}, locating between $(0,0)$ and $(\mathsf{H}+\sigma\phi^\ep_1, \mathsf{V}+\sigma\phi^\ep_2)$, is unique (Theorem \ref{th3.3}), which is exactly $\big(\widehat{\mathsf{H}}_i^\sigma, \widehat{\mathsf{V}}_i^\sigma\big)$.  Certainly, $(\widehat{\mathsf{H}}_i^\sigma, \widehat{\mathsf{V}}_i^\sigma)$ satisfies \qq{4.17}.

{\it Step 3: The uniqueness of continuous positive solution}. Until now, we already established the uniqueness  of positive solution of \qq{4.16} located between $(0,0)$ and $(\mathsf{H}+\sigma\phi^\ep_1, \mathsf{V}+\sigma\phi^\ep_2)$. The global uniqueness of continuous positive solution can also be obtained by repeating the arguments used in Theorem \ref{thB}(1). For reader's convenience, we provide the details here.

Let $(\widetilde{\mathsf{H}}_i^\sigma, \widetilde{\mathsf{V}}_i^\sigma)$ be another continuous positive solution of \qq{4.16}. We can find a constant $0<q<1$ such that $q(\widetilde{\mathsf{H}}_i^\sigma, \widetilde{\mathsf{V}}_i^\sigma)\le(\widehat{\mathsf{H}}_i^\sigma, \widehat{\mathsf{V}}_i^\sigma)$ in $\boo$. Set
\[\bar q=\sup\big\{0<q\le 1: q(\widetilde{\mathsf{H}}_i^\sigma, \widetilde{\mathsf{V}}_i^\sigma)\le(\widehat{\mathsf{H}}_i^\sigma, \widehat{\mathsf{V}}_i^\sigma) \;{\rm ~ in  ~ } \boo\big\}.\]
Then $\bar q$ is well defined and $0<\bar q\le1$, and $\bar q(\widetilde{\mathsf{H}}_i^\sigma, \widetilde{\mathsf{V}}_i^\sigma)\le\big(\widehat{\mathsf{H}}_i^\sigma, \widehat{\mathsf{V}}_i^\sigma)$ in $\boo$. We claim $\bar q=1$. Assume on the contrary that $\bar q<1$. Set
 \[\mathsf{U}(x)=\widehat{\mathsf{H}}_i^\sigma(x)-\bar q\widetilde{\mathsf{H}}_i^\sigma(x),\;\;\;
 \mathsf{Z}(x)=\widehat{\mathsf{V}}_i^\sigma(x)-\bar q\widetilde{\mathsf{V}}_i^\sigma(x).\]
Then $\mathsf{U}, \mathsf{Z}\ge 0$ in $\boo$, and so
$\bar q\widetilde{\mathsf{H}}_i^\sigma\le\widehat{\mathsf{H}}_i^\sigma
<\mathsf{H}+\sigma\phi^\ep_1$ and $\bar q\widetilde{\mathsf{V}}_i^\sigma\le\widehat{\mathsf{V}}_i^\sigma
<\mathsf{V}+\sigma\phi^\ep_2$ in $\boo$. Using these estimates we easily see that the following hold:
  \bess
&(\mathsf{H}+\sigma\phi^\ep_1-\widehat{\mathsf{H}}_i^\sigma)^+
=\mathsf{H}+\sigma\phi^\ep_1-\widehat{\mathsf{H}}_i^\sigma,&\\
&(\mathsf{H}+\sigma\phi^\ep_1-\widetilde{\mathsf{H}}_i^\sigma)^+
<\mathsf{H}+\sigma\phi^\ep_1-\bar q\widetilde{\mathsf{H}}_i^\sigma,&\\
& (\mathsf{H}+\sigma\phi^\ep_1-\widehat{\mathsf{H}}_i^\sigma)^+
-(\mathsf{H}+\sigma\phi^\ep_1-\widetilde{\mathsf{H}}_i^\sigma)^+
>\bar q\widetilde{\mathsf{H}}_i^\sigma-\widehat{\mathsf{H}}_i^\sigma,&\\
&(\mathsf{V}+\sigma\phi^\ep_2-\widehat{\mathsf{V}}_i^\sigma)^+
=\mathsf{V}+\sigma\phi^\ep_2-\widehat{\mathsf{V}}_i^\sigma,&\\ &(\mathsf{V}+\sigma\phi^\ep_2-\widetilde{\mathsf{V}}_i^\sigma)^+
<\mathsf{V}+\sigma\phi^\ep_2-\bar q\widetilde{\mathsf{V}}_i^\sigma,&\\
&(\mathsf{V}+\sigma\phi^\ep_2-\widehat{\mathsf{V}}_i^\sigma)^+
-(\mathsf{V}+\sigma\phi^\ep_2-\widetilde{\mathsf{V}}_i^\sigma)^+
>\bar q\widetilde{\mathsf{V}}_i^\sigma-\widehat{\mathsf{V}}_i^\sigma.&
\eess

Denote
  \bess
\alpha_1(x)&=&\mu_1(x)+c_1(x)[\mathsf{H}(x)-\sigma\phi^\ep_1(x)]
+\ell_1(x)\widetilde{\mathsf{V}}_i^\sigma(x),\\
\alpha_2(x)&=&\mu_2(x)+c_2(x)[\mathsf{V}(x)-\sigma\phi^\ep_2(x)]
+\ell_2(x)\widetilde{\mathsf{H}}_i^\sigma(x).
  \eess
After careful calculation we can show that $(\mathsf{U}, \mathsf{Z})$ satisfies
\bes\left\{\!\!\begin{aligned}
	&-d_1\!\int_\oo\!J_1(x,y)\mathsf{U}(y)\dy+d_1^*(x)\mathsf{U}
	+\alpha_1(x)\mathsf{U}>\ell_1(x)\big[\mathsf{H}(x)+\sigma\phi^\ep_1(x)
	-\widehat{\mathsf{H}}_i^\sigma(x)\big]\mathsf{Z}\ge 0,&&\!\!x\in\boo,\\[0.5mm]
	&-d_2\!\int_\oo\!J_2(x,y)\mathsf{Z}(y)\dy+d_2^*(x)\mathsf{Z}
	+\alpha_2(x)\mathsf{Z}>\ell_2(x)\big[\mathsf{V}(x)+\sigma\phi^\ep_2(x)
	-\widehat{\mathsf{V}}_i^\sigma(x)\big]\mathsf{U}\ge 0,&&\!\!x\in\boo.
\end{aligned}\right.\qquad\lbl{4.19}\ees
It follows that $\mathsf{U},\mathsf{Z}>0$ in $\boo$ by the maximum principle as $\mathsf{U},\mathsf{Z}\ge 0$ in $\boo$. Then there exists $0<\varrho<1-\bar q$ such that $(\mathsf{U}, \mathsf{Z})\ge\varrho(\widetilde{\mathsf{H}}_i^\sigma, \widetilde{\mathsf{V}}_i^\sigma)$, i.e., $(\bar q+\varrho)(\widetilde{\mathsf{H}}_i^\sigma, \widetilde{\mathsf{V}}_i^\sigma)\le(\widehat{\mathsf{H}}_i^\sigma, \widehat{\mathsf{V}}_i^\sigma)$ in $\boo$. This contradicts the definition of $\bar q$. Hence $\bar q=1$, i.e.,
  \bes
  \big(\widetilde{\mathsf{H}}_i^\sigma(x),\, \widetilde{\mathsf{V}}_i^\sigma(x)\big)\le\big(\widehat{\mathsf{H}}_i^\sigma(x),\, \widehat{\mathsf{V}}_i^\sigma(x)\big),\;\;\;x\in\boo.
  \lbl{4.20}\ees
Certainly, $\widetilde{\mathsf{H}}_i^\sigma<\mathsf{H}+\sigma\phi^\ep_1$, $\widetilde{\mathsf{V}}_i^\sigma<\mathsf{V}+\sigma\phi$, and then
 \bess
(\mathsf{H}+\sigma\phi^\ep_1-\widetilde{\mathsf{H}}_i^\sigma)^+
=\mathsf{H}+\sigma\phi^\ep_1-\widetilde{\mathsf{H}}_i^\sigma,\;\;\;
(\mathsf{V}+\sigma\phi^\ep_2-\widetilde{\mathsf{V}}_i^\sigma)^+
=\mathsf{V}+\sigma\phi^\ep_2-\widetilde{\mathsf{V}}_i^\sigma.\vspace{-2mm}\eess

On the other hand, we can find $k >1$ such that
$k(\widetilde{\mathsf{H}}_i^\sigma, \widetilde{\mathsf{V}}_i^\sigma)\ge(\widehat{\mathsf{H}}_i^\sigma, \widehat{\mathsf{V}}_i^\sigma)$ in $\boo$. Set
\[\ud k=\inf\big\{k>1: k(\widetilde{\mathsf{H}}_i^\sigma, \widetilde{\mathsf{V}}_i^\sigma)\ge(\widehat{\mathsf{H}}_i^\sigma, \widehat{\mathsf{V}}_i^\sigma)\;{\rm ~ in  ~ }\, \boo\big\}.\]
Then $\ud k$ is well defined, $\ud k\ge1$ and $\ud k\big(\widetilde{\mathsf{H}}_i^\sigma, \widetilde{\mathsf{V}}_i^\sigma\big)\ge\big(\widehat{\mathsf{H}}_i^\sigma, \widehat{\mathsf{V}}_i^\sigma\big)$ in $\boo$. If $\ud k>1$, then $\mathsf{P}:=\ud k\widetilde{\mathsf{H}}_i^\sigma-\widehat{\mathsf{H}}_i^\sigma\ge 0$ and $\mathsf{Q}:=\ud k\widetilde{\mathsf{V}}_i^\sigma-\widehat{\mathsf{V}}_i^\sigma\ge 0$ in $\boo$. Similarly to the above, we can verify that $\big(\mathsf{P}, \mathsf{Q}\big)$ satisfies a system of differential inequalities similar to \qq{4.19} and derive $\mathsf{P}, \mathsf{Q}>0$ in $\boo$ by the maximum principle, and there exists $0<r<\ud k-1$ such that $\big(\mathsf{P}, \mathsf{Q}\big)\ge r\big(\widetilde{\mathsf{H}}_i^\sigma, \widetilde{\mathsf{V}}_i^\sigma\big)$, i.e., $(\ud k-r)\big(\widetilde{\mathsf{H}}_i^\sigma, \widetilde{\mathsf{V}}_i^\sigma\big)\ge\big(\widehat{\mathsf{H}}_i^\sigma, \widehat{\mathsf{V}}_i^\sigma\big)$ in $\boo$. This contradicts the definition of $\ud k$. Hence $\ud k=1$ and $(\widetilde{\mathsf{H}}_i^\sigma, \widetilde{\mathsf{V}}_i^\sigma)\ge(\widehat{\mathsf{H}}_i^\sigma, \widehat{\mathsf{V}}_i^\sigma)$ in $\boo$. This combined with \qq{4.20} gives the uniqueness.

(2) We only prove the first part as the second part can be done by the same way.
Define
 \bess
 f_1(x,u_1, u_2)&=&\ell_1(x)(\mathsf{H}(x)+\sigma\phi^\ep_1(x)-u_1)u_2,\\
 f_2(x,u_1,u_2)&=&\ell_2(x)(\mathsf{V}(x)+\sigma\phi^\ep_2(x)-u_2)u_1,\eess
and let $(\mathsf{H}_i^\sigma, \mathsf{V}_i^\sigma)$ be a  bounded nonnegative solution of \qq{4.17}. By the following Lemma \ref{l4.2}, for any given $x\in\boo$, the sequel algebraic system \qq{4.21} with $(U, Z)=(\mathsf{H}_i^\sigma, \mathsf{V}_i^\sigma)$ has at most one positive solution $(v_1, v_2)$.
For any given $x\in\boo$. As $b_{11}^\sigma(x)<0, b_{22}^\sigma(x)<0$ in $\boo$, we easily see that algebraic equations
  \bess
  d_1\dd\int_\oo J_1(x,y)\mathsf{H}_i^\sigma(y)\dy+b_{11}^\sigma(x)w_1+f_1(x,w_1,0)=0,
  \eess
 and
  \bess
  d_2\dd\int_\oo J_2(x,y)\mathsf{V}_i^\sigma(y)\dy+b_{22}^\sigma(x)w_2+f_2(x,0,w_2)=0
 \eess
have at most one positive solutions $w_1$ and $w_2$, respectively.

Making use of Theorem \ref{tA.2}, we conclude that
$(\mathsf{H}_i^\sigma, \mathsf{V}_i^\sigma)$ is continuous in $\boo$, and then is unique by the conclusion (1). The proof is complete.\vspace{-2mm} \end{proof}

\begin{proof}[Proof of Theorem \ref{t5.1}] (1) Taking $\sigma=0$ in Lemma \ref{l4.1}, the existence and uniqueness of positive solutions are proved.

(2)\, Assume $\lm({\mathscr B})\le 0$ and prove that \qq{4.11} has no positive solution. Assume on the contrary that \qq{4.11} has a positive solution $(\mathsf{H}_i, \mathsf{V}_i)$. Then $(\mathsf{H}_i, \mathsf{V}_i)\in[C(\boo)]^2$ by Lemma \ref{l4.1}(2). Set $\rho_1=\min_{\boo}\ell_1(x)\mathsf{V}_i(x)$ and $ \rho_2=\min_{\boo}\ell_2(x)\mathsf{H}_i(x)$. Then $\rho_1, \rho_2>0$. By the direct calculations we have that, for $\rho=\min\{\rho_1, \rho_2\}$,
 \bess
 0&=&d_1\!\int_\oo\! J_1(x,y)\mathsf{H}_i(y)\dy+[g_1(x)-a_1(x)-c_1(x)
 \mathsf{H}(x)]\mathsf{H}_i+\ell_1(x)\mathsf{H}(x)\mathsf{V}_i
 -\ell_1(x)\mathsf{H}_i\mathsf{V}_i\\
&\le&d_1\!\int_\oo\! J_1(x,y)\mathsf{H}_i(y)\dy+b_{11}(x)\mathsf{H}_i+b_{12}(x)
\mathsf{V}_i-\rho_1\mathsf{H}_i\\
&\le&d_1\!\int_\oo\! J_1(x,y)\mathsf{H}_i(y)\dy+\bar b_{11}^\ep(x)\mathsf{H}_i-\rho\mathsf{H}_i+\bar b_{12}^\ep(x)\mathsf{V}_i,\;\;\;x\in\boo,\\
0&=&d_2\!\int_\oo\! J_2(x,y)\mathsf{V}_i(y)\dy+[g_2(x)-a_2(x)-c_2(x)
\mathsf{V}(x)]\mathsf{V}_i+\ell_2(x)(\mathsf{V}(x)-\mathsf{V}_i)\mathsf{H}_i\\
&\le&d_1\!\int_\oo\! J_2(x,y)\mathsf{V}_i(y)\dy+\bar b_{22}^\ep(x)\mathsf{V}_i-\rho\mathsf{V}_i+\bar b_{21}^\ep(x)\mathsf{H}_i,\;\;\;x\in\boo.
 \eess
It follows that $\lm_p(\ol{\mathscr B}^\ep)\ge\rho$ for all $0<\ep\ll 1$, so  $\lm({\mathscr B})\ge\rho>0$. This is a contradiction.
\end{proof}\vspace{-2mm}

\begin{lem}\lbl{l4.2} Let $\mathsf{H}, \mathsf{V}, \sigma, \phi^\ep_k, b_{kl}^\sigma$ and $f_1, f_2$ be given in the above; and $U, Z$ be nonnegative functions in $\boo$. Then, for any given $x\in\boo$, the algebraic system
  \bes\left\{\begin{aligned}
&d_1\!\int_\oo\! J_1(x,y)U(y)\dy+b_{11}^\sigma(x)v_1+f_1(x,v_1,v_2)=0,\\
&d_2\!\int_\oo\! J_2(x,y)Z(y)\dy+b_{22}^\sigma(x)v_2+f_2(x,v_1,v_2)=0
 \end{aligned}\rr.\lbl{4.21}\ees
has at most one positive solution $v=(v_1, v_2)$.\vspace{-2mm}
\end{lem}

\begin{proof}  By the expressions of $b_{11}$ and $b_{22}$ we have $b_{11}^\sigma, b_{22}^\sigma<0$ in $\boo$ since $|\sigma|\ll 1$. Denote
 \bess
 h_1&=&d_1\!\int_\oo\! J_1(x,y)U(y)\dy,\;\;
 h_2=d_2\!\int_\oo\! J_2(x,y)Z(y)\dy,\\[1mm]
 r_1&=&-b_{11}^\sigma(x),\;\;r_2=-b_{22}^\sigma(x),\;\;
 p_1=\ell_1(x),\;\;p_2=\ell_2(x),\\
 q_1&=&\mathsf{H}(x)+\sigma\phi^\ep_1(x),\;\;
 q_2=\mathsf{V}(x)+\sigma\phi^\ep_2(x).
 \eess
Then $h_k\ge 0, r_k,p_k,q_k>0$ for $k=1,2$, and \qq{4.21} can be written as
  \bes\left\{\begin{aligned}
&h_1-r_1v_1+p_1(q_1-v_1\big)v_2=0, \\
&h_2-r_2v_2+p_2(q_2-v_2\big)v_1=0.
 \end{aligned}\rr.\lbl{4.22}\ees
Let $(v_1, v_2)$ be a positive solution of \qq{4.22}.  If $v_1=q_1$, then $v_2=\frac{h_2+p_2q_1q_2}{r_2+p_2q_1}$. If $v_1\not=q_1$, then
 \[v_2=\frac{r_1v_1-h_1}{p_1(q_1-v_1)},\;\;\;
 h_2-r_2\frac{r_1v_1-h_1}{p_1(q_1-v_1)}+p_2
 \left(q_2-\frac{r_1v_1-h_1}{p_1(q_1-v_1)}\rr)v_1=0.\]
By carefully calculations we see that the second equation is equivalent to
 \[p_2(r_1+p_1q_2)v_1^2+(r_1r_2+p_1h_2-p_1p_2q_1q_2-p_2h_1)v_1
 -(r_2h_1+p_1q_1h_2)=0.\]
Since $r_2h_1+p_1q_1h_2\geq 0$, we see that the above equation has at most one positive solution. So \qq{4.22}, and thus \qq{4.21} has at most one positive solution.
\end{proof}\vspace{-1mm}

\subsection{Dynamical properties of \qq{4.2}}\lbl{s3}

In this section we study the stabilities of nonnegative equilibrium solutions.

\begin{theo}\lbl{th4.3} Let $(H_u, H_i, V_u, V_i\big)$ be the unique positive solution of \qq{4.2}. \vspace{-2mm}
\begin{enumerate}[$(1)$]
\item\; In the case of $\lm({\cal G}_1)>0$, $\lm({\cal G}_2)>0$ and $\lm({\mathscr B})>0$, we have
 \bes
 \lim_{t\to+\yy}\big(H_u(x,t), H_i(x,t), V_u(x,t), V_i(x,t)\big)=\big(\mathsf{H}(x)-\mathsf{H}_i(x),\mathsf{H}_i(x), \mathsf{V}(x)-\mathsf{V}_i(x),\mathsf{V}_i(x)\big)\qquad
 \lbl{4.23}\ees
uniformly in $\boo$, where $\mathsf{H}(x)$, $\mathsf{V}(x)$ and $\big(\mathsf{H}_i(x), \mathsf{V}_i(x)\big)$ are the unique positive solutions of $(4.6_1)$, $(4.6_2)$ and \qq{4.11}, respectively, and $\mathsf{H}_i(x)<\mathsf{H}(x), \mathsf{V}_i(x)<\mathsf{V}(x)$.\vvv

\item\; In the case of $\lm({\cal G}_1)>0$, $\lm({\cal G}_2)>0$ and $\lm({\mathscr B})\le 0$, we have
 \bes
\lim_{t\to+\yy}\big(H_u(x,t), H_i(x,t), V_u(x,t), V_i(x,t)\big)=\big(\mathsf{H}(x), 0, \mathsf{V}(x), 0\big)\;\;\mbox{uniformly in}\;\, \boo.
 \quad\lbl{4.24}\ees

\item\; In the case of $\lm({\cal G}_1)\le 0$, $\lm({\cal G}_2)>0$, we have
 \bess
\lim_{t\to+\yy}\big(H_u(x,t), H_i(x,t), V_u(x,t), V_i(x,t)\big)=\big(0, 0, \mathsf{V}(x), 0\big)\;\;\mbox{uniformly in}\;\, \boo.\eess

In the case of $\lm({\cal G}_1)>0$, $\lm({\cal G}_2)\le 0$, we have
 \bess
\lim_{t\to+\yy}\big(H_u(x,t),\, H_i(x,t), V_u(x,t),  V_i(x,t)\big)=\big(\mathsf{H}(x), 0, 0, 0\big)\;\;\mbox{uniformly in}\;\, \boo.\eess

In the case of $\lm({\cal G}_1) \le 0$, $\lm({\cal G}_2) \le0$, we have
  \bess
  \lim_{t\to+\yy}\big(H_u(x,t), H_i(x,t), V_u(x,t), V_i(x,t)\big)=(0, 0, 0, 0)\;\;\mbox{uniformly in}\;\, \boo.\eess
 \end{enumerate}\vspace{-2mm}
  \end{theo}

\begin{proof} Throughout this proof, the matrix $B_\sigma(x)$ is defined by \qq{4.15}. Let $(H_u, H_i, V_u, V_i)$ be the unique solution of \qq{4.2}. Then $H=H_u+H_i$ and $V=V_u+V_i$ satisfy, respectively, \qq{4.3} and \qq{4.4} with $T=\infty$. According to the conditions $\lm({\cal G}_1)>0$, $\lm({\cal G}_2)>0$, we have that, by Theorem \ref{t3.3},
 \bes
 \lim_{t\to+\yy}H(x,t)=\mathsf{H}(x),\;\;\;\lim_{t\to+\yy}V(x,t)=\mathsf{V}(x)
 \;\;\;{\rm uniformly\,\, in}\;\;\boo.\lbl{4.25}\ees

Similar to the above, let $\phi^\ep_k>0$ with $\|\phi^\ep_k\|_{G^\infty(\oo)}=1$ be the unique solution of $(4.13_k)$. For any given $0<\sigma\le\sigma_0$, there exists a $T_\sigma\gg 1$ such that
 \bes\left\{\begin{aligned}
&0<\mathsf{H}(x)-\sigma\phi^\ep_1(x)\le H(x,t)\le \mathsf{H}(x)+\sigma\phi^\ep_1(x),&&\!\!x\in\ol\oo,\; t\geq T_\sigma,\\[1mm]
&0<\mathsf{V}(x)-\sigma\phi^\ep_2(x)\le V(x,t)\le \mathsf{V}(x)+\sigma\phi^\ep_2(x),&&\!\!x\in\ol\oo,\; t\geq T_\sigma.
  \end{aligned}\rr.\lbl{4.26}\ees
Making use of $H_u=H-H_i$, $V_u=V-V_i$ and \qq{4.26}, we see that $(H_i, V_i)$ satisfies
 \bes\left\{\!\begin{aligned}
&H_{it}\le d_1\!\dd\int_\oo\! J_1(x,y)H_i(y,t)\dy-d^*_1(x) H_i+ \ell_1(x)\big(\mathsf{H}(x)+\sigma\phi^\ep_1(x)-H_i\big)^+V_i\\[.1mm]
 &\hspace{16mm}-\big\{\mu_1(x)+c_1(x)[\mathsf{H}(x)-\sigma\phi^\ep_1(x)]
 \big\}H_i\\[.1mm]
&\hspace{6mm}=d_1\!\dd\int_\oo\! J_1(x,y)H_i(y,t)\dy+b_{11}^\sigma(x) H_i+ \ell_1(x)\big(\mathsf{H}(x)+\sigma\phi^\ep_1(x)-H_i\big)^+V_i,
&&\!\!x\in\ol\oo,\; t\geq T_\sigma,\\[.1mm]
&V_{it}\le d_2\!\dd\int_\oo\! J_2(x,y)V_i(y,t)\dy-d^*_2(x)V_i+ \ell_2(x)\big(\mathsf{V}(x)+\sigma\phi^\ep_2(x)-V_i\big)^+H_i\\[.1mm]
 &\hspace{16mm}-\big\{\mu_2(x)+c_2(x)[\mathsf{V}(x)-\sigma\phi^\ep_2(x)]\big\}V_i\\[.1mm]
&\hspace{6mm}=d_2\!\dd\int_\oo\! J_2(x,y)V_i(y,t)\dy+b_{22}^\sigma(x)V_i+ \ell_2(x)\big(\mathsf{V}(x)+\sigma\phi^\ep_2(x)-V_i\big)^+H_i,
&&\!\!x\in\ol\oo,\; t\geq T_\sigma,\\
&H_i\le H\le \mathsf{H}(x)+\sigma\phi^\ep_1(x),\;\;
V_i\le V\le \mathsf{V}(x)+\sigma\phi^\ep_2(x),&&\!\! x\in\boo,\,t=T_\sigma,
 \end{aligned}\rr.\lbl{4.27}\ees
where $b_{11}^\sigma(x)$ and $b_{22}^\sigma(x)$ are defined by \qq{4.14}.

(1)\, Assume that $\lm({\cal G}_1)>0$, $\lm({\cal G}_2)>0$ and $\lm({\mathscr B})>0$.

{\it Step 1}. Since $\lm({\mathscr B})>0$, there exists a $0<\sigma_0\ll 1$ such that $\lm({\mathscr B}_\sigma)>0$ for all $|\sigma|\le\sigma_0$ by the continuity. Let $(U_\sigma,Z_\sigma)$ be the solution of
\bes\left\{\!\begin{aligned}
	& U_{\sigma t}=d_1\!\dd\int_\oo\! J_1(x,y)U_\sigma(y,t)\dy+b_{11}^\sigma(x)U_\sigma+ \ell_1(x)\big(\mathsf{H}(x)+\sigma\phi^\ep_1(x)-U_\sigma\big)^+Z_\sigma,&&\!\!x\in\oo,\; t>0,\\
	&Z_{\sigma t}=d_2\!\dd\int_\oo\! J_2(x,y)Z_\sigma(y,t)\dy+b_{22}^\sigma(x)Z_\sigma+ \ell_2(x)\big(\mathsf{V}(x)+\sigma\phi^\ep_2(x)-Z_\sigma\big)^+U_\sigma,
	&&\!\!x\in\oo,\; t>0,\\
	&(U_\sigma(x,0), Z_\sigma(x,0))=(U_0(x), Z_0(x)),&&\!\!x\in\Omega.
\end{aligned}\rr.\quad\;\;\;\;\lbl{4.28}\ees
with $U_0(x)=\mathsf{H}(x)+\sigma\phi^\ep_1(x)$,  $Z_0(x)=\mathsf{V}(x)+\sigma\phi^\ep_2(x)$. We have shown that $(\mathsf{H}+\sigma\phi^\ep_1, \mathsf{V}+\sigma\phi^\ep_2)$ is a strict upper solution of \qq{4.16} provided $|\sigma|\ll 1$ in Step 1 of the proof of Lemma \ref{l4.1}(1). In view of $\lm({\mathscr B}_\sigma)>0$, by repeating the arguments to those in Theorem \ref{thB}(1), we have
\bes
\lim_{t\to+\yy}(U_\sigma(x,t), Z_\sigma(x,t))=(\widehat{\mathsf{H}}_i^\sigma(x),
\widehat{\mathsf{V}}_i^\sigma(x))\;\;\;\mbox{uniformly in}\;\;\boo.
\lbl{4.29}\ees

On the other hand, by the comparison principle,
 \bes
 (H_i(x,t+T_\sigma),V_i(x,t+T_\sigma))\le(U_\sigma(x,t), Z_\sigma(x,t))\;\;\;\mbox{in}\;\; \boo\times[0,+\yy).
 \lbl{4.30}\ees
This combines with \qq{4.29} gives
\bes
 \limsup_{t\to+\yy}\big(H_i(x,t), V_i(x,t)\big)\leq (\widehat{\mathsf{H}}_i^\sigma(x),
 \widehat{\mathsf{V}}_i^\sigma(x))=(\mathsf{H}_i^\sigma(x), \mathsf{V}_i^\sigma(x)) \;\;\;{\rm uniformly\; in}\; \;\boo.
\lbl{4.31}\ees
Noticing that the positive solution $(\mathsf{H}_i^\sigma, \mathsf{V}_i^\sigma)$ of \qq{4.17} exists and is unique. From the expression \qq{4.14} of $b_{kl}^\sigma$, it is easy to see that $b_{kl}^\sigma$ is increasing in $\sigma$. Then $(\mathsf{H}_i^\sigma, \mathsf{V}_i^\sigma)\ge(\mathsf{H}_i, \mathsf{V}_i)$ and $(\mathsf{H}_i^\sigma, \mathsf{V}_i^\sigma)$ is increasing in $\sigma$ by the comparison principle and the uniqueness. Therefore, $\lim_{\sigma\to0^+}({\mathsf{H}_i^\sigma}, {\mathsf{V}_i^\sigma})=({\mathsf{H}^*_i}, {\mathsf{V}^*_i})$ exists and $({\mathsf{H}^*_i}, {\mathsf{V}^*_i})$ is a positive solution of \qq{4.11}. Take $\sigma=0$ in Lemma \ref{l4.1}(2) to deduce that $(\mathsf{H}^*_i, \mathsf{V}^*_i)=(\mathsf{H}_i, \mathsf{V}_i)$. Hence, $\lim_{\sigma\to0^+}(\mathsf{H}_i^\sigma, \mathsf{V}_i^\sigma)=(\mathsf{H}_i, \mathsf{V}_i)$ uniformly in $\boo$. This combined with \qq{4.31} yields
\bes
 \limsup_{t\to+\yy}(H_i(x,t), V_i(x,t))\leq(\mathsf{H}_i(x), \mathsf{V}_i(x)) \;\;\;{\rm uniformly\; in}\;\;\boo.
\lbl{4.32}\ees

{\it Step 2}. Noticing that $\mathsf{H}_i<\mathsf{H}$ and $\mathsf{V}_i<\mathsf{V}$ in $\boo$. There exists $0<\tau_0<\sigma_0$ such that\vspace{-1mm}
 \bess
 \mathsf{H}_i(x)<\mathsf{H}(x)-2\tau\phi^\ep_1(x),\; \;\; \mathsf{V}_i(x)<\mathsf{V}(x)-2\tau\phi^\ep_2(x),\;\;\;x\in\boo\qquad
   \vspace{-1mm}\eess
for all $0<\tau<\tau_0$. This combines with \qq{4.32} yields that there exists $\rrr T_\tau^*\gg 1$ such that
  \bes\begin{cases}
 H_i(x,t)<\mathsf{H}_i(x)+\tau \phi^\ep_1(x)<\mathsf{H}(x)-\tau\phi^\ep_1(x),
 \;\;x\in\boo,\; t>T_\tau^*,\\
 V_i(x,t)<\mathsf{V}_i(x)+\tau \phi^\ep_2(x)<\mathsf{V}(x)-\tau\phi^\ep_2(x),
 \;\;x\in\boo,\; t>T_\tau^*.\end{cases}\lbl{4.33} \ees

{\it Step 3}. For such $\tau$ determined in Step 2. Thanks to \qq{4.25}, there exists $\hat T_\tau\gg 1$ such that
 \bess\begin{aligned}
&0<\mathsf{H}(x)-\tau\phi^\ep_1(x)\le H(x,t)\le \mathsf{H}(x)+\tau\phi^\ep_1(x),&&\!\!x\in\ol\oo,\; t\geq\hat T_\tau,\\
&0<\mathsf{V}(x)-\tau\phi^\ep_2(x)\le V(x,t)\le \mathsf{V}(x)+\tau\phi^\ep_2(x),&&\!\!x\in\ol\oo,\; t\geq\hat T_\tau.
  \end{aligned}\eess
Noticing that $H_u=H-H_i$ and $V_u=V-V_i$. Let $T_\tau=T_\tau^*+\hat T_\tau$. Take advantage of \qq{4.33}, it follows that $(H_i, V_i)$ satisfies
 \bess
H_{it}&\ge& d_1\!\dd\int_\oo\! J_1(x,y)H_i(y,t)\dy-d^*_1(x) H_i+ \ell_1(x)(\mathsf{H}(x)-\tau\phi^\ep_1(x)-H_i)^+V_i\\[1mm]
 &&\hspace{3mm}-\big\{\mu_1(x)+c_1(x)[\mathsf{H}(x)
 +\tau\phi^\ep_1(x)]\big\}H_i\\[1mm]
&=&d_1\!\dd\int_\oo\! J_1(x,y)H_i(y,t)\dy+b_{11}^{-\tau}(x) H_i+ \ell_1(x)(\mathsf{H}(x)-\tau\phi^\ep_1(x)-H_i)^+V_i,\;\;\;x\in\boo,\;
t\ge T_\tau,\\[1mm]
V_{it}&\ge& d_2\!\dd\int_\oo\! J_2(x,y)V_i(y,t)\dy-d^*_2(x)V_i+ \ell_2(x)(\mathsf{V}(x)-\tau\phi^\ep_2(x)-V_i)^+H_i\\[1mm]
&&\hspace{3mm}-\big\{\mu_2(x)+c_2(x)[\mathsf{V}(x)+\tau\phi^\ep_2(x)]\big\}V_i\\[1mm]
&=&d_2\!\dd\int_\oo\! J_2(x,y)V_i(y,t)\dy+b_{22}^{-\tau}(x)V_i+ \ell_2(x)(\mathsf{V}(x)-\tau\phi^\ep_2(x)-V_i)^+H_i,\;\;\;x\in\boo,\;
t\ge T_\tau.
 \eess
Let $(U_{-\tau},Z_{-\tau})$ be the solution of \qq{4.28} with $\sigma=-\tau$ and $(U_0(x),Z_0(x))=\big(H_i(x,T_\tau), V_i(x,T_\tau)\big)$. Then we have
 \bes
 (H_i(x,t+T_\tau),V_i(x,t+T_\tau))\ge(U_{-\tau}(x,t), Z_{-\tau}(x,t))\;\;\;\mbox{in}\;\; \boo\times[0,+\yy)
 \lbl{4.34}\ees
by the comparison principle. As we have known that $\lm({\mathscr B}_\sigma)>0$, the corresponding limit \qq{4.29} holds by repeating the arguments to those in Theorem \ref{thB}(1). This fact combines with \qq{4.34} yields
 \bes
 \liminf_{t\to+\yy}\big(H_i(x,t), V_i(x,t)\big)\geq (\widehat{\mathsf{H}}_i^{-\tau}(x),
 \widehat{\mathsf{V}}_i^{-\tau}(x))=(\mathsf{H}_i^{-\tau}(x), \mathsf{V}_i^{-\tau}(x)) \;\;\;{\rm uniformly\; in}\; \;\boo.\vspace{-1mm}
 \lbl{4.35}\ees
Similar to the arguments in Step 1, we have $\lim_{\tau\to0^+}(\mathsf{H}_i^{-\tau}, \mathsf{V}_i^{-\tau})=(\mathsf{H}_i, \mathsf{V}_i)$ uniformly in $\boo$. Hence, by using of \qq{4.35},
  \bess
 \liminf_{t\to+\yy}\big(H_i(x,t), V_i(x,t)\big)\geq \big(\mathsf{H}_i(x), \mathsf{V}_i(x)\big) \;\;\;{\rm uniformly\; in}\; \;\boo.
\eess
This, together with \qq{4.32}, yields that
 \[\lim_{t\to+\yy}\big(H_i(x,t),V_i(x,t)\big)=\big(\mathsf{H}_i(x),\mathsf{V}_i(x)\big) \;\;\;{\rm uniformly\; in}\; \;\boo.\]
Using \qq{4.25} we conclude that \qq{4.23} holds. The proof of conclusion (1) is complete.

(2)\, Assume that $\lm({\cal G}_1)>0$, $\lm({\cal G}_2)>0$ and $\lm({\mathscr B})\le 0$. In Step 2 of the proof of Lemma \ref{l4.1}(1) we have shown that $(\mathsf{H}+\sigma\phi^\ep_1, \mathsf{V}+\sigma\phi^\ep_2)$ is a strict upper solution of \qq{4.16}.

{\it Case 1: $\lm({\mathscr B})<0$}. By the continuity, there exists a $0<\sigma_0\ll 1$ such that $\lm({\mathscr B}_\sigma)<0$ for all $0\le\sigma\le\sigma_0$. Similar to the proof of Theorem \ref{t5.1}(ii) we can show that \qq{4.17} has no positive solution. This implies that \qq{4.16} has no continuous positive solution.
In fact, if $(\widehat{{\mathsf H}}_i^\ep, \widehat{{\mathsf V}}_i^\ep)\in [C(\boo)]^2$ is a positive solution of \qq{4.16}, then $(\widehat{{\mathsf H}}_i^\ep, \widehat{{\mathsf V}}_i^\ep)<({\mathsf H}+\ep\phi_1, {\mathsf V}+\ep\phi_2)$ by the comparison principle since $({\mathsf H}+\ep\phi_1, {\mathsf V}+\ep\phi_2)$ is a strict upper solution of \qq{4.16}. Therefore, $(\widehat{{\mathsf H}}_i^\ep, \widehat{{\mathsf V}}_i^\ep)$ is also a positive solution of \qq{4.17}, which is a contradiction.

Noticing that \qq{4.26} and \qq{4.27} always hold. Let $(U_\sigma,Z_\sigma)$ be the unique positive solution of \qq{4.28} with $(U_0(x), Z_0(x))=\big(\mathsf{H}(x)+\sigma\phi^\ep_1(x), \mathsf{V}(x)+\sigma\phi^\ep_2(x)\big)$. Then \qq{4.30} holds. Moreover,
 \[\lim_{t\to+\yy}(U_\sigma(x,t),Z_\sigma(x,t)\big)=(0,0) \;\;\text{ uniformly in } \; \overline{\Omega}\]
by Theorem \ref{thB}(2). This combines with \qq{4.30} derives
 \bes
 \lim_{t\to+\yy}(H_i(x,t), V_i(x,t))=(0, 0)\;\;\;{\rm uniformly\; in}\; \;\boo.
 \lbl{4.36}\ees
Therefore, \qq{4.24} holds by \qq{4.25}.

{\it Case 2: $\lm({\mathscr B})=0$}. There exists a $0<\sigma_0\ll 1$ such that $\lm({\mathscr B}_\sigma)>0$ for all $0<\sigma\le\sigma_0$ by the continuity. Thanks to Theorem \ref{thB}(1),  the problem \qq{4.17} has a unique continuous positive solution $(\mathsf{H}_i^\sigma, \mathsf{V}_i^\sigma)$ and $(\mathsf{H}_i^\sigma, \mathsf{V}_i^\sigma)<(\mathsf{H}+\sigma\phi^\ep_1, \mathsf{V}+\sigma\phi^\ep_2)$. Similar to the proof of \qq{4.32} we can show that
\bes
 \limsup_{t\to+\yy}\big(H_i(x,t), V_i(x,t)\big)\leq\big(\mathsf{H}_i^\sigma(x), \mathsf{V}_i^\sigma(x)\big) \;\;\;{\rm uniformly\; in}\; \;\boo.
\lbl{4.37}\ees

Since $b_{kl}^\sigma(x)$ is increasing in $\sigma$, it is easy to see that $(\mathsf{H}_i^{\sigma'}, \mathsf{V}_i^{\sigma'})$ is an strict upper solution of \qq{4.17} when $\sigma'>\sigma$. So, $(\mathsf{H}_i^\sigma, \mathsf{V}_i^\sigma)$ is strict increasing in $\sigma>0$, and then the limit
 \bes\lim_{\sigma\to 0^+}\big(\mathsf{H}_i^\sigma(x), \mathsf{V}_i^\sigma(x)\big)=\big(\mathsf{H}_i^*(x), \mathsf{V}_i^*(x)\big)\lbl{4.38}\ees
exists and is a nonnegative solution of \qq{4.11}, and $(\mathsf{H}_i^*, \mathsf{V}_i^*)\le (\mathsf{H}, \mathsf{V})$. By sequentially using Lemma \ref{l4.2} and Theorem \ref{tA.2}, we can derive that
$(\mathsf{H}_i^*, \mathsf{V}_i^*)$ is continuous in $\boo$. According to Theorem \ref{th3.1}, either $(\mathsf{H}_i^*,\,\mathsf{V}_i^*)\equiv 0$ or $(\mathsf{H}_i^*,\,\mathsf{V}_i^*)\gg 0$ in $\boo$. Since the problem \qq{4.11} has no continuous positive solution, there holds $(\mathsf{H}_i^*,\,\mathsf{V}_i^*)\equiv 0$ in $\boo$. This combines with \qq{4.37} and \qq{4.38} implies that \qq{4.36} holds. The proof is complete.
 \end{proof}\vspace{-1mm}

\section{Discussions}\label{sec:discussion}\renewcommand {\baselinestretch}{1.2}

\subsection{General discussion}
In this paper, we studied systems of nonlocal operators with cooperative and irreducible structure. By constructing the monotonic upper and lower control systems, we obtained the approximation and characterization of the generalized principal eigenvalue and give two applications. Through these two examples, we saw that the generalized principal eigenvalue plays the same role as the usual principal eigenvalue.

In order to obtain the existence of principal eigenvalue of operator ${\mathscr B}$, appropriate additional conditions need to be attached to $B(x)$, $d_i$ or $J_i(x,y)$ (see, for example, \cite{BZ07, Cov, SX15, B-JFA16, LCWdcds17, SWZ23, Zhang24}). For instance, one of the following conditions:\vspace{-2mm}
 \begin{enumerate}
\item[{\rm (i)}] There exists an open set $\dd\Omega_0 \subset \Omega$ such that  $[\max_{\ol\oo}s(B(x))-s(B(x))]^{-1} \not \in L^1(\Omega_0)$.\vskip 3pt
\item[{\rm(ii)}] $d_i$ are suitable large.\vskip 3pt
\item[{\rm (iii)}] $B(x)$ is a suitable small perturbation of a constant matrix.\vskip 3pt
\item[{\rm (iv)}] $J_i(x,y)=\frac{1}{\delta^N}\tilde{J}_i(\frac{x-y}{\delta})$ with  ${\rm supp}(\tilde{J}_i)=B(0,1)=\{z \in \mathbb{R}^N: \Vert z\Vert <1\}$ and $\delta$ is suitable small.
	\end{enumerate}\vspace{-2mm}

For the given systems of nonlocal operators with a cooperative and irreducible structure, these conditions may be satisfied. However, for the system \(\qq{4.11}\), we only know that \(\mathsf{H}(x)\) and \(\mathsf{V}(x)\) are positive solutions of \((4.6_k)\) with \(k=1,2\), respectively, without understanding their further properties. Therefore, we cannot make additional assumptions about \(\mathsf{H}(x)\) and \(\mathsf{V}(x)\) to ensure that one of these additional conditions holds. This also indicates that the method proposed in this article is highly effective.
We suspect that the present method also applies to periodic problems (time or space periods), and we will leave these topics for future work.

\subsection{Partially degenerate case}

In this subsection, we demonstrate that our present method can also be used to deal with the partially degenerate case, that is, there exists $1\le n_1<n$ such that $d_1, \cdots, d_{n_1}$ are positive and $d_{n_1+1},\cdots,d_n$ are zero. Write ${\cal S}_n=\{1,\cdots,n_1\}$ and ${\cal S}_d=\sss\setminus{\cal S}_n$. We introduce following assumptions {\bf(B$'$)}, {\bf(H2$'$)} and {\bf(H4$'$)}:\vspace{-2mm}
\begin{itemize}
	\item[{\bf(B$'$)}] $b_{ik}\in C(\boo)$ for all $i,k\in\sss$ and $B(x)=(b_{ik}(x))_{n \times n}$ is a cooperative and irreducible matrix  for all $x\in\boo$.\vvv
	\item[{\bf(H2$'$)}]  $\kk(\partial_{u_k}f_i(x,u)\rr)_{n\times n}$ is irreducible for all $x\in\boo$ and $u \geq 0$.\vvv
	\item[{\bf(H4$'$)}]For any $ i\in{\cal S}_d$, there exists $l\in {\cal S}_n$ such that $\partial_{u_l}f_i(x,u)$ is positive for all $x\in\boo$ and $u\ge 0$.
\end{itemize}\vspace{-2mm}

Thanks to \cite[Theorem A]{Zhang24}, the conclusion of Lemma \ref{le2.1}, and hence, Theorem \ref{thA} is valid when the assumption {\bf(B)} is replaced by {\bf(B$'$)} for the partially degenerate case.

We also have the following strong maximum principle for the partially degenerate case.
\vspace{-2mm}
\begin{theo}{\rm(Strong maximum principle)} Let $p_{ik}\in L^\yy(\oo)$ and $(p_{ik}(x))_{n\times n}$ be cooperative, i.e. $p_{ik}(x)\ge 0$ in $\oo$ when $i\not=k$. Assume that $U\in[L^\yy(\oo)]^n$, $U\ge 0$ in $\boo$ and satisfies
	\bess\begin{cases}
		d_i\dd\int_\oo J_i(x,y)U_i(y)\dy-d^*_i(x)U_i+\sum_{k=1}^np_{ik}(x)U_k\le 0,\;\;\; x\in\ol\oo, \; i\in{\cal S}_n,\\[4mm]
		\hfill \dd\sum_{k=1}^np_{ik}(x)U_k\le 0,\;\;\; x\in\ol\oo, \; i\in{\cal S}_d.
	\end{cases}	\eess
	If $(p_{ik}(x))_{n\times n}$ is irreducible for all $x \in \boo$, then either $U_i\equiv 0$ in $\boo$ for all $i\in\sss$, or $U_i>0$ in $\boo$ for all $i\in{\cal S}$. If, in addition, $U_i>0$ in $\boo$ for all $i\in{\cal S}$, then $\inf_{\oo}U_i>0$ for all $i\in{\cal S}$ when one of the following two statements is valid:\vspace{-2mm}
	\begin{itemize}
		\item[\rm (a)] for any $ i\in {\cal S}_d$ there exists $l \in {\cal S}_n$ such that $\inf_\oo p_{il}>0$;\vskip 3pt
		\item[\rm (b)] $U_i$ is semi-lower continuous for all $i\in{\cal S}_d$.
	\end{itemize}
\end{theo}

\begin{proof} By repeating the arguments to those in Theorem \ref{th3.7}, it is easy to verify that either $U_i\equiv 0$ in $\boo$ for all $i\in\sss$, or $U_i>0$ in $\boo$ for all $i\in{\cal S}$ and $\inf_{\oo}U_i>0$ for all $i\in{\cal S}_n$.
	
	In the case where (a) holds, due to
	\bess
	0\ge p_{ii}(x)U_i(x)+\sum_{k\not=i} p_{ik}(x)U_k\ge
	p_{ii}(x)U_i(x)+\inf_\oo p_{il}\inf_\oo U_l\eess
	for all $x\in\ol\oo$ and $i\in{\cal S}_d$, we have $\inf_{\oo}U_i>0$ for all $i\in{\cal S}_d$.
	
	In the case where (b) holds, assume that there exists $i\in{\cal S}_d$ such that $\inf_{\Omega} U_i=0$. Since $U_i$ is semi-lower continuous, then there exists $x_0$ such that $U_i(x_0)=0$, which is a contradiction.
\end{proof}\vspace{-1mm}

Based on the above strong maximum principle, Theorems \ref{th3.2}, \ref{th3.3} and \ref{th3.4} remain valid when the assumption {\bf (H2)} is replaced with the condition {\bf (H2$'$)} and additionally {\bf (H4$'$) holds.} Furthermore, Theorems \ref{thB}(1) and (2) hold when the assumption {\bf (H2)} is replaced with  the condition {\bf (H2$'$)}, while Theorem \ref{thB}(3) is valid when {\bf (H2)} is replaced with {\bf (H2$'$)} and additionally {\bf (H4$'$)} holds.

\subsection{More applications}
Of course, our method can be used to deal with the nonlocal dispersal version of the model studied in \cite{MWW18, Wang24}. In the following we show that the present method is also effective for more models.

\noindent{\bf May-Nowak model }

Let $u(x,t), v(x,t)$ and $w(x,t)$ be the densities of
healthy uninfected immune cells, infected immune cells, and virus particles. Based on the simple May-Nowak ODE model (\cite{Nowak-B96,Bo97}), the reaction diffusion May-Nowak model can be written as
 \bes\begin{cases}
  u_t=d_1\Delta u-a_1u-b uw+\varphi(x),\;\;&x\in\Omega,\;\;t>0,\\
 v_t=d_1\Delta v-a_1v+b uw,&x\in\Omega,\;\;t>0,\\
 w_t=d_2\Delta w-a_2w+\gamma v,&x\in\Omega,\;\;t>0,
 \end{cases}\lbl{5.1}\ees
where $\varphi(x)$ is a source term. The basic setting of this model is that the healthy uninfected cells are supplied at rate $\varphi$ and become infected on contact with virus at rate $buw$ (the infected cells are supplied at rate $buw$ at the same time), the virus particles are produced by the infected cells at rate $\gamma$, cells and virus particles diffuse with rates $d_1, d_2$, and die linearly with rates $a_1, a_1$.

The nonlocal dispersal version of \qq{5.1} with coefficients depending on $x$ becomes
\bes\begin{cases}
u_t=d_1\dd\int_\oo J_1(x,y)u(y,t)\dy-d_1^*(x)u-a_1(x)u-b(x) uw+\varphi(x),\;\;&x\in\boo,\;\;t>0,\\[3mm]
v_t=d_1\dd\int_\oo J_1(x,y)v(y,t)\dy-d_1^*(x)v-a_1(x)v+b(x)uw,&x\in\boo,\;\;t>0,\\[3mm]
w_t=d_2\dd\int_\oo J_2(x,y)w(y,t)\dy-d_2^*(x)w-a_2(x)w+\gamma(x) v,&x\in\boo,\;\;t>0,\\[1mm]
u=u_0(x)>0,\;\;v=v_0(x)>0,\;\;w=w_0(x)>0,&x\in\boo,
 \end{cases}\lbl{5.2}\ees
where $a_1, a_2, b, \vp$ and $\gamma$ are all continuous positive functions in $\boo$.

It is obvious that the problem
 \bess
 d_1\dd\int_\oo J_1(x,y)Z(y)\dy-d_1^*(x)Z-a_1(x)Z+\varphi(x)=0,\;\;\;x\in\boo
 \eess
has a unique positive solution $Z(x)$. Set
 \[B(x)=\left(\begin{array}{cc}
 -\big(d_1^*(x)+a_1(x)\big)\;&b(x)Z(x)\\
 \gamma(x)\;&\;-\big(d_2^*(x)+a_2(x)\big)\end{array}\right).\]
Similar to the discussions in Section \ref{S5} we have the following conclusions.
\vspace{-2mm}
 \begin{enumerate}[$(1)$]
 \item If $\lm({\mathscr B})>0$, then \qq{5.2} has a unique bounded positive equilibrium solution $(U(x), V(x), W(x))$ which is continuous in $\boo$, and the solution $(u,v,w)$ of \qq{5.2} satisfies
   \[\lim_{t\to+\yy}(u(x,t), v(x,t), w(x,t))=(U(x), V(x), W(x))\;\;\;\mbox{uniformly in}\;\;\boo.\]
 \item If $\lm({\mathscr B})< 0$, then \qq{5.2} has no positive equilibrium solution, and the solution $(u,v,w)$ of \qq{5.2} satisfies
  \[\lim_{t\to+\yy}(u(x,t), v(x,t), w(x,t))=(Z(x), 0, 0)\;\;\;\mbox{uniformly in}\;\;\boo.\]
 \end{enumerate} \vspace{2mm}

\noindent{\bf Capasso and Maddalena model}

Capasso and Maddalena \cite{CM81} considered the reaction-diffusion system
 \bes\begin{cases}
u_t=d_1\Delta u-\gamma_{11}u+\gamma_{12}v, \;\;&x\in\oo,\; t>0,\\
 v_t=d_2\Delta v-\gamma_{22}v+G(u), \;\;&x\in\oo,\; t>0,\\[1mm]
 \dd\frac{\partial u}{\partial\nu}+\alpha_1u=\dd\frac{\partial v}{\partial\nu}+\alpha_2v=0\;\;&x\in\partial\oo,\; t>0,\\
 u(x,0)=u_0(x)>0,\;\;v(x,0)=v_0(x)>0,\;\;&x\in\oo,
 \end{cases}\lbl{5.3}\ees
where $u(x,t)$ and $v(x,t)$ represent, respectively, the average concentration of the infective agents (bacteria or virus) and the infective human population at location $x\in\oo$ and time $t$, with $\oo$ a bounded domain in $\mathbb{R}^N$. Constants $\gamma_{11}, \gamma_{12}$ and $\gamma_{22}$ are positive, and the nonlinear function $G$ satisfies\vspace{-2mm}
  \begin{enumerate}
 \item[{\bf(G)}] $G\in C^1([0,+\yy))$, $G(0)=0$, $G'(z)>0$ and $\frac{G(z)}z$ is strictly decreasing for $z>0$, and $  \lim_{z\to+\yy}\frac{G(z)}z<\frac{\gamma_{11}\gamma_{22}}{\gamma_{12}}$.\vspace{-2mm}
   \end{enumerate}
They found two positive constants $R_M\ge R_m$ ($R_M=R_m$ if $\alpha_1=\alpha_2$) such that \qq{5.3} has a unique positive equilibrium solution which is globally asymptotically stable when $R_m>1$, and \qq{5.3} has no positive equilibrium solution and the trivial equilibrium solution $(0,0)$ is globally asymptotically stable when $R_M<1$.

The nonlocal dispersal version of \qq{5.3} with coefficients depending on $x$ can be written as
  \bess\begin{cases}
u_t=d_1\dd\int_\oo J_1(x,y)u(y,t)\dy-d_1^*(x)u-\gamma_{11}(x)u+\gamma_{12}(x)v, \;\;&x\in\boo,\; t>0,\\[3mm]
 v_t=d_2\dd\int_\oo J_2(x,y)v(y,t)\dy-d_2^*(x)v-\gamma_{22}(x)v+G(x,u), \;\;&x\in\boo,\; t>0,\\[1mm]
 u(x,0)=u_0(x)>0,\;\;v(x,0)=v_0(x)>0,\;\;&x\in\boo.
 \end{cases}\eess
Assume that coefficients $\gamma_{ik}(x)$ and nonlinear function $G(x,u)$ has the similar properties to the case that they don't depend on $x$, i.e.,\vspace{-2mm}
 \begin{enumerate}
 \item[{\bf(G$_x$)}] $\gamma_{11}, \gamma_{12}, \gamma_{22}\in C(\boo)$ and are positive, $G\in C^{0,1}(\boo\times[0,+\yy))$, $G(x,0)=0$, $G_z(x,z)>0$ and $\frac{G(x,z)}z$ is strictly decreasing for $z>0$ for all $x\in\boo$, and \[\lim_{z\to+\yy}\frac{G(x,z)}z<\frac{\dd\min_{\boo}\gamma_{11}(x)\min_{\boo}\gamma_{22}(x)}
     {\dd\max_{\boo}\gamma_{12}(x)}\;\;\;\mbox{ uniformly in }\;\;x\in\boo.\]
   \end{enumerate}\vspace{-2mm}
If $J_i(x,y)$ is symmetric, i.e., $J_i(x,y)=J_i(y,x)$, $i=1,2$, then we can find a large constant $C$ such that $(C, \frac{\min_{\boo}\gamma_{11}(x)}{\max_{\boo}\gamma_{12}(x)}C)$ is an upper solution of the equilibrium problem
 \bes\begin{cases}
d_1\dd\int_\oo J_1(x,y)u(y,t)\dy-d_1^*(x)u-\gamma_{11}(x)u+\gamma_{12}(x)v=0, \;\;&x\in\boo,\\[3mm]
 d_2\dd\int_\oo J_2(x,y)v(y,t)\dy-d_2^*(x)v-\gamma_{22}(x)v+G(x,u)=0, \;\;&x\in\boo.
 \end{cases}\lbl{5.4}\ees
Moreover, we can determine the generalized principal eigenvalue $\lm({\mathscr B})$ of the operator ${\mathscr B}$:
 \bess\begin{cases}
{\mathscr B}_1[(u,v)]=d_1\dd\int_\oo J_1(x,y)u(y,t)\dy-d_1^*(x)u-\gamma_{11}(x)u+\gamma_{12}(x)v, \\[3mm]
{\mathscr B}_2[(u,v)]=d_2\dd\int_\oo J_2(x,y)v(y,t)\dy-d_2^*(x)v-\gamma_{22}(x)v+G(x,0)u,
 \end{cases}\eess
such that the problem \qq{5.4} has a unique positive solution $(U(x), V(x))\in[C(\boo)]^2$ which is globally asymptotically stable if $\lm({\mathscr B})>0$, and \qq{5.4} has no positive solution and the trivial equilibrium solution $(0,0)$ is globally asymptotically stable when $\lm({\mathscr B})< 0$. The details are omitted here.

\noindent{\bf Benthic-drift model}

Based on  \cite{PLNL2005TPB,HJL2016SIADS}, we consider the population dynamics described by the following benthic-drift model in a river
\begin{equation}\label{equ:benthic-drift:local}
\begin{cases}
	u_t=du_{xx}-\alpha u_x-m_d(x)u-\sigma u+\frac{A_b(x)}{A_d(x)}\mu v, & x\in(0,L),~t>0,\\
	v_t=g(x,v)v-m_b(x)v+\frac{A_d(x)}{A_b(x)}\sigma u-\mu v, & x\in[0,L],~t>0,\\
	du_x(0,t)-\alpha u(0,t)=\alpha b_uu(0,t), & t>0,\\
	du_x(L,t)-\alpha u(L,t)=-\alpha b_du(L,t), & t>0,\\
	u(x,0)=u_{0}(x),~v(x,0)=v_0(x), & x\in(0,L).\\
\end{cases}
\end{equation}
Here $u(x,t)$ and $v(x,t)$ represent the population densities at location $x$ and time $t$ in the drift and benthic zone, respectively.
The diffusion rate and advection rate of the population in the drift zone are denoted by $d$ and $\alpha$, respectively. The drift and benthic population release rate are denoted by $\sigma$ and $\mu$, respectively.  The cross-sectional areas of the drift zone and the benthic zone are represented by $A_d(x)$ and $A_b(x)$, respectively. Additionally, $m_d(x)$ and $m_b(x)$ indicate the mortality rates of the drift and benthic populations at location $x$, respectively. The function $g(x,v)$ represents the per capita increase rate of the benthic population at location $x$. In this paper, we make the following assumptions:
\begin{itemize}\vspace{-2mm}
\item [\bf (A1)] $A_d$, $A_b$, $m_d$ and $m_b$ are positive, continuously differentiable functions on $[0,L]$.
\item [\bf (A2)] $g(x,v)$ is continuously differentiable with respect to $x\in[0,L]$ and $v\in[0,+\infty)$. Moreover, $g(x,v)$ is strictly decreasing with respect to $v$. For each $x$, there exists a unique $K(x)>0$ such that $g(x,K(x))-m_b(x)\leq0$.
\end{itemize}\vspace{-2mm}

The nonlocal version of \eqref{equ:benthic-drift:local} without advection can be written as
\begin{equation}\label{equ:benthic-drift:nonlocal}
	\begin{cases}
		u_t=d\dd\int_\oo J_1(x,y)u(y,t)\dy-d_1^*(x)u-m_d(x)u-\sigma u+\frac{A_b(x)}{A_d(x)}\mu v, & x\in(0,L),~t>0,\\
		v_t=g(x,v)v-m_b(x)v+\frac{A_d(x)}{A_b(x)}\sigma u-\mu v, & x\in[0,L],~t>0,\\
		u(x,0)=u_{0}(x),~v(x,0)=v_0(x), & x\in(0,L).\\
	\end{cases}
\end{equation}
If $J_1(x,y)$ is symmetric, i.e., $J_1(x,y)=J_1(y,x)$, then by repeating the argument to those in \cite[Lemma 3.2]{LZZ2017SIMA} the following equilibrium problem
\begin{equation}\label{equ:benthic-drift:nonlocal:E}
	\begin{cases}
		d\dd\int_\oo J_1(x,y)u(y)\dy-d_1^*(x)u-m_d(x)u-\sigma u+\frac{A_b(x)}{A_d(x)}\mu v=0, & x\in(0,L),\\
		g(x,v)v-m_b(x)v+\frac{A_d(x)}{A_b(x)}\sigma u-\mu v=0, & x\in[0,L],
	\end{cases}
\end{equation}
 has a strictly positive upper solution. We can determine the generalized principal eigenvalue $\lm({\mathscr B})$ of the operator ${\mathscr B}=({\mathscr B}_1, {\mathscr B}_2)$:
\bess\begin{cases}
	{\mathscr B}_1[(u,v)]=d_1\dd\int_\oo J_1(x,y)u(y,t)\dy-d_1^*(x)u-m_d(x)u-\sigma u+\frac{A_b(x)}{A_d(x)}\mu v, \\[3mm]
	{\mathscr B}_2[(u,v)]=g(x,0)v-m_b(x)v+\frac{A_d(x)}{A_b(x)}\sigma u-\mu v.
\end{cases}\eess
Therefore, we have the following results:
\begin{itemize}
	\item[(1)] If $\lm({\mathscr B})>0$, then \eqref{equ:benthic-drift:nonlocal:E} has a globally asymptotically stable positive solution.
	\item[(2)] $\lm({\mathscr B})<0$, then \eqref{equ:benthic-drift:nonlocal:E} has no positive solution, and zero solution is globally asymptotically stable.
\end{itemize}

We also remark that there is no standard method to characterize nonlocal advection, so we omit the advection in our nonlocal version.

\appendix\def\theequation{\Alph{section}.\arabic{equation}}
\section{Appendix}

In this appendix we shall prove Lemma \ref{le2.1}. To this end, we first prove a lemma.

\begin{lem}\label{lem:LJG}Assume that the condition {\bf(B)} holds, and define operators ${\cal B}$ and $\cal J$ by
 \[[{\cal B}](x)=B(x)\]
and
 \[{\cal J}[u]=({\cal J}_1[u],\cdots,{\cal J}_n[u]),\;\;\;
	{\cal J}_i[u]=d_i\dd\int_\oo J_i(x,y)u_i(y)\dy,\]
respectively. Then the following statements are valid:\vspace{-2mm}
 \begin{itemize}
\item[\rm (i)] $\sigma_e({\mathscr B})=\cup_{x\in\boo}\sigma(B(x))$, where $\sigma_e({\mathscr B})$ is the essential spectrum of ${\mathscr B}$ {\rm(}see, e.g., {\rm\cite{S1971book})}.\vvv
\item[\rm (ii)] The operator ${\mathscr B}$ has a principal eigenvalue if and only if $s({\mathscr B})>\max_{x \in \boo} s(B(x))$, where $s({\mathscr B})$ is the spectral bound of ${\mathscr B}$.\vvv
\item[\rm (iii)] If $r({\cal J}(\lambda_0 I-{\cal B})^{-1})>1$ for some $\lambda_0>\max_{x\in\boo}s(B(x))$, then $s({\mathscr B})$ is the principal eigenvalue of ${\mathscr B}$.\vspace{-1mm}
\end{itemize}
	\end{lem}

\begin{proof} This Lemma essentially can be obtained from \cite{BS17}, \cite{Zhang24} and \cite{LCWdcds17}. For reader's convenience, we provide the outline of the proof.
		
(i)  Thanks to \cite[Theorem 7.27]{S1971book} and \cite[Proposition 2.7]{LZZ2019JDE}, $\sigma_e({\mathscr B})=\sigma({\mathscr B})=\cup_{x \in \boo} \sigma(B(x))$.
		
(ii) In view of {\bf(B)}, there exist $c_0$ large enough and some $m_0$ such that $B(x) +c_0 I$ is positive and $(B(\tilde{x})+c_0 I)^m$ is strongly positive for $m\geq m_0$ (see, e.g., \cite[(8.3.5)]{M2000Book} and \cite[Lemma 2.7]{Zhang24}). Noticing that ${\cal J}+c_0 I+{\cal B}$ is positive, and
 \[({\cal J}+c_0 I+{\cal B})^m\varphi\geq{\cal J}(c_0 I+{\cal B})^{m-1}\varphi, \;\;\forall\; \varphi \geq 0,\]
and ${\cal J} (c_0 I+{\cal B})^{m-1}$ is strongly positive. We have that $({\cal J}+c_0 I+{\cal B})^m$ is strongly positive for all $m\geq m_0+1$.
		
``$\Longleftarrow$'' If $s({\mathscr B})>\max_{x \in \boo} s(B(x))=\eta$, then the operator ${\mathscr B}$ has a principal eigenvalue due to a variation of the generalized Krein-Rutman theorem (see, e.g., \cite[Corollary 2.2]{N1981FPT} and \cite[Lemma 2.4]{Zhang24}).
		
``$\Longrightarrow$'' If $s({\mathscr B})$ is the principal eigenvalue of ${\mathscr B}$  corresponding to the principal eigenfunction $\phi$. Choose $\bar{x} \in \overline{\Omega}_0$ such that $s(B(\bar{x}))=\eta$. Thus,
 \[d_i\dd\int_\oo J_i(\bar{x},y)\phi_i(y)\dy+\sum_{k=1}^n b_{ik}(\bar{x})\phi_k(\bar{x})= s({\mathscr B})\phi_i(\bar{x}),\;\; i\in\sss.\]
Let $\psi$ be left positive eigenvector of $B(\bar{x})$ corresponding to $s(B(\bar{x}))$. Multiplying the above equation by $\psi_i$ and summing them together, we have
 \[\sum_{i=1}^{n}\psi_i d_i\dd\int_\oo J_i(\bar{x},y)\phi_i(y)\dy
 +\sum_{i=1}^{n}\sum_{k=1}^n \psi_i b_{ik}(\bar{x})\phi_k(\bar{x})= s({\mathscr B})\sum_{i=1}^{n}\psi_i \phi_i(\bar{x}),\]
and hence,
  \[s(B(\bar{x}))\sum_{i=1}^{n}\psi_i\phi_i(\bar{x})< s({\mathscr B})\sum_{i=1}^{n}\psi_i\phi_i(\bar{x}).\]
This implies that $s({\mathscr B})>s(B(\bar{x}))=\eta$.
		
(iii) Notice that $r({\cal J}(\lambda I-{\cal B})^{-1})$ is non-increasing and continuous with respect to $\lambda\in(\eta, +\infty)$. We claim that
  \bes
  \lim\limits_{\lambda\to+\infty}r({\cal J}(\lambda I-{\cal B})^{-1})<1.
   \lbl{A.1}\ees
Under this claim, we can find a $\lambda_1>\eta$ such that $r({\cal J}(\lambda_1 I-{\cal B})^{-1})=1$. By the Krein-Rutman theorem (see, e.g., \cite[Theorem 19.3]{Dei85}), there exists positive $\psi$ such that ${\cal J}(\lambda_1 I-{\cal B})^{-1}\psi=\psi$. Letting $(\lambda_1 I-{\cal B})^{-1}\psi=\phi$ we then have ${\cal J}\phi+{\cal B}\phi=\lambda_1 \phi$ and $s({\mathscr B})\geq\lambda_1>\eta$. Thanks to (ii), $ s({\mathscr B})$ is the principal eigenvalue of ${\mathscr B}$.
		
We finally prove \qq{A.1}. Suppose that $\lim\limits_{\lambda\to+\infty}r({\cal J}(\lambda I-{\cal B})^{-1})\geq 1$. Write $\gamma_0=r({\cal J}(\lambda_0 I-{\cal B})^{-1})>1$, and set
 \[\bar{b}=\max_{i,\,k\in\sss}\,\max_{x\in\boo}|b_{ik}(x)|,\quad \bar{d}=\max_{i\in\sss} d_i,\quad\bar{J}
   =\max_{i\in\sss}\max_{(x,y)\in\boo\times\boo} J_i(x,y).\]
Choose $\lambda_2 >\max\big\{\lambda_0,\,n\gamma_0\bar{b}+\bar{J}\bar{d}|\oo|\big\}$, and write $\gamma_2=r({\cal J}(\lambda_2 I-{\cal B})^{-1})$. Clearly, $\gamma_0 \geq \gamma_2\geq 1$. By the Krein-Rutman theorem, there exists positive $\vp$ such that ${\cal J}(\lambda_2 I-{\cal B})^{-1}\vp=\gamma_2 \vp$. Let $(\lambda_2 I-{\cal B})^{-1}\vp=\phi$. We have ${\cal J}\phi=\gamma_2(\lambda_2 I-{\cal B})\phi$, and hence,
 \[d_i\dd\int_\oo J_i(x,y)\phi_i(y)\dy+\gamma_2\sum_{k=1}^n b_{ik}(x)\phi_k(x)=\gamma_2\lambda_2 \phi_i(x),\;\; x \in \boo \;\; i\in\sss.\]
Integrating the above equation and summing them together, we have
 \bess
 \lambda_2\sum_{i=1}^n\int_\oo \phi_i(x)&\leq& \gamma_2\lambda_2\sum_{i=1}^n   \int_\oo \phi_i(x)\dx\\
&=&\sum_{i=1}^nd_i\int_\oo \int_\oo J_i(x,y)\phi_i(y)\dy\dx+ \gamma_2\sum_{i=1}^n\sum_{k=1}^n\int_\oo b_{ik}(x)\phi_k(x)\dx\\
 &\leq&\bar{J}\bar{d}|\Omega|\sum_{i=1}^n\int_\oo \phi_i(x)\dx
+n\gamma_0\bar{b}\sum_{i=1}^n\int_\oo\phi_i(x)\dx,
 \eess
which is a contradiction. Thus, \qq{A.1} holds and the proof is complete.
	\end{proof}\vspace{-2mm}

We remark that if \( B(\bar{x}) \) is irreducible, (ii) has been proved in \cite{BS17}, and our proof holds regardless of whether \(B(\bar{x})\) is irreducible or not; (iii) is a generalization of \cite[Theorem 2.2]{B1988MZ}, where the irreducibility condition is weakened to a point.\vspace{-2mm}

\begin{proof}[Proof of Lemma \ref{le2.1}] Choose $\bar{x} \in \overline{\Omega}_0$ such that $s(B(\bar{x}))=\eta:=\max_{x\in\ol\oo} s(B(x))$.

In the case where \(B(\bar{x})\) is irreducible, the desired conclusion can be derived by modifying the arguments to those in \cite[Proposition 3.4 and Lemma 4.1]{BS17}. For reader's convenience, we provide the details.
		
By the continuity of $J_i$, there exist $r_0>0$ and $c_0>0$ such that
  \[J_i(x,y)>c_0\;\;\;\mbox{for\, all}\;\;x,y\in\boo\;\;\mbox{with}\;\;|x-y|<r_0\;\;\mbox{and}\;\;i\in\sss.\]
According to the Perron-Frobenius theorem, the {\rrr matrix-valued} function $B(x)$ admits an eigenvalue $s(B(x))$ corresponding to an eigenfunction $w(x)$ which is non-negative and continuous in $\boo$, and $\max_{i\in\sss}\max_{\boo}w_i=1$. Since \( B(\bar{x})\) is irreducible, $w(\bar{x})$ is strongly positive. By the continuity, there exists $\sigma>0$ such that $B(\bar{x},\sigma)\subset\Omega$ and $w$ is strongly positive in $\overline{B(\bar{x},\sigma)}$. Let
 \[c_1= \min_{i\in\sss}\min_{\overline{B(\bar{x},\sigma)}}w_i.\]
In view of $\frac 1{\eta-s(B(x))}\not\in L^1(\Omega_0)$, we can choose $0<\delta<\min\{\sigma,\,r_0/3\}$ and $\beta_0>\eta$ such that
 \[B(\bar{x},2\delta)\subset\Omega,\;\;\;\int_{B(\bar{x},\,\delta)}
 \frac 1{\beta_0-s(B(x))}\dx \geq \frac 2{c_0c_1}.\]
Let $p(x)$ be a continuous function in $\boo$ with $\max_{\boo}p=1$ and
 \[p(x)=\begin{cases}
	1, &x \in B(\bar{x},\delta),\\
	0, &x \not \in B(\bar{x},2\delta).\end{cases}\]
Write $\hat{w}(x)=w(x)p(x)$, $x \in \boo$. It then follows that
  \bess
  \int_{\Omega}\frac{J_{i}(x,y)}{\beta_0-s(B(y))}\hat{w}_i(y)\dy&\geq& 0=2\hat{w}(x), \;\; \forall\, x \in \boo \setminus B(\bar{x},2\delta), \; i\in\sss,\\[1mm]
\int_{\Omega} \frac{J_{i}(x,y)}{\beta_0-s(B(y))}\hat{w}_i(y)\dy
&\geq& \int_{B(\bar{x},\delta)}\frac{J_{i}(x,y)}{\beta_0-s(B(y))}\hat{w}_i(y)\dy\\
&=&\int_{B(\bar{x},\delta)}\frac{J_{i}(x,y)}{\beta_0-s(B(y))}w_i(y)\dy\\[1mm]
&\geq& 2\hat{w}_i(x),\;\; \forall\, x\in B(\bar{x},2\delta), \; i\in\sss.
 \eess
This implies ${\cal J}(\beta_0 I-{\cal B})^{-1}\hat{w}\geq 2\hat{w}$, and the desired conclusion can be derived by Lemma \ref{lem:LJG}.
		
In the case where $B(\bar{x})$ is reducible, there exists a block $\hat{B}(\bar{x})$ of $B(\bar{x})$ such that $\hat{B}(\bar{x})$ is irreducible and $s(\hat{B}(\bar{x}))=s(B(\bar{x}))=\eta$. More precisely, there is an index set $\Sigma\subset\sss$ such that $\hat{B}(x)=(b_{ik}(x))_{i,k\in\Sigma}$, and $\max_{\ol\oo} s(\hat{B}(x))=\eta$. So the following eigenvalue problem
	\[d_i\dd\int_\oo J_i(x,y)\phi_i(y)\dy+\sum_{k\in\Sigma}b_{ik}(x)\phi_k(x)=\lm \phi_i,\; x\in\oo,\;\; i \in\Sigma\]
	has the principal eigenvalue $\hat{\lambda}$. It is not hard to verify that $s({\mathscr B})\geq\hat{\lambda}>\eta$. According to Lemma \ref{lem:LJG}, $s({\mathscr B})$ is the principal eigenvalue of ${\mathscr B}$.
	\end{proof}

\vspace{-1mm}


\begin{thebibliography}{10}\setlength{\itemsep}
{0mm}\linespread{1.2}\selectfont

\bibitem{BS17}
X. Bao and W. Shen, {\it Criteria for the existence of principal
	eigenvalues of time periodic cooperative linear systems with nonlocal
	dispersal}, Proc. Amer. Math. Soc., 145 (2017),  2881-2894.
	
\bibitem{BZ07} P. Bates and G. Zhao, {\it Existence, uniqueness and stability of the stationary solution to a nonlocal evolution equation arising in population dispersal}. J. Math. Anal. Appl. 332 (2007), 428-440.

\bibitem{B-JFA16}  H. Berestycki, J. Coville and H.H. Vo,
{\it On the definition and the properties of the principal eigenvalue of some nonlocal operators}. J. Funct. Anal. 271 (2016), 2701-2751.

\bibitem{Bo97}S. Bonhoeffer, R. M. May, G. M. Shaw and M. A. Nowak, {\it Virus dynamics and drug therapy}. Proc. Natl. Acad. Sci. USA 94 (1997), 6971-6976.

\bibitem{B1988MZ} R.~B{\"u}rger, {\it Perturbations of positive semigroups and applications to population genetics}, Math. Z., 197 (1988),  259-272.

 \bibitem{CM81}V. Capasso and L. Maddalena, {\it Convergence to equilibrium states for a reaction-diffusion system modelling the spatial spread of a class of bacterial and viral diseases}. J. Math. Biol. 13 (1981/82), 173-184.

\bibitem{Cov} J. Coville, {\it On a simple criterion for the existence of a principal eigenfunction of some nonlocal operators}.
J. Differential Equations 249 (2010), 2921-2953.

\bibitem{Dei85} K. Deimling, {\it Nonlinear Functional Analysis}, Springer-Verlag, Berlin, Heidelberg, 1985.

\bibitem{EPS72} D. Edmunds, A. Potter, and C.Stuart, {\it Non-compact positive
operators}, Proc. R. Soc. Lond. Ser. A Math. Phys. Eng. Sci., 328 (1972), pp.~67--81.

\bibitem{FLRX24} Y.-X. Feng, W.-T. Li, S.~Ruan, and M.-Z. Xin, {\it Principal spectral theory of time-periodic nonlocal dispersal cooperative systems and
	applications}, SIAM J. Math. Anal., 56 (2024), 4040-4083.

\bibitem{HJL2016SIADS}{\sc Q. Huang, Y. Jin, and M. A. Lewis}, {\it {$R_0$} analysis of a benthic-drift model for a stream population}, SIAM J. Appl. Dyn. Syst., 15
(2016),  278-321.
	
\bibitem{LCWdcds17} F. Li, J. Coville and X.F. Wang, {\it On eigenvalue problems arising from nonlocal diffusion models}. Discrete Contin. Dyn. Syst. 37(2) (2017), 879-903.

\bibitem{LZZ2017SIMA} X. Liang, L. Zhang and X.-Q. Zhao,
{\it The principal eigenvalue for degenerate periodic reaction-diffusion systems,} SIAM J. Math. Anal., 49 (2017), 3603-3636.

\bibitem{LZZ2019JDE}
 X.~Liang, L.~Zhang, and X.-Q. Zhao, {\it The principal eigenvalue for
	periodic nonlocal dispersal systems with time delay}, J. Differential
Equations, 266 (2019),  2100-2124.

\bibitem{KangR22}
 H.~Kang and S.~Ruan, {\it Principal spectral theory and asynchronous
	exponential growth for age-structured models with nonlocal diffusion of
	{N}eumann type}, Math. Ann., 384 (2022),  575-623.
	
\bibitem{MWW18}P. Magal, G.F. Webb and Y.X. Wu, {\it On a vector-host epidemic model with spatial structure}. Nonlinearity 31 (2018), 5589-5614.

\bibitem{M2000Book}
C.~D. Meyer, {\it Matrix analysis and applied linear algebra}, vol.~71,
SIAM, Philadelphia, 2000.

\bibitem{Nowak-B96} M. A. Nowak and C. R. M. Bangham, {\it Population dynamics of immune responses to persistent viruses}. Science 272 (1996), 74-79.

\bibitem{N1981FPT}
R. Nussbaum, {\it Eigenvectors of nonlinear positive operators and the
	linear {Krein-Rutman} theorem}, Fixed Point Theory, 886 (1981), 309-330.
	
\bibitem{PLNL2005TPB} E. Pachepsky, F. Lutscher, R. M. Nisbet and M. A. Lewis,
{\it Persistence, spread and the drift paradox},  Theoretical  Population Biology, 67 (2005) 61-73.

\bibitem{S1971book}
 M.~Schechter, {\it Principles of Functional Analysis}, vol.~2, Academic,
New York, 1971.


\bibitem{SX15} W.X. Shen and X.X. Xie, {\it On principal spectrum points/principal eigenvalues of nonlocal dispersal
operators and applications}, Discrete Contin. Dyn. Syst., 35(4) (2015), 1665-1696.

\bibitem{SZ10JDE}
 W.X. Shen and A.J. Zhang, {\it Spreading speeds for monostable equations with
	nonlocal dispersal in space periodic habitats}, J. Differential Equations, 249 (2010), 747-795.

\bibitem{SZ2012PAMS} W.~Shen and A.~Zhang, {\it Stationary solutions and spreading speeds of nonlocal monostable equations in space periodic habitats}, Proc. Amer. Math. Soc., 140 (2012), 1681-1696.


\bibitem{SWZ23} Y.-H. Su, X.F. Wang and T. Zhang, {\it Principal spectral theory and variational characterizations for cooperative systems with nonlocal and coupled diffusion}. J. Differential Equations 369 (2023), 94-114.

\bibitem{SLLW2023JMPA} Y.-H. Su, W.-T. Li, Y.~Lou, and X.~Wang, {\it Principal spectral theory for nonlocal systems and applications to stem cell regeneration models}, J.
Math. Pures Appl. (9), 176 (2023), 226-281.

\bibitem{Thieme1998DCDS} H.R. Thieme, {\it Remarks on resolvent positive operators and their perturbation}, Discrete Contin. Dynam. Systems, 4 (1998), 73-90.

\bibitem{Wang24}M.X. Wang, {\it Note on a vector-host epidemic model with spatial structure}. 2024. arXiv:2406.11407.

\bibitem{WZhang24} M.X. Wang and L. Zhang, {\it On a reaction-diffusion virus model with general boundary conditions in heterogeneous environments}. 2024. arXiv:

\bibitem{YLR2019JDE} F.-Y. Yang, W.-T. Li, and S.~Ruan, {\it Dynamics of a nonlocal dispersal 	sis epidemic model with neumann boundary conditions}, J. Differential
Equations, 267 (2019),  2011-2051.

\bibitem{Zhang24} L. Zhang, {\it Principal spectral theory and asymptotic
behavior of the spectral bound for partially degenerate nonlocal dispersal systems}, Adv. Nonlinear Stud., 24 (4)2024, 1012-1041.

\end{thebibliography}
\end{document}